\documentclass{amsart}
\usepackage{amssymb,amsfonts,amscd}
\usepackage{amsmath}
\usepackage[arrow, matrix, curve]{xy}
\usepackage{eucal}
\usepackage{dsfont} 
\usepackage{hyperref}
\usepackage{pstricks}
\usepackage{pst-3dplot}

\DeclareFontFamily{U}{matha}{\hyphenchar\font45}
\DeclareFontShape{U}{matha}{m}{n}{
      <5> <6> <7> <8> <9> <10> gen * matha
      <10.95> matha10 <12> <14.4> <17.28> <20.74> <24.88> matha12
      }{}
\DeclareSymbolFont{matha}{U}{matha}{m}{n}
\DeclareFontSubstitution{U}{matha}{m}{n}
\DeclareMathSymbol{\operp}         {2}{matha}{"6B}

\newcommand{\bp}{{\bar{P}}}
\newcommand{\D}{{\mathcal{D}}}
\newcommand{\V}{{\mathcal{V}}}
\newcommand{\T}{{\mathcal{T}}}
\newcommand{\C}{{\mathcal{C}}}
\newcommand{\K}{{\mathcal{K}}}
\newcommand{\R}{\mathds{R}}
\newcommand{\sfe}{\mathds{S}^1}
\newcommand{\N}{\mathds{N}}
\newcommand{\mx}{\mathfrak{X}}
\newcommand{\lie}[1]{\mathfrak{#1}}
\newcommand{\dr}{\mathbf{d}}

\newcommand{\ldr}[1]{{\pounds}_{#1}}
\newcommand{\ip}[1]{{\mathbf{i}}_{#1}}
\newcommand{\an}[1]{\arrowvert_{#1}}
\newcommand{\pair}[1]{\langle {#1}\rangle}
\newcommand{\pairing}{\langle\cdot\,,\cdot\rangle}

\DeclareMathOperator{\erz}{span}
\DeclareMathOperator{\dom}{Dom}
\DeclareMathOperator{\Id}{Id}
\DeclareMathOperator{\ann}{ann}
\DeclareMathOperator{\SO}{SO(3)} 
\newcommand{\poi}[1]{\{#1\}}
\newcommand{\pois}{\poi{\cdot\,,\cdot}}

\newtheorem{theorem}{Theorem}[section]
\newtheorem{lemma}[theorem]{Lemma}
\newtheorem{proposition}[theorem]{Proposition}
\newtheorem{corollary}[theorem]{Corollary}

\theoremstyle{definition}
\newtheorem{definition}[theorem]{Definition}
\newtheorem{example}[theorem]{Example}

\theoremstyle{remark}
\newtheorem{remark}[theorem]{Remark}

\numberwithin{equation}{section}
\begin{document}
\allowdisplaybreaks[1]
\date{}
\title{Singular reduction of Dirac structures.\footnote{\emph{\tiny Appeared in Transactions
of the American Mathematical Society (2011), Volume 363, 2967-3013.
Corrected version, 17/10/2011: The definition of ``locally finitely generated distributions'' and the (wrong)
statement about their integrability  were not needed and have 
been erased.}}}

\author{M. Jotz}
\address{Section de Math{\'e}matiques\\ 
Ecole Polytechnique
  F{\'e}d{\'e}rale de Lausanne\\ 
1015 Lausanne\\ Switzerland\\}
\curraddr{}
\email{madeleine.jotz@a3.epfl.ch}

\author{T.S. Ratiu}
\address{Section de Math{\'e}matiques\\
 et Centre Bernouilli\\ 
Ecole Polytechnique
  F{\'e}d{\'e}rale de Lausanne\\ 
1015 Lausanne\\ Switzerland\\}
\curraddr{}
\email{tudor.ratiu@epfl.ch}
\thanks{The first two authors were partially supported by
   Swiss NSF grant 200021-121512.}

\author{J. \'Sniatycki}
\address{Department of Mathematics and Statistics\\
  University of Calgary\\ Calgary, Alberta\\ Canada \\}
\curraddr{}
\email{sniat@math.ucalgary.ca}
\thanks{The third author was supported by an NSERC Discovery Grant.}

\maketitle

\begin{abstract}
The regular reduction of a Dirac manifold acted upon freely and properly by a
Lie group is generalized to a nonfree action. For this, several facts about
$G$-invariant vector fields and one-forms are shown.


\hspace{3cm} Secondary subjects: 70G45, 53D17, 53D99

\end{abstract}

\tableofcontents

\section{Introduction}

Dirac structures generalize  Poisson and symplectic manifolds.
They also provide a
convenient geometric setting for the theory of nonholonomic systems.
This concept was introduced in \cite{CoWe88} and \cite{Courant90a} and has seen a
significant development in the recent past both from the geometric
point of view as well as in applications to mechanical systems and
circuit theory. In the presence of symmetry, one can perform a reduction
to eliminate variables. The modern global formulation of reduction of
Hamiltonian systems with symmetry is due to Marsden and Weinstein  
\cite{MaWe74} who treat
free and proper symplectic actions admitting an equivariant momentum
map. This was generalized to Poisson manifolds  in \cite{MaRa86}. When
dealing with implicit Hamiltonian systems, which can be seen as sets
of algebraic and differential equations, the geometric description is
based on Dirac structures. Hence it is natural to ask if a symmetric
Dirac manifold can be reduced. This was carried out for a free and
proper Dirac action in \cite{BlvdS01} and \cite{Blankenstein00} within
the context of generalized Poisson structures and can be derived 
as an easy case  of the results in the paper
\cite{BuCaGu07} about reduction of Courant algebroids.
The methods of \cite{BlvdS01} and \cite{BuCaGu07} in the particular case of 
interest to us are equivalent up to a small difference in assumptions
(see \cite{JoRa11c}).
It is shown in \cite{JoRaZa11} that the  assumptions 
in \cite{BlvdS01}  can be weakened to the hypotheses 
of \cite{BuCaGu07} in the case of a free and proper action on the underlying
manifold $M$.
Singular Dirac reduction was treated in \cite{BlRa04}
using the following setup: the symmetric Dirac structure is viewed as
a generalized Poisson structure with a momentum map and a reduction of
implicit Hamiltonian systems is performed at all values
of the momentum map, including singular ones. It turns out that each  
stratum of the reduced
space (which is a Whitney stratified cone space since the action
is proper) inherits a Dirac structure. In addition, Hamiltonian  
dynamics on the  original manifold descend to each stratum of the  
quotient.

In this paper, we study the reduction of a smooth Dirac manifold
$(M,D)$ by a proper Dirac action of a Lie group $G$ completely within
the Dirac category: that is, certain nontrivial technical hypotheses
on various distributions present in \cite{BlvdS01},
\cite{Blankenstein00} and  \cite{BuCaGu07} 
(in the case of interest to us) are eliminated. 
This is
achieved by working directly with smooth structures on stratified
spaces.  This approach, known as singular reduction, was initiated in  
\cite{Cushman83} and formalized in \cite{ArCuGo91}. In \cite{CuSn01},  
singular reduction was shown to be an application of the theory of  
differential spaces.

The concepts of vector fields and one-forms on Whitney
stratified spaces are reviewed and applied to the quotient space of
the manifold by the action. We show in Theorem \ref{singred2} that the
descending sections of the
Dirac structure push forward to a subset of the pairs formed by local
vector fields and one-forms on the reduced space that is, in a sense,
self-orthogonal. This leads to the following natural question: do the
strata of the reduced space inherit Dirac structures induced by $D$?
We show in Theorem \ref{singred} that this is true if one assumes that
the set of descending sections generates a certain subdistribution of
the Pontryagin bundle of $M$. To achieve this, we employ several new
techniques. Using the existence of $G$-invariant Riemannian metrics
for proper actions on  paracompact manifolds and the tube theorem, we
introduce averages of vector fields and one-forms on $G$-invariant
open subsets of $M$. In a crucial step of the proof, we use the fact
that, in certain situations, $G$-invariant averages of vector fields
and one-forms vanish. Also, we study the relationship between the  
pointwise and the smooth  orthogonal distribution of a smooth  
generalized subdistribution of a vector bundle endowed
with a symmetric nondegenerate pairing. This allows us to
describe the smooth annihilator of an intersection of
smooth generalized distributions in certain cases of interest to us.

The paper is organized as follows. Dirac structures are reviewed in
Section \ref{sec:Dirac_structures} and vector fields on differential
spaces, stratifications, and orbit type manifolds in Section
\ref{sec:differential_spaces}. Generalized distributions and the
integrability of tangent distributions  as well as pointwise and
smooth annihilators are introduced and discussed in Section \ref{dis}.
The averaging method is presented in Section \ref{sec:proper}. Using
this technique, we show that the strata of the quotient $\bar{M}$  
correspond to  the quotients of the orbit type strata on $M$ and that  
the local
one-forms  on the manifold $M$ descend to analogous objects on the
reduced stratified space $\bar{M}$. Then we study the properties of
descending sections of the Pontryagin bundle and get many technical
results needed in the final reduction proof. Section 6 is devoted to
the main result of the paper, namely singular Dirac reduction. First we
recall  the reduction procedure in the case of conjugated
isotropy subgroups. Then the two main theorems of the paper (Theorems
\ref{singred2} and \ref{singred}) are proved and the reduced dynamics of implicit
Hamiltonian systems is constructed. Several examples are also given.
\medskip

\noindent\emph{Conventions, definitions, and notations}
In this paper we are working in the\emph{\ smooth category}. All sets
considered here are smooth \emph{subcartesian spaces}; see Section 2. In
particular, all manifolds and maps are assumed to be smooth. Moreover, the
manifold $M$ is \emph{paracompact} and the Lie group $G$ acting on it  
is \emph{connected}. If not mentioned in the text, the action of $G$  
on $M$ is
always assumed to be \emph{proper}.

We will write $C^\infty(M)$ for the sheaf of local functions
$C^\infty_{\rm  loc}(M)$ on $M$. That is,
an element $f \in C^\infty(M)$ is a smooth function $f:U\to \R$, with $U$ an
open subset of $M$.
In the same manner, if $E$ is a vector bundle over $M$, or a generalized
distribution on $M$, we will denote by
$\Gamma(E)$ the set of local sections of $E$. In particular, the sets of local
vector fields and one-forms on $M$ will be denoted by $\mx(M)$ and
$\Omega^1(M)$, respectively. We will write $\dom(\sigma)$ for the open  
domain of definition  of the section $\sigma$ of $E$.

A section $X$ (respectively $\alpha$) of $TM$ (respectively $T^*M$) will be
called \emph{$G$-invariant} if $\Phi_g^*X=X$ (respectively
$\Phi_g^*\alpha=\alpha$) for all $g\in G$, where $\Phi: G\times M\to  
M$ is the action of $G$ on $M$. Here, the vector field $\Phi_g^\ast X$  
is defined by $\Phi_g^\ast X=T\Phi_{g^{-1}}\circ X\circ\Phi_g$, that is,
$(\Phi_g^\ast X)(m)=T_{gm}\Phi_{g^{-1}}X(gm)$ for all $m\in M$.

Recall that a subset $ N \subset  M $ is an \emph{initial} submanifold
of $M$ if  $N$ carries a manifold structure such that the inclusion
$\iota :N \hookrightarrow M$ is a smooth immersion and satisfies
the following condition: for any smooth manifold $P$, an arbitrary
map $g: P \rightarrow  N$ is smooth if and only if $\iota \circ  g: P
\rightarrow  M$ is smooth. The notion of initial submanifold lies strictly
between those of injectively immersed and embedded submanifolds.

In the following, we write $TM\operp TM^*$ for the sum of the vector  
bundles  $TM$ and $T^*M$ and use the same notation for the sum of a  
tangent
(that is, a subdistribution of $TM$) and cotangent distribution (a  
subdistribution of
$T^*M$; see Section \ref{dis} for the definitions of those objects). We choose
this notation because we want to distinguish these direct sums from  
direct sums
of subdistributions of a bundle, which will be written as usual with $\oplus$.

\section{Generalities on Dirac structures}\label{sec:Dirac_structures}

\noindent\emph{Dirac structures} \label{admsble}
The \emph{Pontryagin bundle} $TM \operp T^* M$ of  a smooth manifold $M$ is
endowed with a nondegenerate symmetric fiberwise bilinear form of signature
$(\dim M, \dim M)$ given by
\begin{equation}\label{pairing}
\left\langle (u_m, \alpha_m), ( v_m, \beta_m ) \right\rangle : 
= \left\langle\beta_m , u_ m \right\rangle + \left\langle\alpha_m, v _m \right\rangle
\end{equation}
for all $u _m, v _m \in T _mM$ and $\alpha_m, \beta_m \in T^\ast_mM$. A
\emph{Dirac structure} 
(see \cite{Courant90a}) on $M $ is a Lagrangian subbundle $D \subset TM \operp
T^* M $. That is, 
$ D$ coincides with its orthogonal relative to \eqref{pairing} and so its
fibers are necessarily 
$\dim M $-dimensional.

The space $\Gamma(TM \operp T ^\ast M) $ of local sections of the Pontryagin
bundle is also endowed with an
$\mathds{R}$-bilinear skew-symmetric bracket (which does not 
satisfy the Jacobi identity) given by
\begin{align}\label{wrong_bracket}
[(X, \alpha), (Y, \beta) ] : &
= \left( [X, Y],  \ldr{X} \beta - \ldr{Y} \alpha + \frac{1}{2}
  \mathbf{d}\left(\alpha(Y) 
- \beta(X) \right) \right) \nonumber \\
&= \left([X, Y],  \ldr{X} \beta - \mathbf{i}_Y \mathbf{d}\alpha 
- \frac{1}{2} \mathbf{d} \left\langle (X, \alpha), (Y, \beta) \right\rangle
\right)
\end{align}
(see \cite{Courant90a}). The Dirac structure is \emph{integrable} or\emph{closed} if $[ \Gamma(D), \Gamma(D) ] \subset \Gamma(D) $. Since 
$\left\langle (X, \alpha), (Y, \beta) \right\rangle = 0 $ if $(X, \alpha), (Y,
\beta) \in \Gamma(D)$, 
integrability of the Dirac structure is often expressed in the literature
relative to a non-skew-symmetric
 bracket that differs from 
\eqref{wrong_bracket} by eliminating in the second line the third term of
the second component. This truncated expression which satisfies the Jacobi
identity but is no longer skew-symmetric is called the \emph{Courant bracket}:
\begin{equation}\label{Courant_bracket}
[(X, \alpha), (Y, \beta) ] : = \left( [X, Y],  \ldr{X} \beta - \ip{Y} \dr\alpha \right).
\end{equation}  

\noindent\emph{Symmetries of Dirac manifolds}
Let $G$ be a Lie group and
$\Phi: G\times M \rightarrow M$ a smooth left action. Then $G$ is called a
\emph{symmetry Lie group of} $D$ if for every $g\in G$ the condition
$(X,\alpha) \in \Gamma(D)$ implies that  $\left( \Phi_g^\ast X, \Phi_g^\ast
  \alpha \right) \in \Gamma(D)$. We say then that the Lie group $G$ acts
\textit{canonically} or \textit{by Dirac actions} on $M$. 

Let $\mathfrak{g}$ be a Lie algebra and $\xi \in \mathfrak{g} \mapsto \xi_M \in
\mathfrak{X}(M)$ be a smooth left Lie algebra action; that is, the map $(x, \xi) \in
M \times \mathfrak{g} \mapsto \xi_M(x) \in TM $ is smooth and $\xi \in\mathfrak{g}
\mapsto \xi_M  \in \mathfrak{X}(M)$ is a Lie algebra anti-homomorphism.  The Lie
algebra $\mathfrak{g}$ is said to be a 
\emph{symmetry Lie algebra of} $D$ if for every $\xi \in \mathfrak{g}$ the condition
$(X,\alpha) \in \Gamma(D)$ implies that  
$\left(\ldr{\xi_M}X,\ldr{\xi_M}\alpha \right) \in
\Gamma(D)$.  Of course, if $\mathfrak{g}$ is the Lie algebra of
$G $ and $\xi\mapsto \xi_M$ the  Lie algebra anti-homomorphism, then if $G $
is a symmetry Lie group  of $D$ it follows that $\mathfrak{g}$ is a symmetry Lie
algebra of $D$.

\section{Differential spaces}\label{sec:differential_spaces}

\subsection{Subcartesian spaces}\label{sub}

A differential structure on a topological space $S$ is a family $C^{\infty
}_{\rm glob}(S)$ of
real-valued functions on $S$ such that:

\begin{enumerate}
\item[\textbf{A1.}] The family 
\[
\{f^{-1}((a,b))\mid f\in C^{\infty }_{\rm glob}(S),\,a,b\in \mathds{R}\} 
\]
is a subbasis for the topology on $S.$

\item[\textbf{A2.}] If $f_{1},...,f_{n}\in C^{\infty }_{\rm glob}(S)$ and $F\in C^{\infty }(
\mathds{R}^{n})$, then $F(f_{1},...,f_{n})\in C^{\infty }_{\rm glob}(S).$

\item[\textbf{A3.}] If $f:S\rightarrow \mathds{R}$ is such that, for every $x\in S$,
there exist an open neighborhood $U_{x}$ of $x$ and a function $f_{x}\in
C^{\infty }_{\rm glob}(S)$ satisfying 
\[
f_{x}\an{U_{x}}=f\an{U_{x}}, 
\]
then $f\in C^{\infty }_{\rm glob}(S).$ 
\end{enumerate}
Here the vertical bar $\mid $ denotes the
restriction. Note that we write $C^{\infty}_{\rm glob}(S)$  to distinguish
this set of functions, whose elements are defined on the whole of $S$, from
sheaves of smooth functions if the space is also endowed with a smooth
structure.

\noindent A differential space is a space $S$ endowed with a differential
structure $C^{\infty }_{\rm glob}(S)$.  Clearly,
smooth manifolds are differential spaces. However, the category of
differential spaces is much larger than the category of manifolds.

Let $R$ and\thinspace $S$ be differential spaces with differential
structures $C^{\infty }_{\rm glob}(R)$ and $C^{\infty }_{\rm glob}(S),$ respectively. 
A map $\phi
:R\rightarrow S$ is said to be smooth if $\phi ^{\ast }(f)
=f\circ\phi\in C^{\infty }_{\rm glob}(R)$
for all $f\in C^{\infty }_{\rm glob}(S)$. A smooth map between differential spaces is a
diffeomorphism if it is invertible and its inverse is smooth.

If $R$ is a differential space with differential structure $C^{\infty }_{\rm glob}(R)$
and $S$ is a subset of $R$, then we can define a differential structure 
$C^{\infty }_{\rm glob}(S)$ on $S$ as follows. A function $f:S\rightarrow \mathds{R}$ is
in $C^{\infty }_{\rm glob}(S)$ if and only if, for every $x\in S,$ there is an open
neighborhood $U$ of $x$ in $R$ and a function $f_{x}\in C^{\infty }_{\rm glob}(R)$ such
that $f\an{S\cap U}=f_{x}\an{S\cap U}$. The differential structure $C^{\infty
}_{\rm glob}(S)$ described above is the smallest differential structure on $S$ such
that the inclusion map $\iota :S\rightarrow R$ is smooth. We shall refer to $
S$ with the differential structure $C^{\infty }_{\rm glob}(S)$ described above as a
differential subspace of $R$. If $S$ is a closed subset of $R,$ then the
differential structure $C^{\infty }_{\rm glob}(S)$ described above consists of
restrictions to $S$ of functions in $C^{\infty }_{\rm glob}(R)$.

A differential space $R$ is said to be locally diffeomorphic to a
differential space $S$ if, for every $x\in R$, there exists a neighborhood $
U$ of $x$ diffeomorphic to an open subset $V$ of $S$. More precisely, we
require that the differential subspace $U$ of $R$ is diffeomorphic to the
differential subspace $V$ of $S$. A differential space $R$ is a smooth
manifold of dimension $n$ if and only if it is locally diffeomorphic to 
$\mathds{R}^{n}$. A Hausdorff differential space that is locally
diffeomorphic to subsets of $\mathds{R}^{n}$ is called a \emph{subcartesian space}.
In the following, we restrict our considerations to subcartesian spaces.

Differential spaces were introduced in  \cite{Sikorski67};
see also \cite{Sikorski71} and \cite{Sikorski72}. The original definition of a
subcartesian space, in terms of a singular atlas, was given in
\cite{Aronszajn67}. 
The characterization of subcartesian spaces
used here can be found in \cite{Walczak73}, where the term \emph{
differential spaces of class $D_{0}$} is used. A comprehensive bibliography
of differential spaces is given in \cite{BuHeMuSa93}.

\subsection{Vector fields}

In this subsection, we review integration of vector fields and distributions
on subcartesian spaces following \cite{Sniatycki03}.

Let $S$ be a subcartesian space with differential structure 
$C^{\infty }_{\rm glob}(S)$.
 A derivation on $C^{\infty }_{\rm glob}(S)$ is an $\R$-linear map $X:C^{\infty
}_{\rm glob}(S)\rightarrow C^{\infty }_{\rm glob}(S):f\mapsto X(f)$ 
satisfying Leibniz' rule 
\begin{equation}
X(f_{1}f_{2})=X(f_{1})f_{2}+f_{1}X(f_{2}).  \label{Leibniz}
\end{equation}
We denote the space of derivations of $C^{\infty }_{\rm glob}(S)$ by 
$\operatorname{Der}
C^{\infty }_{\rm glob}(S).$ It has the structure of a Lie algebra 
with the Lie bracket $
[X_{1},X_{2}]$ defined by 
\[
\lbrack X_{1},X_{2}](f)=X_{1}(X_{2}(f))-X_{2}(X_{1}(f)) 
\]
for every $X_{1},X_{2}\in \operatorname{Der}C^{\infty }_{\rm glob}(S)$ 
and $f\in C^{\infty
}_{\rm glob}(S).$

Let $I$ be an interval in $\mathds{R}$. A smooth map $c:I\rightarrow
S$ is an integral curve of a derivation $X$ if 
\begin{equation}
\frac{d}{dt}f(c(t))=X(f)(c(t)) \label{integral}
\end{equation}
for all $f\in C^{\infty }_{\rm glob}(S)$ and $t\in I$. If $I$ is closed and $t$ is an
endpoint of $I$, then the derivative on the left hand side of equation \eqref{integral} is one-sided. In the limiting case, when $I$ consists of only
one point, the left hand side of (\ref{integral}) is undefined. We extend
the definition of an integral curve to this case by declaring that a map 
$c:\{t_{0}\}\rightarrow S$, with domain consisting of a single point in 
$\mathds{R}$, is an integral curve of every derivation $X$. An integral curve
of $X$ through a point $x_{0}\in S$ is an integral curve $c:I\rightarrow S$
of $X$ such that $0\in I$ and $c(0)=x_{0}$. An integral curve 
$c:I\rightarrow S$ of $X$ through $x_{0}$ is maximal if its domain $I$
contains the domain of every integral curve of $X$ through $x_{0}.$

\begin{remark}
Let $S$ be a subcartesian space. For every derivation 
$X\in \operatorname{Der}C^{\infty }_{\rm glob}(S)$ and each $x_{0}\in S$ there 
exists a unique maximal
integral curve of $X$ through $x_{0}$. A proof of this can be found in 
\cite{Sniatycki02}.
\end{remark}

A \emph{vector field on a subcartesian space} is a derivation $X$ of $C^{\infty
}_{\rm glob}(S) $ such that translations along integral curves of $X$ give rise to a
one-parameter local group $\phi _{t}^{X}$ of local diffeomorphisms of $S$.
In other words, 
\begin{equation*}
\frac{d}{dt}f\left( \phi _{t}^{X}(x)\right) =X(f)\left( \phi
_{t}^{X}(x)\right)
\end{equation*}
for every $f\in C^{\infty }_{\rm glob}(S)$ and each $(t,x)\in \mathbb{R}\times S$ for
which $\phi _{t}^{X}(x)$ is defined. Let $\mathfrak{X}_{\rm glob}(S)$ denote the family
of all vector fields on a subcartesian space $S$. The orbit $S_{x}$ of $
\mathfrak{X}_{\rm glob}(S)$ through $x$ is given by 
\begin{equation}
S_{x}=\{\phi _{t_{n}}^{X_{n}}\circ \ldots \circ \phi _{t_{1}}^{X_{1}}\mid
n\in \mathbb{N},\text{ }t_{1},...,t_{n}\in \mathbb{R},\text{ }
X_{1},...,X_{n}\in \mathfrak{X}(S)\}.  \label{Sx}
\end{equation}

\begin{remark}
\label{inclusion_remark}
Let $\mathfrak{X}_{\rm glob}(S)$ be the family of all vector fields on a subcartesian
space $S$. For each $x\in S$, the orbit $S_{x}$ is a manifold and the
inclusion map $S_{x}\hookrightarrow S$ is smooth. For every family $\mathcal{F}$ 
of vector fields on $S$, orbits of $\mathcal{F}$ are contained in orbits
of $\mathfrak{X}_{\rm glob}(S).$ For a proof of this, see Theorem 4 in 
\cite{Sniatycki03}.
Smoothness of the inclusion map $S_{x}\hookrightarrow S$ is discussed in the
proof of Theorem 3 in \cite{Sniatycki03}.
\end{remark}

\subsection{Stratifications}\label{strat}

A \emph{decomposition} of a differential space $S$ is a partition of $S$ by
a locally finite family $\mathcal{D}$ of smooth manifolds $S_{\alpha }$ of $
S $ such that

\begin{enumerate}
\item each manifold $S_{\alpha }\in \mathcal{D}$ with its manifold structure
is a locally closed differential subspace of $S$ 
\end{enumerate}
and
\begin{enumerate}
\setcounter{enumi}{1}
\item for $S_{\alpha },S_{\beta }\in \mathcal{D}$, if $S_{\alpha }\cap \bar{S
}_{\beta }\neq \emptyset$, then either $S_{\alpha }=S_{\beta }$ or $
S_{\alpha }\subset \bar{S}_{\beta }\backslash S_{\beta }.$
\end{enumerate}

\noindent Manifolds $S_{\alpha }\in \mathcal{D}$ are called \emph{strata} of
the decomposition $\mathcal{D}.$

Decompositions of a differential space $S$ can be partially ordered by
inclusion. If $\mathcal{D}^{1}=\{S_{\alpha }^{1}\}$ and $\mathcal{D}
^{2}=\{S_{\beta }^{2}\}$ are two decompositions of $S$, we say that $
\mathcal{D}^{1}$ is a \emph{refinement} of $\mathcal{D}^{2}$, and write
$\mathcal{D}^1\geq\mathcal{D}^2$, 
if, for every $
S_{\alpha }^{1}\in \mathcal{D}^{1}$, there exists $S_{\beta }^{2}\in \mathcal{
D}^{2}$ such that $S_{\alpha }^{1}\subseteq S_{\beta }^{2}$. We say that $
\mathcal{D}$ is a minimal (coarsest) decomposition of $P$ if it is not a
refinement of a different decomposition of $P$. Note that if $P$ is a
manifold, then the minimal decomposition of $M$ consists of a single
manifold $M=P$. Similarly, we say that $\mathcal{D}$ is a maximal (finest)
decomposition of $P$ if  $\mathcal{D}^{\prime }\geq \mathcal{D}$ implies $
\mathcal{D}^{\prime }=\mathcal{D}$.

Let $\mathcal{D}=\{S_{\alpha }\}$ be a decomposition of $S$. The \emph{stratification}
corresponding to $\mathcal{D}$ is a map $\mathcal{S}$ which associates to
each $x\in S$ the germ at $x$ of the stratum $S_{\alpha }$ containing $x$.
If all strata $S_{\alpha }$ of $\mathcal{D}$ are connected, then $\mathcal{D}$
is uniquely determined by the stratification $\mathcal{S}$ corresponding to $
\mathcal{D}$ (see \cite{LuSn08}). In the following we identify
decompositions of $S$ with connected strata with corresponding
stratifications of $S$.

\subsection{Orbits of a proper action}\label{subsection_orbits}

In this section we consider a smooth and proper action 
\begin{equation}
\begin{array}{cccl}
\Phi &:G\times M&\rightarrow& M\\
&(g,m)&\mapsto& \Phi (g,m)\equiv \Phi
_{g}(m)\equiv gm\equiv g\cdot m  \label{action}
\end{array}
\end{equation}
of a Lie group $G$ on a manifold $M$. Our aim is to describe the
differential structure of the orbit space $\bar M=M/G$. We denote the orbit map
by $\pi :M\rightarrow \bar M$.

For each closed Lie subgroup $H$ of $G$ we define the \emph{isotropy type} set
\begin{equation*}
M_{H}=\{m\in M\mid G_{m}=H\},
\end{equation*}
where $G_{m}=\{g\in G\mid gm=m\}$ is the isotropy subgroup of $m\in M$. Since the
action is proper, all isotropy groups are compact. The sets $M_{H}$, where $
H $ ranges over the closed Lie subgroups of $G$ for which $M_{H}$ is
nonempty, form a partition of $M$, and therefore they are the equivalence 
classes of an equivalence relation in $M$.
Define the normalizer of $H$ in $
G $ by
\[
N(H)=\{g\in G\mid gHg^{-1}=H\}. 
\]
$N(H)$ is a closed Lie subgroup of $G$. Since $H$ is a normal subgroup of $
N(H)$ the quotient $N(H)/H$ is a Lie group. If $m\in M_{H}$, we have $G_{m}$ 
$=H$ and, for all $g\in G,$ $G_{gm}=gHg^{-1}$. As a consequence, $gm$ lies in 
$M_{H}$ if and only if $g\in N(H)$. The action of $G$ on $M$ restricts to an
action of $N(H)$ on $M_{H}$, which induces a free and proper action of 
$N(H)/H$ on $M_{H}$.

Define the \emph{orbit type} set
\begin{equation}
M_{(H)}=\{m\in M\mid G_{m}\text{ is conjugated to }H\mathbf{\}.}  \label{P(H)}
\end{equation}
Then, 
\[
M_{(H)}=\{gm\mid g\in G,m\in M_{H}\}=\pi ^{-1}(\pi (M_{H})). 
\]
Connected components of  $M_{H}$ and $M_{(H)}$ are embedded submanifolds of $M$; 
therefore $M_H$ is called an \emph{isotropy type manifold} and $M_{(H)}$ an 
\emph{orbit type manifold}.
Moreover, 
\[
\pi \left(M_{(H)} \right)=\{gm\mid m\in M_{H}\}/G=M_{H}/N(H)=M_{H}/(N(H)/H). 
\]
But the action of $N(H)/H$ on $M_{H}$ is free and proper which implies that \linebreak
$M_{H}/(N(H)/H)$ is a quotient manifold of $M_{H}$. Hence, $\pi (M_{(H)})$ is
a manifold contained in the orbit space $\bar M=M/G$.

Since the action of $G$ on $M$ is proper, the Slice Theorem of \cite{Palais61} 
ensures that for each $m\in M$ there exists a slice $S_{m}$ for
this action and that $\pi (S_{m})$ is an open subset of $\bar M$ homeomorphic
to $S_{m}/G_{m}$. It follows that 
\[
C^{\infty }_{\rm glob}(\bar M)
=\{f\in C^{0}(\bar M)\mid \pi ^{\ast }(f)\in C^{\infty }(M)\} 
\]
is a differential structure on the orbit space $\bar M$; see Theorem 3.4 of 
\cite{CuSn01}. Moreover, for each slice $S_{m}$, its projection 
$\pi (S_{m})$ to $\bar M$ is diffeomorphic to $S_{m}/G_{m}$ in the sense of
differential spaces. Since $G_{m}$ is compact and the action of $G_{m}$ on 
$S_{m}$ is linear, it follows that the space $C^{\infty }_{\rm diff}(S_{m})^{G_{m}}$ of 
$G_{m}$-invariant smooth functions on $S_{m}$ is given by smooth functions of
algebraic invariants (see \cite{Schwarz75}). Hilbert's theorem ensures that the ring
of $G_{m}$-invariant polynomials on $S_{m}$ is finitely generated (\cite{Weyl46}, 
page 274). Hence, $\bar M$ is locally diffeomorphic to a subset of a
finite dimensional space, which implies that $\bar M$ is subcartesian; see 
\cite{Sniatycki02}.

A partition of the orbit space $\bar{M}=M/G$ by connected components of $\pi
(M_{(H)})$ is a decomposition of the differential space $\bar{M}$. The
corresponding stratification of $\bar{M}$ is called the orbit type
stratification of the orbit space (see \cite{DuKo00}, and \cite{Pflaum01}).
It is a minimal stratification in the partial order discussed above (see 
\cite{Bierstone75}). This implies that the  strata $\pi (M_{(H)})$ of the orbit
type stratification are orbits of the family of all vector fields on $\bar{M}
$ (see \cite{LuSn08}).

Now let $C^\infty(\bar M)$ be the sheaf of smooth functions on $\bar M$ defined as
follows. A function $f:V\to \R$ is an element of $C^\infty(\bar M)$ if
$V\subseteq \bar M$ is an open subset and $\pi^*f\in C^\infty(M)$. This
really defines a sheaf of smooth functions on $\bar M$; see \cite{OrRa04} or
\cite{Duistermaat}.
Proposition 4.7 in \cite{Duistermaat} states that this sheaf can equivalently
be constructed as follows: $f:V\to \R$ is an element of $C^\infty(\bar M)$ if
$V\subseteq \bar M$ is an open subset and for all $x\in V$ there exists
$U_{x}\subseteq \bar M$ open, $x\in U_x$, and 
$f_{x}\in C^\infty_{\rm glob}(\bar M)$ such that 
\[f_{x}\an{U_{x}}=f\an{U_x}.\]
In an analogous manner, we define $C^\infty(\bar P)$ for a stratum $\bar P$ of
$\bar M$. A function $f_{\bar P}:V_{\bar P}\to\R$ is an element of
$C^\infty(\bar P)$
if $V_{\bar P}\subseteq \bar P$ is an open subset and for all $x\in V_{\bar
  P}$ there exists an open neighborhood $U\subseteq M$ of $x$ such that 
$U\cap\bar P\subset V_{\bar P}\subseteq\bar P$ and
$f\in C^\infty_{\rm glob}(\bar M)$ such that 
\[ f_{\bar P}\an{U\cap \bar P}=f\an{U\cap \bar P}.\]
Note that this
implies that for  any $f \in C ^{\infty}(\bar{P})$ and any point $x$ in the domain
of definition of $f$, there exists an open neighborhood $U\subseteq \bar M$ of
$x$ and $f_x\in
C^\infty(\bar M)$ such that 
\[f\an{U\cap \bar P}=f_x\an{U\cap\bar P}.\] 
Hence, by shrinking the domain of definition of $f$, we can see the function
$f_x$ as an extension of $f$ at $x$. We shall  often use this
property in the rest of the paper (without mentioning the ``shrinking'' of the
domain of definition). We will see later that this smooth structure on $\bar
P$ is exactly its smooth structure as the quotient of the stratum
$P=\pi^{-1}(\bar P)$ of $M$.

\medskip

We end this subsection with a proposition on the uniqueness of the restriction
of a vector field on $\bar M$ to a stratum of $\bar M$.
\begin{proposition}\label{restriction_of_X}  
Let $\bp$ be a stratum of $\bar M$.
We know by the considerations above that each vector field $\bar X$ on $\bar
M$ restricts to a 
vector field $X_\bp$ on $\bar P$. We write $X_\bp\sim_{\iota_\bp}\bar X$. If
$X_\bp^1$ and 
$X_\bp^2$ are such that $X^1_\bp\sim_{\iota_\bp}\bar X$ and
$X^2_\bp\sim_{\iota_\bp}\bar X$, 
then they have to be equal.
\end{proposition}

\begin{proof}
If $\bar X\in\mx(\bar M)$ restricts to a global vector field $X_\bp\in\mx(\bp)$, 
we have for all
$\bar f\in
C^\infty_{\rm glob}(\bar M)$:
\[X_\bp(\iota_\bp^*\bar f\,)=\bar X(\bar f)\circ\iota_\bp.
\]
Since each function $f_\bp\in C^\infty(\bp)$ is locally the restriction to
$\bar P$ of  some $\bar f\in C^\infty_{\rm glob}(\bar M)$ and the derivations on $\bp$
correspond exactly to the vector fields on $\bp$ (since $\bar P$ is a smooth
manifold), 
we  automatically get the
uniqueness of $X_\bp$.
\end{proof}

\section{Generalized distributions and orthogonal spaces}\label{dis}

We will need a few standard facts from the theory of generalized distributions
on a smooth manifold 
$ M $ (see \cite{Stefan74a, Stefan74b, Stefan80}, \cite{Sussmann73} for the
original articles and 
\cite{LiMa87}, \cite{Vaisman94}, \cite{Pflaum01}, or \cite{OrRa04},  
for a quick review of this theory). 
\medskip

Let $E$ be a vector bundle over $M$.
A  \emph{generalized subdistribution} $ \Delta $ of $E$ is a subset
$\Delta$ of $E$  such that for each $m\in M$, the set $ \Delta (m) : 
= \Delta \cap E(m)$ is a
vector subspace of $E_m$. The number $\dim \Delta (m) $ is called the
\emph{rank}  of $ \Delta $ at $ m  \in  M $. A point $m\in M$ is a
\emph{regular} point of the distribution $\Delta$ if there exists a
neighborhood $U$ of $m$ such that the rank of $\Delta$ is constant on
$U$. Otherwise, $m$ is a \emph{singular} point of the distribution. 

A local \emph{differentiable
  section} of $ \Delta $ is a smooth section $\sigma\in \Gamma(E)$
defined on some open subset $ U \subset  M $ such that $ \sigma(u) \in  \Delta (u)
$ for each $ u \in  U $. We denote by $ \Gamma(
\Delta ) $ the space
of local sections of $ \Delta $. A generalized
subdistribution is said to be \emph{differentiable} or \emph{smooth} if for every
point $ m \in  M $ and every vector $ v \in  \Delta (m) $, there is a
differentiable section $ \sigma \in  \Gamma ( \Delta ) $ defined on an open
neighborhood $ U $ of $ m $ such that $ \sigma(m) = v $.

A smooth generalized subdistribution of  the tangent space $TM$ 
(that is, with $E=TM$) will
 simply be called a \emph{smooth tangent distribution}; a smooth generalized
subdistribution of the
cotangent space $T^*M$ will
be called a \emph{smooth cotangent distribution}. We will work most of the time 
with smooth generalized
subdistributions of the Pontryagin bundle $E=TM\operp T^*M$, which will
be called \emph{smooth generalized distributions}.   

\begin{example}\label{ex_Dirac_dis}
A Dirac structure $D$ on  a manifold $M$ defines two smooth tangent distributions 
$\mathsf{G}_0, \mathsf{G}_1 \subset TM $ and two smooth cotangent distributions 
$\mathsf{P}_0, \mathsf{P}_1 \subset T^*M$:
\begin{align*}
\mathsf{G}_0(m)&:= \left\{X(m) \in T_mM \left|
X \in \mathfrak{X}(M), (X, 0) \in
\Gamma(D) \right.\right\}, \\
\mathsf{G}_1(m)&:= \left\{X(m) \in T_mM \left|\begin{array}{c}  
X \in \mathfrak{X}(M), \, \text{there exists}\;
\alpha \in \Omega^1(M)\\ \text{such that}\; 
(X, \alpha) \in \Gamma(D)
\end{array}\right.\right\}
\end{align*}
and
\begin{align*}
\mathsf{P}_0(m)&:= \{\alpha(m) \in T^*_mM \mid \alpha \in \Omega^1(M), 
(0, \alpha) \in \Gamma(D) \}, \\
\mathsf{P}_1(m)&:= \left\{\alpha(m) \in T^*_mM \left|
\begin{array}{c}  \alpha \in \Omega^1(M),\, \text{there
exists }\;
X \in\mathfrak{X}(M)\\
 \text{such that}\; (X, \alpha) \in \Gamma(D) 
\end{array}\right.\right\}.
\end{align*}
The smoothness of $\mathsf{G}_0, \mathsf{G}_1, \mathsf{P}_0, \mathsf{P}_1$ is
obvious since, by definition, 
they are generated by smooth local sections. In general, these are not vector
subbundles of $TM$ and 
$T ^\ast M $, respectively. It is also clear that $\mathsf{G}_0 \subset
\mathsf{G}_1$ and 
$\mathsf{P}_0 \subset \mathsf{P}_1$. \end{example}

\subsection{Generalized foliations and integrability of tangent distributions}
To give content to the notion of integrability of a smooth tangent
distribution and elaborate 
on it, we need to quickly review the concept and main properties of generalized 
foliations. 
A \emph{generalized foliation} on $ M $ is a partition $\mathfrak{F} : = 
\{ \mathcal{L} _\alpha \}_{ \alpha \in  A}$ of $ M $  into disjoint connected
sets, called \emph{leaves},  such that each point $ m \in  M $ has a
\emph{generalized foliated chart} $ (U, \varphi : U \rightarrow  V \subseteq
\mathds{R} ^{\dim M}) $, $ m \in  U $. This means that  there is some natural
number $ p_\alpha \leq \dim M $, called the \emph{dimension} of the leaf $
\mathcal{L} _\alpha $, 
and a subset $S _\alpha \subset  \mathds{R} ^{\dim M - p_\alpha } $ such that 
$ \varphi(U \cap \mathcal{L} _\alpha ) = \{(x^1, \ldots ,x^{\dim M} ) \in  V
\mid  
(x^{p_\alpha+1}, \ldots  ,x^{\dim M} ) \in  S _\alpha \} $. The key difference
with 
the concept of foliation is that the number $ p_\alpha $ can change from leaf
to leaf.  
Note that each $ (x^{p_\alpha+1}_ \circ  , \ldots  ,x^{\dim M}_ \circ  ) \in
S _\alpha$ 
determines a connected component $ ( U \cap \mathcal{L}_\alpha )_ \circ $ of
$U \cap \mathcal{L}_\alpha$, 
that is, $ \varphi (( U \cap \mathcal{L}_\alpha )_ \circ ) 
= \{(x^1, \ldots , x^{p_\alpha}, x^{p_\alpha+1}_ \circ  , \ldots  ,x^{\dim M}_
\circ  ) \in V \} $. 
The generalized foliated charts induce on each leaf a smooth manifold
structure that 
makes them into initial submanifolds of $ M $.

A leaf $ \mathcal{L} _\alpha $ is called \emph{regular} if it has an open
neighborhood 
that intersects only leaves whose dimension equals $ \dim \mathcal{L} _\alpha
$. 
If such a neighborhood does not exist, then $ \mathcal{L} _\alpha $ 
is called a \emph{singular} leaf. A point is called \emph{regular}
(\emph{singular}) 
if it is contained in a regular (singular) leaf. The set of vectors tangent to
the 
leaves of  $ \mathfrak{F} $ is defined by
\[
T(M,  \mathfrak{F}): 
= \bigcup _{ \alpha \in  A} 
\bigcup_{m \in  \mathcal{L} _\alpha} T _m \mathcal{L} _\alpha \subset  TM.
\]

\medskip
Let us now turn  to the relationship between distributions and generalized
foliations.
In all that follows, $ \T $ is a smooth tangent distribution. An
\emph{integral manifold} of $ \T $ is an injectively immersed connected
manifold $ \iota _L : L \hookrightarrow M $,  where $ \iota _L $  is the
inclusion, satisfying the condition $ T _m \iota_L (T _m  L )  \subset  \T (m)
$  for every $ m \in  L $. The integral manifold $ L $ is of \emph{maximal
  dimension} at $ m \in  L $ if $ T _m \iota_L (T _m  L )  =  \T (m) $. The
distribution $ \T $ is \emph{completely integrable} if for every $m \in  M $
there is an integral manifold $ L $ of $ \T $, $ m \in  L $, everywhere of
maximal dimension. The distribution $\T$ is \emph{involutive} if it is
invariant under the (local) flows associated to differentiable sections of
$\T$. The distribution $\T$ is \emph{algebraically involutive} if for any two
smooth vector fields defined on an open set of $M $ which take values in $\T$,
their bracket also takes values in $\T$. Clearly involutive distributions are
algebraically involutive and the converse is true if the distribution is a
subbundle. 

Recall that the Frobenius theorem states that a vector subbundle of $ TM $ is  
(algebraically) involutive if and only if it is the tangent bundle of a
foliation on $M $. 
The same is true for distributions: \textit{A smooth distribution is
  involutive 
if and only if it coincides with the set of vectors tangent to a generalized
foliation, 
that is, it is completely integrable.}  This is known as the Stefan-Sussmann Theorem.

We will give the Stefan-Sussmann Theorem in the more general setting of a
smooth tangent distribution spanned by a family of vector fields. Note that 
each smooth tangent distribution is spanned by the family of its smooth sections.

\medskip

Let 
$F$  be an everywhere defined family of local vector fields on $M$.
By {\it everywhere defined} we mean that for every $m \in M$ there
exists
$X \in F$ such that $m \in {\rm Dom}(X) $. 
We can
associate to the flows of the vector fields in $F$ a
set of local
diffeomorphisms $\mathcal{A}_F:=\{\phi _t\mid \,\phi _t \text{ flow of } X \in
F\} $ of $M$ and
a pseudogroup of transformations generated by it,
\[
A _F:=(\mathds I, M)\cup\{\phi _{t _1}^1 \circ \cdots\circ \phi _{t
_n}^n\mid n \in \mathds
N\text{ and  } \phi _{t _n}^n \in \mathcal{A}_F\text{ or } (\phi
_{t _n}^n) ^{-1} \in \mathcal{A}_F\}.
\]
Analogously, we also define, for any $z  \in M $, the following vector
subspaces of $T
_zM $:
\begin{eqnarray*}
\mathcal{D}_F (z)&:=&{\rm span\,}\left\{ \left. \frac{d}{dt}\right\arrowvert_{t=t_0}
\!\!\!\phi _t (y)
\,\Bigr|\, \phi _t\text{ flow of }X
\in F,
\, \phi _{t_0} (y)=z\right\}\\
        &=&\operatorname{span}\{X(z)\in T_zM| \,X\in F \text{ and }
z \in {\rm Dom}(X)\},\\
D _F (z) &:=&{\rm span\,}\{T _y \mathcal{F}_T \cdot \mathcal{D}_F
(y)\mid \mathcal{F}_T
\in A _F,
\mathcal{F}_T(y)=z\}.
\end{eqnarray*}
Note that, by construction, $\mathcal{D}_F$ is a smooth tangent distribution.   We
will say that
$\mathcal{D}_F$ is the smooth tangent distribution \emph{spanned} by
$F$. 

The $A _F $-orbits,
also called the
\emph{accessible sets}
of the family $F$,  form a generalized foliation  whose leaves
have as tangent spaces the values of $D _F$ (see, for example, \cite{OrRa04}). 
An important question is
determining when
the smooth tangent distribution $\mathcal{D}_F $ spanned by $F$ is
integrable.

\begin{theorem}
\label{frobenius with vector fields}
{\rm(\cite{Stefan74a} and~\cite{Sussmann73})}.
Let $\mathcal{D}_F$ be a differentiable generalized distribution on the smooth
manifold $M$ spanned by an everywhere defined
family of vector fields $F$. The following properties are
equivalent:
\begin{enumerate}
\item The distribution $\mathcal{D}_F$ is invariant  under the
pseudogroup of transformations generated by $F$; that is, for each
$\mathcal{F}_T\in
A_{F}$ and for each
$z\in M$ in the domain of $\mathcal{F}_T$,
\[
T_z\mathcal{F}_T(\mathcal{D}_F (z))=\mathcal{D}_F(\mathcal{F}_T(z)).
\]
\item  $\mathcal{D}_F= D _F $.
\item For any $X \in F $ with flow $\phi _t$ and any $x\in {\rm
Dom}(X)$, there
exist:
\begin{enumerate}
\item A finite set $\{X _1, \ldots, X _p\}\!\subset\! F$ such
that $$\mathcal{D}_F
(x)= {\rm  span}\{X _1(x), \ldots, X _p(x)\}.$$
\item A constant $\epsilon>0 $ and Lebesgue integrable functions
$\lambda_{ij}:(-
\epsilon,
\epsilon) \rightarrow \mathds R$ ($1 \leq i,j \leq p$) such that for
every $t \in (-
\epsilon, \epsilon)$ and $j \in \{1, \ldots, p\} $:
\[
[X, X _j](\phi _t(x))=\sum _{i=1}^{p}\lambda_{ij}(t) X _i(\phi _t(x))
\]
and $\mathcal{D}_F
(\phi _t(x))= {\rm  span}\{X _1(\phi _t(x)), \ldots, X _p(\phi _t(x))\} $.
\end{enumerate}
\item The distribution $\mathcal{D}_F$ is integrable and its
maximal integral
manifolds are  the $A_{F}$-orbits.
\end{enumerate}
\end{theorem}

\medskip

As already mentioned, given an involutive (and hence a completely
integrable) distribution $ \T $, each point $ m \in  M $ belongs to
exactly one connected integral manifold $ \mathcal{L} _m $ that is maximal
relative to inclusion. It turns out that $ \mathcal{L} _m $ is an initial
submanifold and that it is also the \emph{accessible} set of $ m $; that is, $
\mathcal{L} _m $ equals the subset of points in $ M $ that can be reached by
applying to $ m $ a composition of a finite number of  of flows of elements of
$\Gamma (\T)$.  The collection of all maximal integral submanifolds of $
\T $ forms a generalized foliation $ \mathfrak{F} _\T $ such that $
\T = T(M, \mathfrak{F} _\T) $.  Conversely, given a generalized
foliation $ \mathfrak{F} $ on $ M $, the subset $ T(M, \mathfrak{F}) \subset
TM $ is a smooth completely integrable (and hence involutive) distribution
whose collection of maximal integral submanifolds coincides with $
\mathfrak{F} $. These two statements expand the Stefan-Sussmann Theorem cited
above.

\subsection{Generalized smooth subdistributions and annihilators}
Assume in this section that $E$ is a vector bundle on $M$ that is endowed 
with a smooth nondegenerate
symmetric pairing $\pairing_E$.
If $E=TM\operp T^*M$ is the Pontryagin bundle, this pairing
$\pairing_{TM\operp T^*M}$ will always be the symmetric pairing $\pairing$ 
defined in \eqref{pairing}.
If $\Delta \subset E $ is a smooth subdistribution of $E$, its
\emph{smooth} \emph{orthogonal} distribution  is the smooth
generalized subdistribution $\Delta^\perp $ of $E$  defined by 
\begin{align*}
\Delta^ \perp (m): =
 \left\{\tau(m) \left|  
\begin{array}{c}\tau \in
\Gamma(E) \text{ with } m\in \dom(\tau) \text{ is such that  for all }\\
\sigma \in\Gamma(\Delta) \text{  with } m\in\dom(\sigma),\\
\text{ we have }  \left\langle \sigma,\tau
\right\rangle_E = 0 \text{ on } \dom(\tau)\cap\dom(\sigma)\end{array}\right.\right\}.
\end{align*}
Here we have the, in general strict, inclusion $\Delta \subset \Delta^{ \perp\perp } $.
 Note that the smooth orthogonal distribution of a smooth generalized
 subdistribution
 is smooth by construction. If the distribution $\Delta$ is
a vector subbundle of $E$, then its smooth orthogonal distribution is
also a vector subbundle of $E$. Note that the
smooth orthogonal distribution of a smooth generalized subdistribution $\Delta$ of
$E$ is in general different from the \emph{pointwise}
orthogonal distribution of $\Delta$, defined by
\begin{align*}
\Delta^{\perp_{\rm p}} (m): = \{v_m\in E(m)\mid
\pair{v_m,w_m}_E=0 \text{ for all } w_m\in\Delta(m)\},
\end{align*}
where the subscript $p$ stands for ``pointwise''. The pointwise orthogonal
distribution of a smooth generalized subdistribution $\Delta$ is not smooth in
general. The proof of the following
proposition is easy, and we omit it here.
\begin{proposition}\label{point_smooth_annihi}
Let $\Delta$ be a smooth generalized subdistribution of $E$. Then we have
\[\Delta^\perp\subseteq \Delta^{\perp_{\rm p}},
\quad \Delta=\Delta^{\perp_{\rm p}\perp_{\rm p}},
\text{ and } \quad \Delta\subseteq\Delta^{\perp\perp}.\]
If $\Delta$ is itself a vector bundle over $M$, the smooth orthogonal
distribution $\Delta^\perp$ of
$\Delta$ is also a subbundle of $E$, and  we have
$\Delta^\perp=\Delta^{\perp_{\rm p}}$.
\end{proposition}

We use this to show the following proposition about the smooth annihilator of a sum
of vector subbundles of $E$.
\begin{proposition}\label{prop_intersection}
Let $\Delta_1$ and $\Delta_2$ be smooth subbundles of the vector bundle
$(E,\langle,\rangle)$.
Since $\Delta_1$ and $\Delta_2$ have constant ranks on $M$, their smooth
orthogonal spaces $\Delta_1^\perp$ 
and $\Delta_2^\perp$ are also smooth subbundles of $E$ and equal to the
pointwise orthogonals of $\Delta_1$ and $\Delta_2$. 
The following are equivalent:
\begin{enumerate}
\item The
intersection
$\Delta_1^\perp\cap\Delta_2^\perp$ is smooth.
\item $
(\Delta_1+\Delta_2)^\perp=\Delta_1^\perp\cap\Delta_2^\perp$.
\item
$(\Delta_1^\perp\cap\Delta_2^\perp)^\perp=\Delta_1+\Delta_2$.
\item $\Delta_1^\perp\cap\Delta_2^\perp$ has constant rank on $M$.
\end{enumerate}
\end{proposition}

\begin{proof}
Let  $\sigma \in\Gamma\left((\Delta_1+\Delta_2)^\perp\right)$. Then for
all $\sigma_1\in\Gamma(\Delta_1)$ and $\sigma_2\in\Gamma(\Delta_2)$,
we have $\langle \sigma,\sigma_1+\sigma_2\rangle=0 $ on the common domain of
definition of the three sections. Applying this to
$\sigma_1\in\Gamma(\Delta_1)$ and $\sigma_2=0$ (respectively
$\sigma_2\in\Gamma(\Delta_2)$ and $\sigma_1=0$),
we get $\sigma \in\Gamma(\Delta_1^\perp)$ (respectively
$\sigma\in\Gamma(\Delta_2^\perp)$).
Hence, we have shown that the inclusion $(\Delta_1+\Delta_2)^\perp\subseteq
\Delta_1^\perp\cap\Delta_2^\perp$ is always true.

\medskip

Using this, we show  first that if  $\Delta_1^\perp\cap\Delta_2^\perp$ 
is smooth, we have 
\begin{equation}\label{intersmooth}
(\Delta_1+\Delta_2)^\perp=\Delta_1^\perp\cap\Delta_2^\perp.
\end{equation}

We have only to show the inclusion
$(\Delta_1+\Delta_2)^\perp\supseteq\Delta_1^\perp\cap\Delta_2^\perp$.
Choose
$e_m\in(\Delta_1^\perp\cap\Delta_2^\perp)(m)$. Since the intersection
$\Delta_1^\perp\cap\Delta_2^\perp$ is smooth, there exists a section
$\sigma\in\Gamma\left(\Delta_1^\perp\cap\Delta_2^\perp\right)$ with
$\sigma(m)=e_m$.
Let $\sigma_1\in\Gamma(\Delta_1)$  and $\sigma_2\in\Gamma(\Delta_2)$. Since
$\sigma \in\Gamma\left(\Delta_1^\perp\cap\Delta_2^\perp\right)$, we have
$\langle \sigma,\sigma_1\rangle=\langle \sigma,\sigma_2\rangle=0$, and  hence 
$\langle \sigma,\sigma_1+\sigma_2\rangle=0$. From this it follows that 
$\sigma\in\Gamma\left((\Delta_1+\Delta_2)^\perp\right)$ and hence
$e_m\in(\Delta_1+\Delta_2)^\perp(m)$.

Conversely, if the equality in \eqref{intersmooth} holds, the intersection
$\Delta_1^\perp\cap\Delta_2^\perp$ is the smooth annihilator of
$\Delta_1+\Delta_2$ and  is thus smooth by definition.
Hence, we have shown ``$(1)\Leftrightarrow (2)$''.

\medskip

If \eqref{intersmooth} holds, we have
\begin{equation}\label{eq67}
(\Delta_1+\Delta_2)^{\perp_{\rm p}}(m)
=(\Delta_1(m)+\Delta_2(m))^\perp=\Delta_1(m)^\perp\cap\Delta_2(m)^\perp
= (\Delta_1^\perp\cap\Delta_2^\perp)(m),
\end{equation}
and hence, using  Proposition \ref{point_smooth_annihi}:
\[(\Delta_1^\perp\cap\Delta_2^\perp)^\perp=
\left((\Delta_1+\Delta_2)^{\perp_{\rm p}}\right)^\perp\subseteq
(\Delta_1+\Delta_2)^{\perp_{\rm p}\perp_{\rm p}}=\Delta_1+\Delta_2.\]
The converse inclusion follows also from Proposition  \ref{point_smooth_annihi}:
\[\Delta_1+\Delta_2\subseteq
(\Delta_1+\Delta_2)^{\perp\perp}=(\Delta_1^\perp\cap\Delta_2^\perp)^\perp.\]

Conversely, the equality
$\Delta_1+\Delta_2=(\Delta_1^\perp\cap\Delta_2^\perp)^\perp$
implies that $\Delta_1^\perp\cap\Delta_2^\perp$ $\subseteq
(\Delta_1^\perp\cap\Delta_2^\perp)^{\perp\perp}=(\Delta_1+\Delta_2)^\perp$
with Proposition  \ref{point_smooth_annihi}, and 
we have shown the converse implication at the beginning of this proof.
This shows ``$(2)\Leftrightarrow (3)$''.

\medskip

Assume again that \eqref{intersmooth} holds. Then it implies
\eqref{eq67} as above.
The equalities \eqref{eq67} and  \eqref{intersmooth}  then yield together:
\[(\Delta_1+\Delta_2)^{\perp_{\rm
    p}}\overset{\eqref{eq67}}
=\Delta_1^\perp\cap\Delta_2^\perp\overset{\eqref{intersmooth}}
=(\Delta_1+\Delta_2)^{\perp}.\]

 But this is only possible if $\Delta_1+\Delta_2$ has constant rank
on $M$, which yields, using \eqref{intersmooth},
the fact that $\Delta_1^\perp\cap\Delta_2^\perp$ has constant rank on $M$,
too. Hence, we have proved the implication ``$(2)\Rightarrow (4)$''.

\medskip

To finish the proof, we see that the implication ``$(4)\Rightarrow (3)$''
is easy. If $\Delta_1^\perp\cap\Delta_2^\perp$ has constant rank on $M$, then
its smooth annihilator is equal to its pointwise annihilator and we get
\[\left(\Delta_1^\perp\cap\Delta_2^\perp\right)^\perp
=\Delta_1^{\perp\perp}+\Delta_2^{\perp\perp}
=\Delta_1+\Delta_2
\]
since $\Delta_1$ and $\Delta_2$ have constant rank on $M$.
\end{proof}

A tangent (respectively cotangent) distribution $\T\subseteq TM$ (respectively
$\mathcal{C}\subseteq T^*M$) can be identified with
the smooth generalized distribution $\T\operp\{0\}$ 
(respectively $\{0\}\operp\mathcal{C}$). 
The smooth orthogonal distribution of $\T\operp\{0\}$ in $TM\operp T^*M$ is 
easily computed to be
$(\T\operp\{0\})^\perp=TM\operp \T^\circ$,
where
\[\T^\circ(m)=\left\{\alpha(m)\left|
\begin{array}{c}
 \alpha\in\Omega^1(M), m\in\dom(\alpha) \text{
  and } 
\alpha(X)=0\\ \text{ on }\dom(\alpha)\cap\dom(X)\text{ for all }
X\in\Gamma(\T)
\end{array}
\right.\!\!\right\}
\]
for all $m\in M$. This smooth cotangent distribution will be called the
\emph{smooth annihilator} of $\T$. Analogously, we define the smooth annihilator
$\C^\circ$ of a
cotangent distribution $\C$. Then $\C^\circ$ is a smooth tangent distribution
and we have $(\{0\}\operp\C)^\perp=\C^\circ\operp T^*M$. 
The pointwise annihilator of a smooth tangent distribution $\T$ (respectively
of a smooth cotangent distribution $\C$), will be written 
$\T^{\ann}$ (respectively $\C^{\ann}$), and is such that 
$(\T\operp\{0\})^{\perp_{\rm p}}=TM\operp \T^{\ann}$ 
(respectively $(\{0\}\operp\C)^{\perp_{\rm p}}=\C^{\ann}\operp T^*M$).
We get as in Proposition \ref{point_smooth_annihi}:
\[\T^\circ\subseteq \T^{\ann},\quad  \T=\T^{\ann\ann}, \quad \text{ and }
\T\subseteq \T^{\circ\circ},\]
and analogously for $\C$. If $\T$ is a smooth subbundle of $TM$, then
$\T^\circ=\T^{\ann}$ is also a smooth subbundle of $T^*M$.

\medskip
  
The tangent distribution $\V$ spanned by the fundamental vector fields of the
action of a Lie group $G$ on a manifold $M$ will be of great importance later on. 
At every point $m \in M $ it is defined by
\[\V(m)=\{\xi_M(m)\mid \xi\in\lie g\}.\]
If the action is not free, the rank of the fibers of $\V$
can vary on $M$. The smooth annihilator  $\V^\circ$ of $\V$ is given by
\[\V^\circ(m)=\left\{\alpha(m)\left|\begin{array}{c}
 \alpha\in\Omega^1(M), \,m\in\dom(\alpha),\\\text{
  such that } \alpha(\xi_M)
=0 \text{ for
  all }\xi \in \lie g
\end{array}
\right.\!\!\right\}.
\]
We will also use the smooth generalized distribution $\K:=\V\operp \{0\}$ and its
smooth orthogonal space $\K^\perp=TM\operp\V^\circ$.

\section{Proper actions and orbit type manifolds}\label{sec:proper}
\subsection{Tube Theorem and $G$-invariant average}\label{tubeth}
If the action of the Lie group $G$ on $M$ is proper, we can find for each
point $m\in M$ a $G$-invariant neighborhood of $m$ such that the action can be
described easily on this neighborhood. The proof of the following theorem can
be found, for example,  in \cite{OrRa04}.
\begin{theorem}[Tube Theorem]
Let $M$ be a manifold and $G$ a Lie group
acting properly on $M$. For a given point $m\in M$ denote $H := G_m$. 
Then there exists a 
$G $-invariant open neighborhood $U$ of the orbit $G\cdot m$, called 
\emph{the tube at $m$}, and a $G $-equivariant diffeomorphism
$G \times_H B \stackrel{\sim}\longrightarrow  U$. The set $B$ is an open 
$H$-invariant neighborhood of $0$
in an $H $-representation space $H$-equivariantly isomorphic to $T_mM/T_m(G\cdot m)$. 
The $H $-representation on $T_mM/T_m(G\cdot m)$ is given by 
$h\cdot (v + T_m(G \cdot m)) := T_m\Phi_h\cdot v + T_m(G
\cdot m)$, $h \in H $, $v \in T_mM $. The smooth manifold  
$ G \times_H B$ is the quotient of the smooth 
free and proper (twisted)
action $\Psi$ of $H$ on $G\times B$ given by $\Psi(h,(g,b)):=(g
h^{-1},h\cdot b)$, $g \in G $, $h \in H $, $b \in B$. 
The $G $-action on $G \times _H B $ is given by $k\cdot [g, b]: = [kg, b]_H $, 
where $k, g \in G $, $b \in B $, and $[g, b]_H \in G \times _H B $ is the
equivalence class 
{\rm (}i.e., $H $-orbit{\rm )} of $(g,b)$.
\end{theorem}

\emph{$G$-invariant average.}
Let $m\in M$ and $H:=G_m$. If the action of $G$ on $M$ is proper, then the isotropy
subgroup $H$ of $m$ is a compact Lie subgroup of $G$. Hence, there exists a
Haar measure $dh $ on $H$, that is, a $G$-invariant measure on $H$ satisfying
$\int_Hdh=1$ (see, for example, \cite{DuKo00}). Here the left $G$-invariance of
$dh$ is equivalent to the right $G$-invariance of $dh$, and we have
$R_{h'}^*dh=dh=L_{h'}^*dh$ for all $h'\in H$, where $L_h:H\to H$ (respectively
$R_h:H\to H$) denotes left (respectively right) translation by $h$ on $H$.  

Let $X\in\mx(M)$ be defined on the tube $U$ at $m\in M$ of the proper action 
of the Lie
group $G$ on $M$. Using the Tube Theorem, we write  the points of $U$ as
equivalence classes $[g,b]_H$ with $g\in G$ and $b\in B$.
Note that for all $h\in H$, we have
$[g,b]_H=[gh^{-1},hb]_H$. Furthermore, the action of $G$ on $U$ is given by
$\Phi_{g'}([g,b]_H)=[g'g,b]_H$. Define the vector field $X_G$ by the following:
\begin{align*}
X_G([g,b]_H)=\left(\Phi_{g^{-1}}^*\left(\int_H\Phi_h^*Xdh\right)\right)([g,b]_H);
\end{align*}
that is, for each point $m'=[g,b]_H\in U$ we have
\begin{align*}
X_G([g,b]_H)=T_{[e,b]_H}\Phi_g
\left(\int_H\left(T_{[h,b]_H}\Phi_{h^{-1}}X([h,b]_H)\right)dh\right).
\end{align*}
We have to show that this definition doesn't depend on the choice of the
representative $[g,b]_H$ for the point $m'$. 
Write $m'=[gh^{-1},hb]_H$ with some $h\in H$, and compute
\begin{align*}
&X_G([gh^{-1},hb]_H)\\
=&\ T_{[e,hb]_H}\Phi_{gh^{-1}}\left(\int_H\left(T_{[\tilde
      h,hb]_H}\Phi_{\tilde h^{-1}}X([\tilde h,hb]_H)\right)d\tilde h\right)\\
=&\ T_{[h^{-1},hb]_H}\Phi_{g}\circ T_{[e,hb]_H}\Phi_{h^{-1}}
\left(\int_H\left(T_{[e,\tilde hhb]_H}\Phi_{\tilde
      h^{-1}}X([e,\tilde hhb]_H)\right)d\tilde h\right)\\
=&\ T_{[e,b]_H}\Phi_{g}\left(\int_H\left(T_{[e,\tilde hhb]_H}\Phi_{h^{-1}\tilde
      h^{-1}}X([e,\tilde hhb]_H)\right)d\tilde h\right)\\
=&\ T_{[e,b]_H}\Phi_{g}\left(\int_H\left(T_{[e,\tilde hhb]_H}\Phi_{(\tilde
      hh)^{-1}}X([e,\tilde hhb]_H)\right)R_h^*d\tilde h\right)\\
\overset{h':=\tilde hh}=&T_{[e,b]_H}\Phi_{g}
\left(\int_H\left(T_{[e,h'b]_H}\Phi_{
      {h'}^{-1}}X([e,h'b]_H)\right)dh'\right)\\
=&\ X_G([g,b]_H),
\end{align*}
where we have used the equality $d\tilde h=R_h^*d\tilde h$. 
The vector field  $X_G$ is an element of $\mx(M)^G$: letting $[g,b]_H\in U$ and
$g'\in G$, we have
\begin{align*}
(\Phi_{g'}^*X_G)([g,b]_H)&=T_{[g'g,b]_H}\Phi_{{g'}^{-1}}X_G([g'g,b]_H)\\
&=T_{[g'g,b]_H}\Phi_{{g'}^{-1}}\circ T_{[e,b]_H}\Phi_{g'g}
\left(\int_H\left(T_{[h,b]_H}\Phi_{
      h^{-1}}X([h,b]_H)\right)dh\right)\\
&=T_{[e,b]_H}\Phi_{g}
\left(\int_H\left(T_{[h,b]_H}\Phi_{
      h^{-1}}X([h,b]_H)\right)dh\right)\\
&=X_G([g,b]_H).
\end{align*}
At last, we should show that $X_G$ is smooth. Let $X^H: =\int_H\Phi_h^*Xdh$ be
the averaged vector 
field which is clearly smooth on $U\simeq G\times_H B$. Let 
$\Psi:H\times(G\times B)\to G\times
B$ be the (smooth free and proper) twisted action of $H$ on $G\times B$, that
is, $\Psi(h(g,b))=\Psi_h(g,b)=(gh^{-1},hb)$ for all $g\in G$, $b\in B$, 
$h\in H$, and let $\pi_H:G\times B\to
G\times_H B\simeq U$ be the projection. We write $\Phi:G\times(G\times B)\to
G\times B$ for the left action of $G$ on $G\times B$, given by $g\cdot
(g',b)=(gg',b)$. Note that $\pi_H$ is $G$-equivariant.
Let $\widetilde{X^H}$ be an $H$-invariant vector field on $G\times B$ such that
$\widetilde{X^H}\sim_{\pi_H} X^H$. Since $\widetilde{X^H}\in\mx(G\times B)$,
it can be written as a 
sum $\widetilde{X^H}=X^{G}+X^{B}$ with
$X^{G}\in\Gamma(TG\times 0_B)$ and $X^{B}\in\Gamma(0_G\times TB)$.
Since $X^G$ is smooth, $X^G\an{\{e\}\times B}$ is also smooth, and there
exists a smooth function 
$\xi:B\to
\lie g$ such that $X^G(e,b)=(\xi(b),0)\in\lie g\times 0_b$ for all $b\in
B$. Let $\phi^B_t$ be the flow of $X^B$. The points $\phi^B_t(e,b)$ are elements of
$\{e\}\times B$ for each $t$ where $\phi^B_t(e,b)$ is defined.
Define $Y\in\mx(G\times B)$ by 
\begin{align*}
Y(g,b)&:=T_{(e,b)}\Phi_g\widetilde{X^H}(e,b)
=T_{(e,b)}\Phi_gX^G(e,b)+T_{(e,b)}\Phi_gX^B(e,b)\\
&=:Y^G(g,b)+Y^B(g,b).
\end{align*}
The vector fields $Y^G$ and $Y^B$ have $\phi^G_t(g,b)=\Phi_{g\exp(t\xi(b))}(e,b)$
and $\phi^B_t(g,b)=\Phi_g\circ \phi^B_t(e,b)$ as flows, which are obviously
smooth. Hence the two vector fields $Y^G$ and $Y^B$ are smooth and so is
$Y$. It is easy to see, using the fact that
$\Psi_h\circ\Phi_g=\Phi_g\circ\Psi_h$ for all $g\in G$ and $h\in H$, that the
vector field $Y$ remains $H$-invariant and hence descends to $G\times_H
B$. The construction of $Y$ and the $G$-equivariance of $\pi_H$, yield that
$Y\sim_{\pi_H}X_G$. This automatically implies that $X_G$ is smooth.
We call $X_G$ the \emph{$G$-invariant average\/} of the vector field
$X$. Note that $X_G$ is, in general, not equal to $X$ (at any point); it can
even vanish. Indeed, we will see in the following that $G$-invariant vector fields
are tangent to the orbit type manifolds (in reality, they are even tangent to the
isotropy type manifolds; see \cite{OrRa04}). Hence, if we choose a $G$-invariant 
Riemannian 
metric on $M$
and a section $X$ of the ($G$-invariant) orthogonal $TP^\perp\subseteq TM\an{P}$ 
of $TP$
relative to this metric, where $P$ is a stratum of $M$, its $G$-invariant
average will remain a section of $TP^\perp$, but will also be tangent to
$P$. Hence, it will be the zero section. For an analogous statement, see
\cite{CuSn01}, Lemma 2.4. 

In the same manner, define for $\alpha\in\Omega^1(M)$ the $G$-invariant
average $\alpha_G\in \Omega^1(M)^G$ of $\alpha$ as follows:
\begin{align*}
\alpha_G([g,b]_H)=\left(\Phi_{g^{-1}}^*
\left(\int_H\Phi_h^*\alpha dh\right)\right)([g,b]_H);
\end{align*}
that is, for each point $m'=[g,b]_H\in U$ we have
\begin{align}
\label{alpha_G}
{\alpha_G}([g,b]_H)&=\left(\int_H\Phi_h^*\alpha dh\right)_{[e,b]_H}\circ
T_{[g,b]_H}\Phi_{g^{-1}} \nonumber \\
&=\left(\int_H(\alpha([h,b]_H)\circ T_{[e,b]_H}\Phi_h) dh\right)\circ
T_{[g,b]_H}\Phi_{g^{-1}}.
\end{align}
In an analogous manner as above, we can show that $\alpha_G$ is well-defined, 
smooth, and
$G$-invariant. In the following, the one-form $\int_H\Phi_h^*\alpha dh$
 will be called $\alpha^H$.

If $(X,\alpha)$ is a section of a $G$-invariant generalized
distribution $\D$, the section $(X_G,\alpha_G)$ is a $G$-invariant section
of $\D$. 

Note that, in the same manner, we can define the $G$-invariant average $f_G$
of a smooth 
function $f$ defined on the tube $U$ of the action of $G$ at $m$. The function 
$f_G$ is defined by
\[f_G([g,b]_H)=\int_{h\in H}f([h,b]_H)dh.\]
Again, it is easy to check that $f_G$ is well-defined. The smoothness of $f_G$
can be shown with similar arguments as for the smoothness of $X_G$. 

\medskip

Let $P $ be a connected component of an orbit type manifold (recall \eqref{P(H)}) 
and $\bar P: = \pi(P) $,
where $\pi: M \rightarrow M/G 
= : \bar{M}$ is the orbit space projection. 
Since $G$ is
connected, the subgroup $G^P$ 
of $G$ such that $\Phi_g(P)\subseteq P$ for all $g\in G^P$ is equal to
$G$. Hence the proper action of 
$G$ on $M$ restricts to a proper action  $\Phi^P$ of $G$ on $P$ satisfying 
$\iota_P\circ \Phi^P_g=\Phi_g\circ\iota_P$ for all $g\in G$. Moreover, the
action of 
$G$ on $P$ has conjugated isotropy subgroups and thus the quotient $P/G$ is a
smooth manifold. 
Let $\pi_P$ be the quotient map.
Using the previous discussion, we can relate the
differential structures on $\bar{P}$, seen as  the quotient manifold of $P$ by 
the smooth and
proper $G$-action, and as a stratum of the stratified space $\bar M$.
\begin{proposition}\label{equ_of_smooth_structures}
Let $P$ be a connected component of an orbit type manifold $M_{(H)}$.  
The quotient $P/G$ is diffeomorphic  to the
stratum $\pi(P)=\bar P$ 
of $\bar M$. 
\end{proposition}

\begin{proof}
The bijectivity of the well-defined map $\Lambda:P/G\to\bar P$,
$\pi_P(p)\mapsto \pi(\iota_P(p))$ 
is easy. Note that we have $\pi\circ\iota_P=\iota_\bp\circ\Lambda\circ\pi_P$.

Let $f_\bp\in C^\infty(\bar P)$ and $\bar p\in\bar P$ in the domain of
definition of $f_\bp$. We have to find a neighborhood $U_\bp\subseteq \bar P$
of $\bar p$  such that
$\Lambda^*(f_\bp\an{U_\bp})\in C^\infty(P/G)$. Since  $f_\bp\in C^\infty(\bar
P)$, there exists a neighborhood $U\subset M$ of $\bar p$ and $\bar f\in
C^\infty(\bar M)$ such that $f_\bp\an{U_\bp}=\bar f\circ
\iota_\bp\an{U_{\bp}}$. Assume without loss of generality that $U_\bp=U\cap
\bp$. Since $\bar f$ is a smooth function on $\bar M$, there exists $f\in
C^\infty(M)^G$ 
such that $f=\pi^*(\bar f\,)$.  But then we have 
\[\pi_P^*(\Lambda^*(f_\bp\an{U_\bp}))=(\pi_P^*\circ\Lambda^*\circ\iota_\bp^*)(\bar
f)
=( \iota_P^*\circ\pi^*)(\bar f\,)=\iota_P^*(f)\in C^\infty(P),\]
and hence $\Lambda^*(f_\bp\an{U_\bp})\in C^\infty(P/G)$.

Let $f_{P/G}\in C^\infty(P/G)$. We have to show that
$(\Lambda^{-1})^*(f_{P/G})$ is an 
element of $C^\infty(\bar P)$. Define $f_P : = \pi_P^*(f_{P/G}) \in
C^\infty(P)^G$ and 
extend it to a function $f\in C^\infty(M)$, that is,
$\iota_P^*(f)=f_P$. Without loss 
of generality we can assume that $f$ is $G$-invariant (otherwise, the
$G$-invariant 
average of $f$ will also pull back to $f_{P}$), and thus pushes forward 
to $\bar f\in C^\infty(\bar M)$.  
Then we have
\[
(\pi_P^*\circ\Lambda^*\circ\iota_\bp^*)(\bar f\,)
=(\iota_P^*\circ\pi^*)(\bar f\,)=f_P=\pi_P^*(f_{P/G});
\]
hence 
\[(\Lambda^*\circ\iota_\bp^*)(\bar f\,)=f_{P/G}
\]
since $\pi_P$ is a smooth surjective submersion.
From this follows 
\[ (\Lambda^{-1})^*(f_{P/G}) =\iota_\bp^*(\bar f\,),\]
which is an element of $C^\infty(\bar P)$.
\end{proof}
Thus, in the following, we will identify $\bar P$ and $P/G$ without further 
mentioning it.

\subsection{Push-forward of vector fields and one-forms}\label{pw_of_vf}
Consider a $G$-invariant local vector field $X$ on $M$. Since $X$ is $
G$-invariant, the  push-forward $\bar X:=\pi_{\ast }X$, defined by 
$\pi ^{\ast }((\pi _{\ast }X) (\bar f))=X(\pi ^{\ast }(\bar f))$ 
for every $\bar f\in C^{\infty }(\bar M)$,  is a well-defined (local) derivation
of $C^{\infty }(\bar M)$. Moreover, $X$ generates a
local one-parameter group $\phi_t^X$ of local diffeomorphisms of $M$. Since $X
$ is $G$-invariant, $\phi_t^X$ commutes with the action of $G$ on $M$, and it
induces a local one-parameter group of local diffeomorphisms of $\bar M$
generated by $\pi _{\ast }X$. Hence, $\pi_{\ast }X$ is a (local) vector field on $
\bar M$. 

We write $\mx(\bar M)$ for the sheaf of (local) vector fields on $\bar M$. Then we have 
\begin{equation}\label{duistermaat}
\pi_*\left(\mx(M)^G\right)=\mx(\bar M)
\end{equation}
(see \cite{Duistermaat}, Theorem 6.10).
In particular, for each stratum of $\bar M$ the tangent bundle space of the
stratum is spanned by push-forwards by $\pi $ of $G$-invariant vector
fields on $M$. It is easy to see that the sheaf of local vector fields on
$\bar P$ is the set of local restrictions to $\bar P$ of elements of $\bar
M$. Also, Proposition \ref{restriction_of_X} is also true for local vector fields. 

Yet, the class of vector fields on $M$ that push forward to vector fields on
$\bar M$ is bigger than the class of $G$-invariant vector fields, as  the
next lemma shows.
\begin{lemma}\label{lem:push_down_der}
If $X\in\mx(M)$ is such that $[X,\Gamma(\V)]\subseteq\Gamma(\V)$, 
then it defines a derivation of the
ring $C^\infty(M)^G$ of $G$-invariant functions. Therefore, it pushes down to a
derivation $\bar X$ of $C^\infty(\bar M)$.
The derivation $\bar X$ is a vector field on the subcartesian space $\bar M$.
\end{lemma}

\begin{proof}
Let $f\in C^\infty(M)^G$ and $g\in G$. Since $[X,V]\in\Gamma(\V)$ for all
sections $V\in\Gamma(\V)$, we have, in particular,  $[X,\xi_M]=V_{\xi}\in\Gamma(\V)$
for each $\xi\in\lie g$ and thus:
\[\xi_M(X(f))=X(\xi_M(f))-V_{\xi}(f)=0,\]
since $V(f)=0$ for all $V\in\Gamma(\V)$. 
We get for all $m\in M$:
\[\left.\frac{d}{dt}\right\an{t=0}X(f)\circ\Phi_{\exp(t\xi)}(m)=0\]
and hence, for all $t\in \R$:
\begin{align*}
\frac{d}{dt}X(f)\circ\Phi_{\exp(t\xi)}(m)
&= \left.\frac{d}{ds}\right\an{s=0}X(f)\circ\Phi_{\exp((t+s)\xi)}(m)\\
&=\left.\frac{d}{ds}\right\an{s=0}X(f)
\circ\Phi_{\exp(s\xi)}\left(\Phi_{\exp(t\xi)}(m)\right)\\
&=0.
\end{align*}
Since the Lie group $G$ is connected, it is spanned as a group by every
neighborhood of its neutral element, hence from the image of the exponential
map. With this it follows that the function $X(f)$ is $G$-invariant.
Hence $X$ defines a derivation of
$C^\infty(M)^G$, and hence it induces a derivation $\bar X$ of the ring 
$C^\infty(\bar M)$ of smooth
functions on $\bar M$ as follows:
\[\pi^*(\bar X(\bar f\, )):=X(\pi^*(\bar f\, ))\]
for all $\bar f\in C^\infty(\bar M)$.

We have to show that $\bar X$ is a vector field on the quotient space $\bar
M$. Let $\bar \phi_t$ be the flow of $\bar X$ and let $\bar f\in C^\infty(\bar
M)$, i.e., $\pi^*(\bar f\, )\in C^\infty(M)$. We have to show that $\bar
\phi_t^*(\bar f\, )\in C^\infty(\bar M)$. From the definition of $\bar X$ follows
the equality 
\[\bar \phi_t\circ\pi=\pi\circ \phi_t,\]
where $\phi_t$ is the flow of $X$. 
Thus we have
\[\pi^*\left(\bar \phi_t^*(\bar f\, )\right)=\phi_t^*\left(\pi^*(\bar
  f\, )\right).\]
Since $X$ is a vector field on $M$, the function $\phi_t^*\left(\pi^*(\bar
  f\, )\right)$ is an element of $C^\infty(M)$, and hence $\bar
  \phi_t^*(\bar f\, )$ an element of $C^\infty(\bar M)$.
\end{proof}
Let $X$ be as in the last lemma, and let $\bar X$ be the vector field on $\bar M$
with $X\sim_\pi\bar X$. Since $\bar X$ is a vector field, there exists a
$G$-invariant vector field $X^G\in\mx(M)^G$ with $X^G\sim_\pi \bar X$ (see
\eqref{duistermaat}). 
Thus,
the vector field $X$ can be written as a sum $X=X^G+X^\V$, with $X^\V$ a section of
$\V$ (note that $X^G$ is in general
not equal to the $G$-invariant average $X_G$ of $X$). 
\medskip 

Let $\alpha $ be a (local) $G$-invariant one-form on $M$ annihilating vectors 
tangent to
orbits of the action of $G$ on $M$. For each $G$-invariant vector field $X$
on $M$, the evaluation $\alpha(X)$ is $G$-invariant.
Hence, there exists a smooth function $\pi _{\ast }(\alpha(X)) $ defined on
$\bar M$ by 
$\pi^*(\pi_\ast(\alpha(X)))=\alpha(X)$. Since $\alpha $ annihilates 
vectors tangent to orbits
of $G$, it follows that $\pi _{\ast }(\alpha(X)) $ depends
on $X$ through its push-forwards $\pi _{\ast }X$. In other words, there is a
linear form $\pi _{\ast }\alpha $ on the space of push-forwards by $\pi $
of $G$-invariant vector fields on $M$ such that 
\[
(\pi _{\ast }\alpha)(\pi _{\ast }X)=\pi _{\ast }(\alpha(X))
\quad \text{ for all } X\in\mx(M)^G. 
\]
Moreover, for every $\bar f\in C^{\infty }(\bar M)$, 
\begin{align*}
(\pi _{\ast }\alpha)(\bar f\pi _{\ast }X) &=(\pi _{\ast
}\alpha)(\pi _{\ast }(\pi ^{\ast }(\bar f)X)) =\pi _{\ast }(\alpha (\pi^*(\bar f)X))\\
&=\pi _{\ast }(\pi^*(\bar f))\,\pi _{\ast }(\alpha(X))
=\bar f(\pi _{\ast }\alpha)( \pi _{\ast }X), 
\end{align*}
that is, $\pi_*\alpha$ is $C^\infty(\bar M)$-linear. 
This implies that, for every stratum of $\bar M$, the restriction of $\pi
_{\ast }\alpha $ to the stratum gives rise to a well-defined one-form on the
stratum.

\begin{definition}
Let $G$ be a Lie group acting properly on the manifold $M$.  Let $\V$ be the
vertical space of the action.
A section $(X,\alpha)$ in $\Gamma(TM\operp\V^\circ)$ satisfying
$[X,\Gamma(\V)]\subseteq \Gamma(\V)$ and $\alpha\in\Gamma(\V^\circ)^G$ will be
called a \emph{descending section}.
\end{definition}

\medskip

We will  also need a few more facts about \emph{one-forms} of $\bar
M$. 
Indeed, we have a notion of vector fields on $\bar M$, and we know that these are exactly
the push-forwards of descending vector fields on $M$.
We also want to introduce objects which will play the role of  \emph{one-forms} on 
$\bar M$. The
definition of a one-form
on $\bar M$ should be such that each element $\alpha_\bp\in\Omega^1(\bar P)$,
where $\bar P$ is a stratum of $\bar M$,
is the restriction to $\bar P$ of a one-form on $\bar M$. Thus, we could
define a one-form as a $C^\infty(\bar M)$ linear map $\mx(\bar M)\to
C^\infty(\bar M)$, but since we want a one-to-one correspondence between
sections  $\Gamma(\V^\circ)^G$ and one-forms on $\bar M$, we need to define
these more carefully.

By the space of \emph{K\"ahler differentials of $C^\infty(\bar M)$ over $\R$}
one understands a $C^\infty(\bar M)$-module $\Omega_{C^\infty(\bar M)/\R}$
together with a derivation $\dr: C^\infty(\bar M)\to \Omega_{C^\infty(\bar
  M)/\R}$ called the \emph{K\"ahler derivative} such that the following universal
property is satisfied (see \cite{Pflaum01}):

\medskip
\noindent \emph{For every $C^\infty(\bar M)$-module $\mathcal{M}$ and every derivation
$\delta:C^\infty(\bar M)\to \mathcal{M}$ there exists a unique $\R$-linear
mapping $\mathbf{i}_\delta:\Omega_{C^\infty(\bar M)/\R}\to\mathcal{M} $
such that the diagram 
\[
\begin{xy}
\xymatrix{
C^\infty(\bar M)\ar[d]_{\dr}\ar[r]^\delta&\mathcal{M}\\
\Omega_{C^\infty(\bar M)/\R}\ar[ur]_{\mathbf{i}_\delta}&\\
}
\end{xy}\]
commutes}. 

\medskip
In particular, if $\mathcal{M}=C^\infty(\bar M)$ and
$\delta$ is a vector field $\bar X$ on $\bar M$, we get $\bar X(\bar f)=\ip{\bar
  X}\dr \bar f$ for each function $\bar f\in C^\infty(\bar M)$ and $\ip{\bar
  X}$ is the inner product with $\bar X$. 
\begin{proposition} [\cite{Matsumura80}]\label{matsu}
The space $\Omega_{C^\infty(\bar
  M)/\R}$ exists and can be represented as follows.
Let $\Omega$ be the free $ C^\infty(\bar M)$-module over the symbols
  $\dr \bar f$ with $\bar f\in C^\infty(\bar M)$, and $\mathcal{J}$ the
  $C^\infty(\bar M)$-submodule generated by the relations
\begin{align*}
\dr(\lambda \bar f+\mu\bar g)-\lambda\dr \bar f-\mu\dr \bar g=0 &\quad \text{ for
  all } \quad
\lambda,\mu\in\R,\bar f,\bar g\in C^\infty(\bar M), \\
\dr(\bar f\bar g)-\bar f\dr\bar g-\bar g \dr\bar f=0 &\quad\text{ for
  all } \quad\bar f,\bar g\in C^\infty(\bar M).
\end{align*} 
Then $\Omega_{C^\infty(\bar M)/\R}=\Omega/\mathcal{J}$ and
$\dr:C^\infty(\bar M)\to\Omega_{C^\infty(\bar M)/\R}$ is defined by  $\bar f\mapsto
\dr\bar f+\mathcal{J}$. 
\end{proposition}
From this it follows immediately that each element of $\Omega_{C^\infty(\bar
  M)/\R}$ can be written as a sum $\sum_j\bar g_j\dr \bar f_j$ with finitely
many $\bar g_j, \bar f_j\in C^\infty(\bar M)$. 

Hence, let $\bar \alpha=\sum_{j=1}^k\bar
g_j\dr \bar f_j\in \Omega_{C^\infty(\bar M)/\R}$ and set
$\alpha=\sum_{j=1}^n\pi^*\bar
g_j\dr(\pi^* \bar f_j)\in \Gamma(\V^\circ)$. We then have for each $G$-invariant
vector field $X$ on $M$:
\begin{align*}
\pi^*((\pi_*\alpha)(\pi_*X))&=\alpha(X)=\sum_{j=1}^n\pi^*\bar
g_j\dr(\pi^* \bar f_j)(X)
=\sum_{j=1}^n\pi^*\bar
g_j X(\pi^* \bar f_j)\\
&=\pi^*\left(\sum_{j=1}^n\bar g_j \,(\pi_*X)(\bar f_j)\right)
=\pi^*\left(\sum_{j=1}^n\bar g_j \ip{\pi_*X}\dr\bar f_j\right).
\end{align*} 
Hence, the $C^\infty(\bar M)$-linear map $\pi_*\alpha:\mx(\bar M)\to
C^\infty(\bar M)$ corresponds exactly to the $C^\infty(\bar M)$-linear map 
$\mx(\bar M)\to
C^\infty(\bar M)$
defined by $\bar \alpha$ as follows: \[\bar \alpha(\bar X):=\sum_{j=1}^k\bar
g_j\,\ip{\bar X}\dr \bar f_j\text{ for all } \bar X\in\mx(\bar M).\]
We set $\alpha=:\pi^*\bar\alpha$.
Thus, each K\"ahler differential on $C^\infty(\bar M)$ can be realized as the
push-forward of an element of $\Gamma(\V^\circ)^G$.
Conversely, we will see later that each element $\alpha \in \Gamma(\V^\circ)^G$ 
can be written as a sum
$\alpha=\sum_{j=1}^kg_j\dr f_j$ with $g_j,f_j\in C^\infty(M)^G$ (see  Lemma
\ref{description_V_G}) 
and thus
pushes forward to the K\"ahler differential $\sum_{j=1}^k(\pi_*g_j)\dr(\pi_* f_j)$.
An element $\bar \alpha\in \Omega_{C^\infty(\bar M)/\R}$  will be called a
\emph{one-form} on $\bar M$ and the set of one-forms on $\bar M$ will  be denoted by
$\Omega^1(\bar M)$. 
We have shown the following proposition.
\begin{proposition}\label{corres-one-forms}
The one-forms on $\bar M$ correspond  exactly to the push-forwards of elements
of $\Gamma(\V^\circ)^G$.
\end{proposition}
Note that not every smooth section of the stratified cotangent
space on $\bar{M}$, i.e., a smooth $C^\infty(\bar{M})$-linear map 
$\mx(\bar M)\to C^\infty(\bar M)$, can be realized as a
one-form on $\bar M$ (see \cite{Pflaum01} for the definition and discussion). 
There is a nontrivial condition for this to hold; see
Proposition 2.3.7 in \cite{Pflaum01}. Hence, since each element of
$\Gamma(\V^\circ)^G$ pushes forward to a one-form on $\bar M$, there should be 
smooth $C^\infty(\bar
M)$-linear maps $\mx(\bar M)\to C^\infty(\bar M)$ which cannot be realized
as  push-forwards of  elements of $\Gamma(\V^\circ)^G$.

Note that for vector fields, we have the analogous fact that each vector field
on $\bar M$ is the push-forward of a $G$-invariant vector field on $M$ (see
\eqref{duistermaat}), but that not \emph{all} derivations on $\bar M$ are
vector fields on $\bar M$.

\subsection{Connected components of the orbit types}
Let $F^G$ be the everywhere defined family of local vector fields
\begin{align*}
F^G=\{X\in\mx(M)^G\mid X=X^G+X^\V \text{ with } X^G\in\mx(M)^G \text{ and } X^\V
\in\Gamma(\V)\},
\end{align*}
$\mathcal{A}^G:=\{\phi_t\mid X\in F^G, \, \phi_t\text{ flow of } X\}$, and 
 denote by $A^G$ the
pseudogroup of local diffeomorphisms associated to the flows of the 
family $F^G$, i.e., 
\[
A^G=\{\mathds{I}\}\cup\{\phi_{t_1}^1\circ\dots\circ \phi_{t_n}^n\mid n\in \mathds{N}
\text{ and } \phi_{t_n}^n \text{ or } ( \phi_{t_n}^n)^{-1}\text{ flow of } X^n\in
F^G\}.
\]
Let $\T$ be the smooth generalized distribution spanned by $F^G$, that is, 
\[\T(m)=\erz\{X(m)\mid X\in F^G,\, m\in \dom(X)\}.\]
Note that  with Lemma \ref{lem:push_down_der} and the considerations
following its proof, $F^G$ is equal to 
\[\{X\in\mx(M)\mid [X,\Gamma(\V)]\subseteq\Gamma(\V)\}.\]

We will show that the distribution $\T$ is integrable in the sense of
Stefan-Sussman and compare its leaves with the connected components of the
orbit type manifolds.

\begin{lemma}
For each $\mathcal{F}\in A^G$ and for each $m\in \dom(\mathcal{F})\subseteq
M$, we have 
\[T_m\mathcal{F}(\T(m))=\T(\mathcal{F}(m)).\]
As a consequence, the distribution $\T$ is integrable in the sense of
Stefan-Sussman and its leaves are the $A^G$-orbits.
\end{lemma}

\begin{proof}
Assume first that $\mathcal{F}=\phi_t^X\in A^G$ for one vector field $X\in\mx(M)$
satisfying $[X,\Gamma(\V)]\subseteq \Gamma(\V)$ (the general statement will
follow inductively, since each element of $A^G$ is a composition
of finitely many such diffeomorphisms). Write $X$ as a sum $X^G+X^\V$
with $X^G\in\mx(M)^G$ and $X^\V\in\Gamma(\V)$.  Let $v\in\T(m)$. Then
$v=Y(m)=Y^G(m)+Y^\V(m)$ for sections $Y^G\in\mx(M)^G$ and
$Y^\V\in\Gamma(\V)$. 
By the Trotter Product Formula (see, for example, \cite{OrRa04}), the flows
$\phi^X$ and $\phi^Y$ of the vector fields $X$ and $Y$ are given by 
\[\phi^X_t=\underset{n\to\infty}\lim\left(\phi^{X^G}_{t/n}
\circ \phi^{X^\V}_{t/n}\right)^n\quad
\text{ and }\quad
\phi^Y_t=\underset{n\to\infty}\lim\left(\phi^{Y^G}_{t/n}
\circ \phi^{Y^\V}_{t/n}\right)^n,\]
where $\phi^{X^G}$, $\phi^{X^\V}$, $\phi^{Y^G}$, and  $\phi^{Y^\V}$ are 
the flows of the
vector fields $X^G$, $X^\V$, $Y^G$, and $Y^\V$. But since $X^G$ and $Y^G$ are
$G$-invariant and $X^\V$ and $Y^\V$ are sections of $\V$, the flows of the
vector fields $X^G$ and $Y^G$ commute with the flows of $X^\V$ and
$Y^\V$. Hence, we get
\[\phi^X_t=\phi^{X^G}_t\circ \phi^{X^\V}_t=\phi^{X^\V}_t\circ \phi^{X^G}_t\quad
\text{ and }\quad \phi^Y_t=\phi^{Y^G}_{t}\circ \phi^{Y^\V}_{t}
=\phi^{Y^\V}_{t}\circ \phi^{Y^G}_{t}.\]
The compositions \[\phi_s:=\phi_t^X\circ \phi_s^Y\circ \phi_{-t}^X,\quad\quad
\phi_s^G:=\phi_t^{X^G}\circ
\phi_s^{Y^G}\circ \phi_{-t}^{X^G}\] and \[\phi_s^\V:=\phi_t^{X^\V}\circ
\phi_s^{Y^\V}\circ \phi_{-t}^{X^\V}\] define flows on $M$. Let $Z$, $Z^G$ and $Z^\V$
be the vector fields associated to those flows.
We then have $Z^G\in\mx(M)^G$, $Z^\V\in\Gamma(\V)$ and 
\begin{align*}
\phi_s&=\phi_t^X\circ \phi_s^Y\circ \phi_{-t}^X
=\phi^{X^G}_t\circ \phi^{X^\V}_t\circ \phi^{Y^G}_s\circ \phi^{Y^\V}_s\circ
\phi^{X^G}_{-t}\circ \phi^{X^\V}_{-t}\\
&=\left(\phi^{X^G}_t\circ \phi^{Y^G}_s
\circ \phi^{X^G}_{-t}\right)\circ\left(\phi^{X^\V}_t\circ
\phi^{Y^\V}_s\circ \phi^{X^\V}_{-t}\right)\\
&=\left(\phi^{X^\V}_t\circ
\phi^{Y^\V}_s\circ \phi^{X^\V}_{-t}\right)
\circ\left(\phi^{X^G}_t\circ \phi^{Y^G}_s\circ \phi^{X^G}_{-t}\right).
\end{align*}
The vector field $Z$ is then equal to the sum $Z^G+Z^\V$ and it satisfies 
$[Z,\Gamma(\V)]\subseteq \Gamma(\V)$. The equality
\begin{align*}
T_m\phi_t^X(Y(m))&= \left.\frac{d}{ds}\right\an{s=0}\phi_t^X\circ \phi_s^Y(m)=
\left.\frac{d}{ds}\right\an{s=0}\phi_t^X\circ \phi_s^Y\circ \phi_{-t}^X(\phi^X_t(m))\\
&=\left.\frac{d}{ds}\right\an{s=0}\phi_s(\phi^X_t(m))=Z(\phi^X_t(m))
\end{align*}
then yields the first inclusion $T_m\phi_t^X(\T(m))\in\T(\phi^X_t(m))$.

For the other inclusion, we use a similar method: let $Y=Y^G+Y^\V$ be a vector
field satisfying $[Y,\Gamma(\V)]\subseteq \Gamma(\V)$ and defined on a
neighborhood of $\phi_t^X(m)$. As above, the vector field $Z$ corresponding to
the flow $\phi_s:=\phi_{-t}^X\circ \phi_s^Y\circ \phi_{t}^X$ can be written as a sum
$Z=Z^G+Z^\V$ and is hence a section of $\T$.  We get
\begin{align*}
Y(\phi^X_t(m))&=\left.\frac{d}{ds}\right\an{s=0}\phi^Y_s(\phi^X_t(m))
=\left.\frac{d}{ds}\right\an{s=0}\left(\phi_{t}^X
\circ \phi_{-t}^X\circ \phi_s^Y\circ \phi_t^X\right)(m)\\
&=\left.\frac{d}{ds}\right\an{s=0}\left(\phi_{t}^X
\circ \phi_s\right)(m)=T_m\phi_{t}^X(Z(m))\in T_m\phi_t(\T(m)).
\end{align*}
\end{proof}

\begin{theorem}\label{leaves_of_T}
The integrable leaves of the distribution $\T$ are exactly the connected
components  of the orbit type manifolds.
\end{theorem}

\begin{proof}
Let $N$ be the $A^G$-orbit through the point $m\in M$. We have to show that
$N=P$, where $P$ is the connected component through $m$ of the orbit type manifold
$M_{(G_m)}$.  Let $m'\in N$. Then there exist vector fields
$X_1,\ldots,X_l\in\mx(M)^G$ and $V_1,\ldots,V_l\in\Gamma(\V)$ such that
$$m'=\left(\phi^{V_1}_{t_1}\circ\ldots\circ \phi^{V_k}_{t_k}\circ
\phi^{X_1}_{s_1}\circ\ldots\circ \phi^{X_l}_{s_l}\right)(m)$$ (recall  that the flows of the
$G$-invariant vector fields commute with the flows of the sections of $\V$).
Hence, we can assume without loss of generality that $m'=\phi^X_t(m)$ with a
vector field $X\in\mx(M)^G$ or $m'=\phi^V_t(m)$ with $V$ a section of $\V$. 
In the first case, the vector field $X$ pushes down to a vector field $\bar X$
on $\bar M$. Let $\phi^{\bar X}_t$ be the flow of the vector field $\bar
X$. Since the strata of $\bar M$ are the accessible sets of the vector fields
on $\bar M$, the points $\pi(m')=\left(\pi\circ \phi^X_t\right)(m)=\phi^{\bar X}_t(\pi(m))$ and 
$\pi(m)$
lie in the same stratum of $\bar M$, hence in $\pi(P)$. Since $X$ is
$G$-invariant, 
its flow is also $G$-equivariant. Thus, we have
$\phi_t^X(gm)=g\cdot\phi_t^X(gm)$ 
for all $t$ where it is defined, and hence $\phi_t^X(m)\in M_{(G_m)}$ for all
$t$. 
This yields that  $m$ and
$\phi^X_t(m)$ can be joined by a smooth path in $M_{(G_m)}$, and consequently
that 
they are in the same connected
component of $M_{(G_m)}$, that is, the connected component
$P$. But since $\pi(m')=\pi(\phi^X_t(m))$, there exists $g\in G$ such that
$m'=\Phi_g(\phi^X_t(m))$ and since the action of $G$ on $M$ restricts to $P$, the
point $m'$ is also an element of $P$. 
In the second case we have $m'=\phi_t^V(m)$ with $V\in \Gamma(\V)$. Since the
vector field $V$ is tangent to the $G$-orbits, its integral curve through $m$
lies entirely in the connected component of the orbit type manifold through $m$ and we
are finished.

For the other inclusion, let $m'$ be a point of $P$. We then have $\pi(m)$ and
$\pi(m')\in \pi(P)$, a stratum of $\bar M$. Thus, we can assume without loss
of generality that $\pi(m')=\phi_t^{\bar X}(\pi(m))$ for some vector field $ \bar
X\in\mx(\bar M)$ (in reality, $\pi(m)$ and $\pi(m')$ can be joined by finitely
many such curves).  Let $X\in\mx(M)^G$  be such that $X\sim_\pi \bar X$ and let
$\phi_t^X$ be its flow. Then we have $\left(\pi\circ \phi_t^X\right)(m)=\phi_t^{\bar
  X}(\pi(m))=\pi(m')$. Thus, there exists $g\in G$ such that
$\Phi_g(\phi_t^X(m))=m'$. But since $G$ is connected, we find finitely many
elements $\xi^1,\ldots,\xi^l\in\lie g$ such that
$g=\exp(\xi^1)\cdot\ldots\cdot\exp(\xi^l)$, and hence we have
\[m'=\left(\Phi_{\exp(\xi^1)}\circ\ldots\circ\Phi_{\exp(\xi^l)}\circ \phi_t^X\right)(m).\]
The curves $\Phi_{\exp(t\xi^i)}:[0,1]\to M$, $i=1,\ldots,l$,  are segments of integral
curves of the sections $\xi^i_M$ of $\V$, and
$\Phi_{\exp(\xi^1)}\circ\ldots\circ\Phi_{\exp(\xi^l)}\circ \phi_t^X$ is
consequently an element of $A^G$. From this it follows that $m'\in N$. 
\end{proof}

An alternative proof of this theorem is based on Theorem 3.5.1 (stating that
the distribution $\T_G$ spanned by $G$-invariant vector fields is integrable
with leaves the connected components of the isotropy type manifolds), Proposition
3.4.6 in  \cite{OrRa04}, 
and the fact that $G\cdot M_H^m=M_{(H)}^m$, where $ M_H^m $ 
(respectively $ M_{(H)}^m $) is the connected 
component of $ M_H$ (respectively $ M_{(H)}$) containing $m $.

\medskip

Let $P$ be a stratum of $M$, that is, a leaf of the distribution $\T$.
Since $M$ is paracompact, there exists a $G$-invariant Riemannian metric
$\rho$ on $M$ (see, for example, \cite{DuKo00}). Consider the vector bundle 
$TP=\mathcal{T}\an{P}\subseteq
TM\an{P}$ over $P$, and let $TP^\perp$ be a $G$-invariant orthogonal
complement of $TP$ viewed as a subbundle of  $TM\an{P}$. We can describe 
the codistribution
$TP^\circ$ in the following way:
\[TP^\circ(p)=\{\ip{X}\rho(p)\mid X\in\Gamma(TP^\perp)\}\]
for all $p\in P$. Note that the Riemannian metric $\rho$ allows an
identification of the tangent bundle of $M$ with the cotangent bundle via 
\[X\in\mx(M)\leftrightarrow \ip{X}\rho\in\Omega^1(M).\]
The section $\ip{X}\rho$ is $G$-invariant if and only if $X$ is $G$-invariant.
We will use this in the proof of many of the following propositions and lemmas.

\medskip

In the following, we will make use of the codistribution $\V^\circ_G$ defined
as the span of the $G$-invariant sections of $\V^\circ$:
\[\V^\circ_G(x)=\{\alpha_x\mid \alpha\in\Gamma(\V^\circ)^G\}\]
for all $x\in M$. 

This codistribution is in fact spanned by the exterior derivative of all
$G$-invariant functions on $M$, 
as stated in the following lemma.
\begin{lemma}\label{description_V_G}
The codistribution $\V^\circ_G$ can be described as follows: for each $m\in M$ we have
\[\V^\circ_G(m)=\erz\{\dr f(m)\mid f\in C^\infty(M)^G\}.\]
\end{lemma}
\begin{proof}
We use the identity $\left((T_mGm)^{\ann}\right)^{G_m}=\erz\{\dr
f(m)\mid f\in C^\infty(M)^G\}$ (see \cite{OrRa04}, Theorem 2.5.10), where 
\[(T_mGm)^{\ann} :  
= \{ \alpha_m \in T ^\ast_mM\mid \alpha_m( \xi_M(m)) = 0\; \text{for all} \;
\xi\in \mathfrak{g}\}\] is the pointwise annihilator of the tangent space $T
_mGm $ to the orbit $Gm $. We show 
\[\erz\{\dr f(m)\mid f\in C^\infty(M)^G\}\subseteq
\V^\circ_G(m)\subseteq\left((T_mGm)^{ \ann}\right)^{G_m},\]
and our claim follows from the equality above. The first inclusion is easy
since for each function $f\in C^\infty(M)^G$, we have $\dr f\in\Gamma(\V^\circ)^G$.
For the second inclusion, choose $\alpha(m)\in\V^\circ_G(m)$, with $\alpha$ a
$G$-invariant section of $\V^\circ$. Then we have $\alpha(\xi_M)=0$ for all
$\xi\in\lie g$ and hence $\alpha(m)(\xi_M(m))=0$ for all $\xi\in\lie g$, that
is, $\alpha(m)\in(T_mGm)^{\ann}$. Since $\alpha$ is
$G$-invariant, we have $\Phi_h^*\alpha=\alpha$ for all $h\in G_m\subseteq G$
and hence, for all $v\in T_mM$ we get
\[\alpha_m(T_m\Phi_hv)=\alpha_{\Phi_h(m)}(T_m\Phi_hv)
=(\Phi_h^*\alpha)_m(v)=\alpha_m(v),\]
where we have used that $h\cdot m=m$ since $h\in G_m$. Hence we have
$(T_m\Phi_h)^*(\alpha(m))=\alpha(m)$ for all $h\in G_m$ and hence 
$\alpha(m)\in\left((T_mGm)^{ \operatorname{ann}}\right)^{G_m}$.
\end{proof}

Using this, we can show the following lemma.
\begin{lemma}\label{V_circ_G} 
Let $P$ be a stratum of $M$, and $\V_P$ the vertical space of the induced
action of $G$ on $P$.
We have the equality
\[\iota_P^*(\V^\circ_G)=(\V_P)^\circ\subseteq T^*P.\]
Hence, the map 
$\iota_P^*:\V^\circ_G\an{P}\to (\V_P)^\circ$
is an isomorphism. Thus, $\V^\circ_G\an{P}$ is a vector bundle over $P$ and 
\[\left(\V^\circ_G\an{P}\right)^\circ=\V\an{P}\oplus TP^\perp.\]
\end{lemma}

For the proof of this we will need the following lemma.
\begin{lemma}\label{Vzero}
If the action of a Lie group $G$ on a manifold $M$ is with conjugated isotropy
subgroups, then the  (smooth) annihilator $\V^\circ$ of the vertical bundle $\V$
is spanned by its $G$-invariant sections.
\end{lemma}
\begin{proof}
Since the action of $G$ on $M$ is with conjugated isotropy subgroups,
the vertical space $\V$ is a smooth integrable subbundle of $TM$. Thus,
for each $p\in M$, we find a coordinate neigborhood $U$ of $p$ with
coordinates $(x_1,\ldots,x_n)$ such that
$\V$ is spanned by $\partial_{x_1},\ldots,\partial_{x_k}$, where 
$k=\dim G-\dim G_p=\dim \V$. The annihilator $\V^\circ $
of $\V$ is then spanned by $\dr x_{k+1},\ldots, \dr x_n$ on $U$.
Since $\dr x_{k+1},\ldots, \dr x_n$ vanish on $\V$, and $G$ is connected,
the functions $x_{k+1},\ldots,x_n$ are then $G$-invariant and we get
$\dr x_{k+1},\ldots, \dr x_n\in\Gamma(\V^\circ)^G$.
\end{proof}

\begin{proof}[Proof of Lemma \ref{V_circ_G}]
The inclusion
$\iota_P^*(\V^\circ_G)\subseteq(\V_P)^\circ$ is easy. For the
other inclusion, note that since all isotropy type manifolds of the action of $G$ on
$P$ are conjugated, the
codistribution $(\V_P)^\circ$ is spanned by its $G$-invariant sections by
Lemma \ref{Vzero}. Therefore, by Lemma
\ref{description_V_G}, each $G$-invariant section of $(\V_P)^\circ$ is in the
$C^\infty(P)^G$-span of $\{\dr f\mid f\in C^\infty(P)^G\}$.  Hence, each element
$\tilde \alpha(p)\in(\V_P)^\circ(p)$ can be written as $\tilde \alpha(p)
=\sum_{j=1}^k\tilde f_j(p)\dr
\tilde g_j\an{p}$ with $\tilde f_j,\tilde g_j\in C^\infty(P)^G$. Choose
$f_j,g_j\in C^\infty(M)$ such that $\tilde f_j=\iota_P^*f_j$ and $\tilde
g_j=\iota_P^*g_j$ for $j=1,\ldots,k$. Without loss of generality, the
functions $f_1,\ldots,f_k,g_1,\ldots,g_k$ are $G$-invariant (otherwise, their
$G$-invariant averages will also restrict to $\tilde f_1,\ldots,\tilde
f_k,\tilde g_1,\ldots,\tilde g_k$).
Set $\alpha=\sum_{j=1}^kf_j\dr g_j\in\Gamma(\V^\circ)^G$. Then we have 
\begin{align*}
(\iota_P^*\alpha)(p)&=\left(\iota_P^*\left(\sum_{j=1}^kf_j\dr
    g_j\right)\right)(p)
=\sum_{j=1}^k(\iota_P^*f_j)(p)\dr(\iota_P^*g_j)\an{p}\\
&=\sum_{j=1}^k\tilde f_j(p)\dr
\tilde g_j\an{p}=\tilde\alpha(p),
\end{align*}
and the proof of the first assertion is complete, since we have shown that 
$\tilde\alpha(p)=(\iota_P^*\alpha)(p)\in(\iota^*_P(\V^\circ_G))(p)$.

\medskip

From this it follows that the map $\iota_P^*:\V^\circ_G\an{P}\to (\V_P)^\circ$ is
surjective. For the injectivity, let $\alpha\in\Gamma(\V^\circ)^G$ be defined on
a neighborhood of $p\in P$ and such that $\iota_P^*\alpha=0$. The vector
field $X\in\mx(M)$ satisfying $\ip{X}\rho=\alpha$ is $G$-invariant and hence
tangent to $P$ on $P$. Therefore, there exists $\tilde X\in\mx(P)$ with
$\tilde X\sim_{\iota_P}X$ and we have
$\iota_P^*\alpha=\iota_P^*\ip{X}\rho=\ip{\tilde X}\iota_P^*\rho$. But since
$\iota_P^*\alpha=0$ we get $\tilde X=0$ using the fact that $\iota_P^*\rho$ is a 
Riemannian metric on $P$. Hence, we have shown that $\alpha\an{P}=0$.

\medskip

It remains to show the equality
\[(\V_G^\circ\an{P})^\circ=\V\an{P}\oplus TP^\perp.\]
Since $\V\an{P}\subset TP\subset TM\an{P}$, we have $\V\an{P}\cap TP^\perp=0_P$.
First let $X\in\Gamma(TP^\perp)$, $V\in\Gamma(\V)$ and
$\alpha\in\Gamma(\V^\circ)^G$. Then we have $\alpha=\ip{Y}\rho$ with
$Y\in\mx(M)^G$ and hence 
\begin{align*}
\alpha\an{P}(X+V\an{P})&=\rho\an{P}(Y\an{P},X+V\an{P})
=\rho\an{P}(Y\an{P},X)+\rho\an{P}(Y\an{P},V\an{P})\\
&=\rho\an{P}(Y\an{P},X)+\alpha(V)\circ\iota_P=0,
\end{align*}
since $Y$ is tangent to $P$ on $P$, that is, $Y\an{P}\in\Gamma(TP)$.

Now choose $X\in\Gamma((\V_G^\circ\an{P})^\circ)\subseteq\Gamma(TM\an{P})$ and
write $X=X^\top+X^\perp$ with $X^\top\in\Gamma(TP)$ and
$X^\perp\in\Gamma(TP^\perp)$. Choose an arbitrary
$\alpha\in\Gamma(\V^\circ)^G$. Then $\alpha=\ip{Y}\rho$ with
$Y\in\mx(M)^G$. Again, $Y$ is tangent to $P$ on $P$ and we compute
\begin{align*}
0&=\alpha\an{P}(X)=\alpha\an{P}(X^\top+X^\perp)=\rho\an{P}(Y\an{P},X^\top+X^\perp)\\
&=\rho\an{P}(Y\an{P},X^\top)+\rho\an{P}(Y\an{P},X^\perp)=\rho\an{P}(Y\an{P},X^\top).
\end{align*}
Thus, we have $X^\top\in\Gamma((\V_G^\circ\an{P})^\circ)$.
 Since $X^\top\in\Gamma(TP)$, there exists $X\in\mx(M)$ with $X\an{P}=X^\top$ and
 $\tilde X\in\mx(P)$ with $\tilde X\sim_{\iota_P}X$.  For each section
 $\tilde\alpha\in\Gamma(\V_{P}^\circ)=\iota_P^*(\Gamma(\V^\circ_G))$, we
 have $\tilde\alpha=\iota_P^*\alpha$ with $\alpha\in\Gamma(\V^\circ_G)$ and 
\[\tilde\alpha(\tilde X)=\alpha(X)\circ\iota_P=0.\]
But then $\tilde X\in\Gamma(\V_P)$, which leads to $X^\top\in\Gamma(\V\an{P})$.
\end{proof}

With analogous methods as in the proof of the first part of the last lemma, we
can show the following proposition.
\begin{proposition}\label{one-forms}
Each local one-form on $\bar P$ is the restriction to $\bar P$ 
of a local one-form on $\bar M$.
\end{proposition}

\begin{proof}
Let $\alpha_\bp\in\Omega^1(\bp)$ and consider
$\pi^*_P\alpha_\bp\in\Omega^1(P)$. Hence, we have
$\pi^*_P\alpha_\bp\in\Gamma(\V_P^\circ)^G$, and we can find, as in the proof of
Lemma \ref{V_circ_G}, an element $\alpha$ of $\Gamma(\V^\circ)^G$ satisfying
$\iota_P^*\alpha=\pi^*_P\alpha_\bp$. The one-form $\alpha$ pushes forward to
$\bar\alpha\in\Omega^1(\bar M)$ and, with 
\[\pi^*_P\alpha_\bp=\iota_P^*\alpha=\iota_P^*\pi^*\bar\alpha
=\pi_P^*\iota_\bp^*\bar\alpha\]
and the fact that $\pi_P$ is a smooth surjective submersion,
we get the equality of $\alpha_\bp$ and $ \iota_\bp^*\bar\alpha$.
\end{proof}

\medskip

Our last two lemmas are rather technical. 
Let $E_P$ be the vector bundle
$E_P=TM\an{P}\operp T^*M\an{P}$ over $P$, endowed with
$\pairing_{E_P}=\pairing\an{E_P\times E_P}$.  Note that this pairing is
automatically symmetric 
and nondegenerate,
since these properties are satisfied pointwise. 
\begin{lemma}\label{sum}
If the intersection $D\cap (\T\operp\V^\circ_G)$ is smooth, then we have 
\[(D\an{P}\cap(\mathcal{T}\operp \V^\circ_G)\an{P})^\perp
=D\an{P}+\K\an{P}+(TP^\perp\operp TP^\circ)\]
as smooth generalized subdistributions of $E_P$ endowed with
$\pairing_{E_P}$.
\end{lemma}

\begin{proof}
By Lemma \ref{V_circ_G}, we know that $(\V^\circ_G\an{P})^\circ=\V\an{P}\oplus
TP^\perp$.
From this follows immediately the equality: 
\[{(\T\operp\V_G^\circ)\an{P}}^\perp=(\V\an{P}\oplus
TP^\perp)\operp TP^\circ=
\K\an{P}\oplus(TP^\perp\operp TP^\circ),\]
and hence also 
\[\left(\K\an{P}\oplus(TP^\perp\operp TP^\circ)\right)^\perp
=(\T\operp\V_G^\circ)\an{P}\]
since $(\T\operp\V_G^\circ)\an{P}$ and $\K\an{P}\oplus(TP^\perp\operp TP^\circ)$ are
vector bundles over $P$.

Now, since the intersection $D\cap(\T\operp\V_G^\circ)$ is spanned by the
descending sections of $D$, it is in particular smooth. Its restriction to $P$
is then also smooth and Proposition \ref{prop_intersection} yields 
\[(D\an{P}\cap(\mathcal{T}\operp \V^\circ_G)\an{P})^\perp
= D\an{P}+\K\an{P}+(TP^\perp\operp TP^\circ).\]
(Note that the sum is not necessarily direct anymore.)
\end{proof}

\begin{corollary}\label{orth_of_inter}
If the intersection $D\cap (\T\operp\V^\circ_G)$ is smooth, we have 
\[\left(D\cap (\T\operp\V^\circ_G)\right)^\perp=D+\K\]
as smooth generalized distributions.
\end{corollary}
\begin{proof}
The inclusion $D+\K\subseteq \left(D\cap (\T\operp\V^\circ_G)\right)^\perp$ is
easy.

Let $m\in M$. If $m\in M^{\rm reg}$, the previous lemma shows that 
$$\left(D\cap (\T\operp\V^\circ_G)\right)^\perp(m)=(D+\K)(m)$$ since $M^{\rm
  reg}$ is open and dense in $M$. 

Let $m\in P\subseteq M\setminus M^{\rm reg}$,
where $P$ is a connected component of the orbit type manifold of $m$. Let
$(X,\alpha)$ be a section of $ \left(D\cap (\T\operp\V^\circ_G)\right)^\perp$
defined on a neighborhood $U$ of $m$. Since $U\cap M^{\rm reg}$ is open and
dense in $U$, we find a sequence $(x_n)_{n\in\N}$ in $U\cap M^{\rm reg}$
converging to $m$. Since $(X,\alpha)$ is smooth, we have 
$\lim_{n\to \infty}(X(x_n),\alpha(x_n))=(X(m),\alpha(m))$. But from the above we
know that $(X(x_n),\alpha(x_n))\in(D+\K)(x_n)$ for all $n\in \N$. Since the sum
$D+\K$ is closed, we have $(X(m),\alpha(m))\in(D+\K)(m)$.
\end{proof}

Here we present  an example inspired by \cite{Bierstone75} to illustrate the theory.
In the following, we denote by $\T_G$ the distribution on $M$ spanned by the family of
$G$-invariant vector fields on $M$.
\begin{example}\label{bierstones_example}
We consider the diagonal action $\Phi$
of $G:=\SO$ on $M:=\R^3\times \R^3$,
that is, $\Phi:\SO\times ( \R^3\times \R^3)\to\R^3\times \R^3$,
$\Phi(A,v,w):= A \cdot (v, w): = (Av,Aw)$. 
This action is proper since $\SO$ is a compact Lie group.

We have $\Phi_A(v,w)=(v,w)$
if and only if $Av=v$ and $Aw=w$, i.e., the rotation $A$ fixes 
$v$ and $w$. Hence, we have the following three
cases:
\begin{enumerate}
\item $(v,w)=(0,0)$: in this case the isotropy subgroup
is $G_{(0,0)}=\SO$,
\item $v$ and $w$ are linearly independent: $G_{(v,w)}=\{\Id_3\}$,
\item $v$ and $w$ are linearly dependent and not both equal to zero; 
without loss of generality assume that $v\neq 0$:\newline 
$G_{(v,w)}=\{A\in \SO\mid A\text{ is a rotation }\text{with axis }v\}$.
\end{enumerate}
Thus there are infinitely many isotropy type manifolds (one for each 
fixed direction $v\in \mathbb{R}^3$ in the third case) and 
three orbit type manifolds $M_0:=M_{\left(\SO\right)}, 
M_2:=M_{\left(\{\Id_3\}\right)}$, and
$M_1:=M_{\left(\operatorname{SO}(2)\right)}$, 
where
\[\operatorname{SO}(2)\simeq 
\left\{\left.\begin{pmatrix}
\cos\alpha&-\sin\alpha&0\\
\sin\alpha&\cos\alpha&0\\
0&0&1
\end{pmatrix}\right|\,
\alpha\in\R
\right\}\subset \SO
\]
is the isotropy subgroup of  $(e_3,e_3)$, corresponding to the isotropy
type manifold
\[
M_{\operatorname{SO}(2)}=\left\{(ae_3,be_3)\mid (a,b)\in\R^2\setminus\{(0,0)\}\right\}.
\]
Define $f_1,f_2,f_3:\R^3\times \R^3\to\R$ by
$f_1(v,w)=\|v\|^2$, $f_2(v,w)=\|w\|^2$ and $f_3(v,w)=\langle v,w\rangle$.
The pairs $(v,w)$ and $(v',w')$ are in the same $G $-orbit
if and only if 
the three functions are equal on $(v,w)$ and $(v',w')$. Indeed, if 
$\|v\|=\|v'\|$, $\|w\|=\|w'\|$, and $\langle v,w\rangle=\langle v',w'\rangle$,
then there exists a rotation $A\in\SO$ such that $Av=v'$ and $Aw=w'$.

The orbit space is thus diffeomorphic to the subset $\bar M$ of $\R^3$ defined
by $$\bar M:=\{(f_1,f_2,f_3)(v,w)\mid (v,w)\in\R^3\times\R^3\};$$ hence
$\bar M:=\{(x,y,z)\in\R^3\mid x,y\geq 0\text{ and }
z^2\leq xy\}$ by the Cauchy-Schwarz inequality. This is a stratified
space with three strata 
\begin{align*}
\bar{P}_0&:=\{(0,0,0)\}= M_0/\operatorname{SO}(3),\\
\bar{P}_1&:=\{(x,y,z)\in\R^3\mid x,y\geq 0,\, (x,y)\neq (0,0) \text{ and }
z^2=xy\}= M_1/\operatorname{SO}(3), \\
\bar P_2&:=\{(x,y,z)\in\R^3\mid x,y>0\text{ and }
z^2<xy\}= M_2/\operatorname{SO}(3)
\end{align*}
(compare with Proposition \ref{equ_of_smooth_structures}).
In Figure \ref{figure1} we have represented the strata $\bar P_0$ (the point
$(0,0,0)$)
 and $\bar P_1$ (the surface without the singular point $(0,0,0)$)
of the reduced space $\bar M=(\R^3\times\R^3)/\SO$. The manifold 
$\bar P_2$ is the 
 open set ``inside'' the surface.
\begin{figure}\label{figure1}
\begin{pspicture}(-5,-5)(7,7) 
\psset{Alpha=160,Beta=20}
\pstThreeDCoor[xMin=0,yMin=0,zMin=-5,xMax=7,yMax=7,zMax=7,linewidth=0.5pt]
\psdots(0,0,0)
\psplotThreeD[plotstyle=line,drawStyle=yLines,hiddenLine,
yPlotpoints=20,xPlotpoints=20,linewidth=0.1pt](0,5)(0,5){x y mul sqrt}
\psplotThreeD[plotstyle=line,drawStyle=xLines,hiddenLine,
yPlotpoints=20,xPlotpoints=20,linewidth=0.1pt](0,5)(0,5){x y mul sqrt}
\psplotThreeD[plotstyle=line,drawStyle=yLines,hiddenLine,
yPlotpoints=20,xPlotpoints=20,linewidth=0.1pt](0,5)(0,5){x y mul sqrt neg}
\psplotThreeD[plotstyle=line,drawStyle=xLines,hiddenLine,
yPlotpoints=20,xPlotpoints=20,linewidth=0.1pt](0,5)(0,5){x y mul sqrt neg}
\end{pspicture}
\caption{}
\end{figure}

We use the  coordinates $(x_1,y_1,z_1,x_2,y_2,z_2)$ on $\R^3\times\R^3$:
\[
p:=(v,w)=(x_1,y_1,z_1,x_2,y_2,z_2).\]
The invariant functions $f_1,f_2,f_3$ are given in these 
coordinates by
$f_1(p)=x_1^2+y_1^2+z_1^2$, $f_2(p)=x_2^2+y_2^2+z_2^2$, and
$f_3(p)=x_1x_2+y_1y_2+z_1z_2$.
Using Lemma \ref{description_V_G} and the fact that the three invariant
polynomials $f_1,f_2,f_3$ form a Hilbert basis for the set of $\sfe$-invariant
polynomials on $\R6$ (see \cite{Bierstone75} and
the Theorem of Schwarz-Mather as presented in e.g., \cite{Pflaum01};  
\cite{OrRa04} has a quick summary), we get:
\begin{align*}
\V_G^\circ(p)&=\erz\left\{\dr f_1, \dr f_2,\dr f_3\right\}(p)\\
&=\erz\left\{\begin{array}{c}
x_1\dr x_1+y_1\dr y_1+z_1\dr z_1,
x_2\dr x_2+y_2\dr y_2+z_2\dr z_2,\\
x_1\dr x_2+x_2\dr x_1+y_1\dr y_2+y_2\dr y_1+z_1\dr z_2+z_2\dr z_1
\end{array}
\right\}(p).
\end{align*}
The vertical distribution is easily computed to be
\[\V(p)=\erz\left\{\begin{array}{c}
x_1\partial_{ y_1}-y_1\partial_{ x_1}+x_2\partial_{ y_2}-y_2\partial_{ x_2},\\
x_1\partial_{ z_1}-z_1\partial_{ x_1}+x_2\partial_{ z_2}-z_2\partial_{ x_2},\\
z_1\partial_{ y_1}-y_1\partial_{ z_1}+z_2\partial_{ y_2}-y_2\partial_{ z_2}
\end{array}
\right\}(p).
\]
We also get (see the appendix of \cite{JoRa10b})
\[\T(p)=\erz\left\{\begin{array}{c}
X_1:=x_1\partial_{ x_1}+y_1\partial_{ y_1}+z_1\partial_{ z_1},\\
X_2:=x_2\partial_{ x_2}+y_2\partial_{ y_2}+z_2\partial_{ z_2},\\
X_3:=x_1\partial_{ x_2}+y_1\partial_{ y_2}+z_1\partial_{ z_2},\\
X_4:=x_2\partial_{x_1}+y_2\partial_{y_1}+z_2\partial_{ z_1},\\
X_5:=x_1\partial_{ y_1}-y_1\partial_{ x_1}+x_2\partial_{ y_2}-y_2\partial_{ x_2},\\
X_6:=x_1\partial_{ z_1}-z_1\partial_{ x_1}+x_2\partial_{ z_2}-z_2\partial_{ x_2},\\
X_7:=z_1\partial_{ y_1}-y_1\partial_{ z_1}+z_2\partial_{ y_2}-y_2\partial_{ z_2},\\
X_8:=(y_2z_1-z_2y_1)\partial_{ x_1}+(z_2x_1-z_1x_2)\partial_{ y_1}\\
+(x_2y_1-y_2x_1)\partial_{ z_1},\\
X_9:=(v\times w)_x\partial_{ x_2}+(v\times w)_y\partial_{ y_2}
+(v\times w)_z\partial_{ z_2},\\
X_{10}:=((v\times w)\times v)_x\partial_{ x_1}
+((v\times w)\times v)_y\partial_{ y_1}\\
+((v\times w)\times v)_z\partial_{ z_1}
+((v\times w)\times w)_x\partial_{ x_2}\\
+((v\times w)\times w)_y\partial_{ y_2}
+((v\times w)\times w)_z\partial_{ z_2}
\end{array}
\right\}(p),
\]
where $(v\times w)_x$, $(v\times w)_y$ and $(v\times
w)_z$
are the $x$-, $y$- and $z$-components of the vector product
$v\times w$,
and 
\[\T_G(p)=\erz_\R\left\{
X_1,
X_2,
X_3,
X_4,
X_8,
X_9,
X_{10}
\right\}(p)
\]
for all $p\in\R^3\times \R^3$.

\medskip

We verify the statement of Lemma \ref{V_circ_G} for this particular example.
 We denote by $\iota_{M_i}:M_i\hookrightarrow M$ the inclusions for $i=1,2,3$.
We have to show that $\iota_{M_i}^*\left(\V_G^\circ\right)=\left(\V_{M_i}\right)^\circ$ for  $i=1,2,3$,
where
$\V_{M_i}$ is the vertical space of the induced action of $G$ on the stratum
$M_i$. 
The statement is obvious for the two strata $M_0$ and $M_2$ since the first is
a point and the second is an open set in $M$. Hence, we study the 
manifold $M_1$. We have a $G$-equivariant diffeomorphism
\[\begin{array}{lccc}
\psi:&\mathds{S}^2\times (\R^2\setminus\{(0,0)\})&\to& M_1,\\
&(u,(a,b))&\mapsto&(au,bu),
\end{array}
\]
where the $G$-action on $\mathds{S}^2\times (\R^2\setminus\{(0,0)\})$ is
given by 
\[\begin{array}{lccc}
\Phi:&\SO\times\left(\mathds{S}^2\times (\R^2\setminus\{(0,0)\})\right)&\to& 
\mathds{S}^2\times (\R^2\setminus\{(0,0)\}),\\
&(A,(u,(a,b)))&\mapsto&(Au,(a,b)).
\end{array}
\]
The vertical space of the $\SO$-action on $M_1$ thus corresponds
to the tangent space of the sphere $T\mathds S^2\oplus\{0\}$ via
the identification $\psi$. 
Hence $\psi^*\left(\left(\V_{M_1}\right)^\circ\right)$  is spanned by the two one-forms 
$\dr a$ and $\dr b$, where $a,b$ are the coordinates on the
$\R^2\setminus\{(0,0)\}$-factor.

The functions $(\psi^*\circ\iota_{M_1}^*)f_i$, $i=1,2,3$, are
given by
\[\left((\psi^*\circ\iota_{M_1}^*)f_1\right)(u,a,b)=a^2,\quad  
\left((\psi^*\circ\iota_{M_1}^*)f_2\right)(u,a,b)=b^2,\]
and 
\[\left((\psi^*\circ\iota_{M_1}^*)f_3\right)(u,a,b)=ab.
\]
We hence  get
\[\left((\psi^*\circ\iota_{M_1}^*)\dr f_1\right)(u,a,b)=2a\dr a,\quad
\left((\psi^*\circ\iota_{M_1}^*)\dr f_2\right)(u,a,b)=2b\dr b\]
and
\[
\left((\psi^*\circ\iota_{M_1}^*)\dr f_3\right)(u,a,b)=a\dr b+b\dr a.
\]
Since the coordinates $a$ and $b$ are never simultaneously zero on
$\mathds{S}^2\times (\R^2\setminus\{(0,0)\})$, we conclude that
$(\psi^*\circ\iota_{M_1}^*)(\V_G^\circ)$ is spanned 
at each point of $\mathds{S}^2\times (\R^2\setminus\{(0,0)\})$ by the values
at this point of $\dr a$ and $\dr b$. This proves the  desired equality
 $\left(\V_{M_1}\right)^\circ = \iota_{M_1}^*(\V_G^\circ)$.

\medskip

Another interesting fact to be checked directly is the equality between the  
accessible sets of the distribution
$\T$ (respectively $\T_G$) and the orbit type manifolds 
(respectively the isotropy type manifolds). The flows $\phi^1,\ldots, \phi^{10}$
of the vector fields $X_1,\ldots, X_{10}$ are given 
by
\begin{align*}
\phi^1_t(v,w)&=(e^tv,w),\qquad
\phi^2_t(v,w)=(v,e^tw),\\
\phi^3_t(v,w)&=(v,tv+w),\qquad
\phi^4_t(v,w)=(tw+v,w),\\
\phi^5_t(v,w)&=R_{e_3,t}\cdot(v, w),\qquad
\phi^6_t(v,w)=R_{e_2,-t}\cdot(v,w),\\
\phi^7_t(v,w)&=R_{e_1,t}\cdot(v, w),\\
\phi^8_t(v,w)&=\exp(tB_w)\cdot(v,w),\qquad
\phi^9_t(v,w)=\exp(tB_v)\cdot(v,w),\\
\phi^{10}_t(v,w)&=\exp(tB_{v\times w})\cdot(v,w),
\end{align*} 
where $R_{e_i,t}\in\SO$ is the rotation about the $e_i$-axis by the angle
$t\in \R$ and 
\[B_w:=\begin{pmatrix}
0&z_2&-y_2\\
-z_2&0&x_2\\
y_2&-x_2&0
\end{pmatrix}\in\lie {so}(3), \quad \text{for} \quad w= (x_2, y_2, z_2).
\] 
A straightforward computation for the rotations about the axes shows that\linebreak
$A\exp(tB_w)A^{-1}=\exp(tB_{Aw})$ for all 
$A\in \operatorname{SO}(3)$. Since $B_ww=0$ it follows that  $\exp(tB_w)w=w$
and $\exp(tB_w)\in\SO$ is a rotation with axis $w$ if $w\neq 0$, and the
identity if $w=0$. 

We have $\phi_t^i(0,0)=(0,0)$ for $i=1,\ldots,10$ and all $t\in \R$, which
shows that the accessible set of $\T$ and of $\T_G$ through the origin
$(0,0)$ is the origin and we recover the orbit and isotropy type manifold
$\{(0,0)\}=M_{(\SO)}=M_{\SO}$.

If $v$ and $w$ are linearly independent (respectively dependent), it
is easy to verify that the two components of $\phi^i_t(v,w)$ are linearly
independent (respectively dependent) for $i=1,\ldots,10$ and all $t\in \R$. 
This shows that the flow of each of the
vector fields $X_1,\ldots,X_{10}$ leave the orbit type manifolds invariant and that 
the flows
of each of the spanning vector fields of $\T_G$ leaves the isotropy type manifolds 
invariant (note that 
$\phi^8_t(v,w)=\phi^9_t(v,w)=\phi^{10}_t(v,w)=(v,w)$
if $v$ and $w$ are linearly dependent).
Hence, we have to verify that each two pairs of vectors in the same 
isotropy (respectively orbit) type manifold can be joined by a concatenation of 
paths formed by pieces of integral curves of the
vector fields spanning $\T_G$ (respectively $\T$).

We start with linearly dependent pairs.
Choose
$(v,w)\neq(0,0)$ and $(v',w')\neq(0,0)$ in the same isotropy type
$\{(au,bu)\mid a, b\in\R,\, (a,b)\neq (0,0)\}$ for some $u\neq 0$ in $\R^3$.
Write $(v,w)=(au,bu)$ and $(v',w')=(a'u,b'u)$. There are several different 
cases to be considered.
\begin{enumerate}
\item If $aa'>0$ and $bb'>0$, then  $(v,w)$ can 
be joined to $(v', w') $ by flow lines of $\phi^1$ and $\phi^2$.
Indeed, if we set $t_1=\ln(\frac{a'}{a})$ and  
$t_2=\ln(\frac{b'}{b})$, we get
$(a'u,b'u)=(\phi^2_{t_2}\circ\phi^1_{t_1})(au,bu)$. 
\item If $b\neq 0$ and $a'\neq 0$, 
set $t_1=\frac{a'-a}{b}$ and $t_2=\frac{b'-b}{a'}$ and get
$(\phi^3_{t_2}\circ\phi^4_{t_1})(au,bu)
=\phi^3_{t_2}(a'u,bu)=(a'u,b'u)$.
\item If $a\neq 0$ and $b'\neq 0$, 
set $t_1=\frac{b'-b}{a}$ and $t_2=\frac{a'-a}{b'}$ and get
$(\phi^4_{t_2}\circ\phi^3_{t_1})(au,bu)
=\phi^3_{t_2}(a'u,bu)=(a'u,b'u)$.
\item If $aa'>0$ and $b=b'=0$, set
 $t=\ln(\frac{a'}{a})$. Then $(a'u,0)=\phi^1_{t}(au,0)$.
Use the same method with $\phi^2$ for the case 
$bb'>0$ and $a=a'=0$.
\item If $aa'<0$ and $b=b'=0$, then $(au,-au)=\phi^3_{-1}(au,0)$
and we can continue as in case 2.
Use the same method with $\phi^4$ and case 3 for the case 
$bb'<0$ and $a=a'=0$.
\end{enumerate}

To join two pairs $(au,bu)$ and $(a'v,b'v)$ in the \emph{orbit type} manifold of 
linearly dependent pairs (choose $u$ and $v$ of unit length), we use a 
combination of integral curves of
$X_5,X_6,X_7$ to send $(au,bu)$ to
$\left(av,bv\right)$ 
by a rotation and then we proceed as above 
using integral curves of $X_1,X_2,X_3$, and $X_4$. 

\medskip

The isotropy type manifold through a linearly independent 
pair is equal to the 
orbit type manifold through this pair:
\[M_{\Id_3}=M_{(\Id_3)}=\{(v,w)\in\R^3\times\R^3\mid (v,w)
\text{ linearly independent}\}.
\]
It is then possible to join two pairs of this type by integral curves of the
vector fields $X_1,\ldots,X_4$  and $X_8,X_9,X_{10}$. 
First, we show that if $v',w'$ lie in the span of $v,w$, we can 
join $(v,w)$ to $(v',w')$ by pieces of the integral curves of
$X_1,\ldots,X_4$. Indeed, there exist $a,b,c, d\in \R$ such that
$v'=av+bw$ and $w'=cv+dw$.
\begin{enumerate}
\item If $b=0$ (that is, $v$ and $v'$ are linearly dependent), 
then $a$ and $d$ have to be nonzero because $av+bw=av$ and $cv+dw$ are linearly independent.
We then have several subcases:
\begin{enumerate}
\item If $a>0$ and $d>0$, then
  $(av,cv+dw)=(\phi_{\ln a}^1\circ\phi^3_c\circ\phi^2_{\ln d})(v,w)$.
\item If $a>0$ and $d<0$, then we have 
\begin{align*}
&\left(\phi_{\ln a}^1\circ\phi^3_{c+d\frac{\langle v,w\rangle}{\langle v,v\rangle}}\circ\phi^2_{\ln (-d)}\circ
\phi^9_\pi\circ\phi^3_{-\frac{\langle v,w\rangle}{\langle v,v\rangle}}\right)(v,w)\\
=&\left(\phi_{\ln a}^1\circ\phi^3_{c+d\frac{\langle v,w\rangle}{\langle v,v\rangle}}\circ\phi^2_{\ln (-d)}\circ
\phi^9_\pi\right)\left(v,w-\frac{\langle v,w\rangle}{\langle v,v\rangle}v\right)\\
=&\left(\phi_{\ln a}^1\circ\phi^3_{c+d\frac{\langle v,w\rangle}{\langle v,v\rangle}}\circ\phi^2_{\ln (-d)}\right)
\left(v,-w+\frac{\langle v,w\rangle}{\langle v,v\rangle}v\right)=(av,cv+dw).
\end{align*}
We have used the fact that since $v$ and $w-\frac{\langle v,w\rangle}{\langle v,v\rangle}v$ are
orthogonal,
the rotation $\exp(\pi B_v)$ of angle $\pi$ around the axis spanned by $v$
sends $  w-\frac{\langle v,w\rangle}{\langle v,v\rangle}v$ to $-w+\frac{\langle v,w\rangle}{\langle v,v\rangle}v$.
\item  If $a<0$ and $d>0$, then we have in an analogous manner
$$\left(\phi_{\ln (-a)}^1\circ\phi^3_{-\left(c+d\frac{\langle v,w\rangle}{\langle v,v\rangle}\right)}\circ\phi^2_{\ln d}\circ
\phi^8_\pi\circ\phi^3_{-\frac{\langle v,w\rangle}{\langle v,v\rangle}}\right)(v,w)=(av,cv+dw).$$
\item Finally, if $a<0$ and $d<0$, we have
$$\left(\phi_{\ln (-a)}^1\circ\phi^3_{-c}\circ\phi^2_{\ln (-d)}\circ\phi^{10}_\pi\right)(v,w)=(av,cv+dw).$$
\end{enumerate}
\item If $b\neq 0$, choose $t_1$ such that $t_1b-a\neq 0$ and
  $t_2=\frac{-b}{t_1b-a}$.
Then we have $(\phi^4_{t_2}\circ\phi^3_{t_1})(v,w)=((1+t_1t_2)v+t_2w,
w+t_1v)$.
Since $(1+t_1t_2)b-t_2a=\left(1+t_1\frac{-b}{t_1b-a}\right)b+\frac{ab}{t_1b-a}
=b\left(1+\frac{a-t_1b}{t_1b-a}\right)=0$, the vectors 
$av+bw$ and $(1+t_1t_2)v+t_2w$ are linearly dependent and we continue as in 
case 1. (Using an
integral curve of $\phi^{10}$, we can also first
rotate $v$ and $w$ around the axis $v\times w$ so that the images of $v$ and $v'$ are
linearly dependent, and then continue as in case 1.)
\end{enumerate}

Then, to simplify the problem, we assume, without loss of generality, that the plane 
spanned by $(v,w)$ is the $(x,y)$-plane $\alpha_{xy}$ (spanned by $e_1$ and $e_2$). By the considerations
above, we can bring $(v,w)$ to $(e_1,e_2)$  along pieces of integral
curves of $X_1,\ldots,X_4$.
Hence, to finish the proof it suffices to show that we can use pieces of integral curves of 
$X_8,X_9,X_{10}$ to bring the plane spanned by $e_1,e_2$ to the plane $\alpha$
spanned 
by an arbitrary linearly independent pair $(v',w')$.
If $(v',w')$ spans the $(x,y)$-plane $\alpha_{xy}$, we are done by the
considerations above. Otherwise, there
are again two cases
\begin{enumerate}
\item If $\alpha$ is equal to the plane $\alpha_{xz}$ spanned by $e_1$ and
  $e_3$, then we have $(e_1,e_3)=\phi_{\pi/2}^9(e_1,e_2)$ and we are done.
\item If not, there exists a unit vector 
$u\in\R^3$ spanning the intersection  $\alpha_{xz}\cap\alpha$.
Then there exists $t_1\in \R$ such that 
$(u,e_2)=\exp(t_1B_{e_2})\cdot (e_1,e_2)=\phi^8_{t_1}(e_1,e_2)$. 
If $e_2$ lies in  $\alpha$, the vectors $u$ and $e_2$ are 
linearly independent by construction and we are done.
Otherwise, let $u'$ 
be a unit vector
spanning the intersection of $\alpha$ with the plane spanned by $e_2$ and
$\exp(tB_{e_2})e_3$ (this is the plane orthogonal to $u$). Then there exists 
$t_2$ such that
$(u,u')=\exp(t_2B_u)\cdot (u,e_2)=\exp(t_2B_u)\exp(t_1B_{e_2})\cdot (e_1,e_2)
=(\phi^9_{t_2}\circ\phi^8_{t_1})(e_1,e_2)$.
(See Figure \ref{figure2}.)
\end{enumerate}

\begin{figure}\label{figure2}
\begin{pspicture}(-7,-4)(6,6)
\psset{unit=2cm,drawCoor,beginAngle=90,endAngle=180}
\pstThreeDCoor[drawing=true, linewidth=1pt,linecolor=black,
linestyle=solid,plotstyle=curve,,xMin=-1.5,xMax=3,yMin=-1,yMax=3,zMin=-0.5,zMax=3]
\pstThreeDEllipse[linewidth=1.5pt,arrows=->,beginAngle=0,endAngle=45](0,0,0)(2,0,0)(0,0,2)
\pstThreeDEllipse[linewidth=1.5pt,arrows=->,beginAngle=90,endAngle=135](0,0,0)(2,0,0)(0,0,2)
\pstThreeDEllipse[linewidth=1.5pt,arrows=->,beginAngle=0,endAngle=35](0,0,0)(0,2,0)(-1.4142,0,1.4142)
\pstThreeDDot[dotstyle=none](1,2,3)
\pstThreeDLine[linecolor=red,linewidth=1pt,arrows=c->](0,0,0)(1,2,3)
\pstThreeDDot[dotstyle=none](0.5,-0.5,0)
\pstThreeDLine[linecolor=red,linewidth=1pt,arrows=c->](0,0,0)(0.5,-0.5,0)
\pstThreeDDot[dotstyle=none](1.4142,0,1.4142)
\pstThreeDLine[linecolor=red,linewidth=1pt,arrows=c->](0,0,0)(1.4142,0,1.4142)
\pstThreeDPut(1.1,1.7,2.7){$v'$}
\pstThreeDPut(0.6,-0.6,0.1){$w'$}
\pstThreeDPut(1.6,0,1.6){$u$}
\pstThreeDLine[linecolor=blue,linewidth=1pt,arrows=c->](0,0,0)(2,0,0)
\pstThreeDLine[linecolor=blue,linewidth=1pt,arrows=c->](0,0,0)(0,2,0)
\pstThreeDLine[linecolor=blue,linewidth=1pt,arrows=c->](0,0,0)(0,0,2)
\pstThreeDPut(2.2,-0.1,0.1){$e_1$}
\pstThreeDPut(0.4,2.4,0.2){$e_2$}
\pstThreeDPut(0.4,0,2){$e_3$}
\pstThreeDLine[linecolor=blue,linewidth=1pt,arrows=c->](0,0,0)(-1.4142,0,1.4142)
\pstThreeDDot[dotstyle=none](-1.4142,0,1.4142)
\pstThreeDDot[dotstyle=none](-0.8165,1.63,0.8165)
\pstThreeDLine[linecolor=red,linewidth=1pt,arrows=c->](0,0,0)(-0.8165,1.63,0.8165)
\pstThreeDPut(-1,1.7,0.95){$u'$}
\pstThreeDPut(-1.7,0,1.7){$\exp(t_1B_{e_2})(e_3)$}
\end{pspicture}\caption{}
\end{figure}

These considerations illustrate Theorem \ref{leaves_of_T} stating that the
integral
leaves of $\T$ are the connected components of the orbit type manifolds
and Theorem 3.5.1 in \cite{OrRa04} stating that the
integral
leaves of $\T_G$ are the connected components of the isotropy type manifolds.

\medskip

We now study  properties of the stratified space $\bar M$. We want to show that
the restrictions of the \emph{vector fields}
on $\bar M$ to the strata of $\bar M$ span the tangent space of each stratum. 
The flows $\bar \phi^1,\ldots,
\bar\phi^{10}$ associated to the vector fields $\bar X_1,\ldots, \bar X_{10}$
defined
by $X_i\sim_\pi\bar X_i$ for $i=1,\ldots,10$ are given by
\begin{align*}
\bar\phi^1_t(x,y,z)&=\left(\bar\phi^1_t\circ\pi\right)(v,w)=
\left(\pi\circ\phi^1_t\right)(v,w)=\pi(e^tv,w)=(e^{2t}x,y,e^tz),\\
\bar \phi^2_t(x,y,z)&=(x,e^{2t}y,e^tz),\\
\bar\phi^3_t(x,y,z)&=\pi(v,tv+w)=(x,t^2x+2tz+y,xt+z),\\
\bar\phi^4_t(x,y,z)&=(t^2y+2tz+x,y,yt+z),\\
\bar\phi^5_t(x,y,z)&=\bar\phi^6_t(x,y,z)
=\bar\phi^7_t(x,y,z)=\bar\phi^8_t(x,y,z)\\
&=\bar\phi^9_t(x,y,z)
=\bar\phi^{10}_t(x,y,z)=(x,y,z).
\end{align*} 
This leads to
\begin{align*}
\bar X_1(x,y,z)&=2x\partial_{x}+z\partial_{z},\qquad 
\bar X_2(x,y,z)=2y\partial_{y}+z\partial_{z},\\
\bar X_3(x,y,z)&=2z\partial_{y}+x\partial_z,\qquad 
\bar X_4(x,y,z)=2z\partial_{x}+y\partial_z,\\
\bar X_5(x,y,z)&=\bar X_6(x,y,z)
=\bar X_7(x,y,z)=\bar X_8(x,y,z)\\
&=\bar X_9(x,y,z)=\bar X_{10}(x,y,z)=0.
\end{align*} 
The last equalities are consistent with the fact that  
$X_5,\ldots,X_{10}$ are sections of the vertical space $\V$. At the point
$(0,0,0)$, we have hence $\bar X_1(0,0,0)=\ldots=\bar X_{4}(0,0,0)=0$, 
and we conclude that the (trivial) tangent space of $\bar P_0$ is spanned by the
values at $(0,0,0)$ of the vector fields on $\bar M$.
The stratum $\bar P_1$ can be seen as the manifold given by the equation 
$xy=z^2$ in $(\R^2\setminus\{(0,0)\})\times \R$. Thus, we know that the tangent  space of $\bar{P}_1$ is equal
to the kernel of $x\dr y+y\dr x-2z\dr z$ at points of $\bar{P}_1$. 
We  thus find that the values of $\bar{X}_1,\ldots,\bar{X}_{4}$
at points of $\bar{P}_1$ span the tangent space of $\bar P_1$
(recall that $x$ and $y$ never vanish simultaneously on $\bar P_1$).
The points $(x,y,z)$ of the last stratum $\bar P_2$ satisfy
$x,y>0$ and $z^2<xy$. Hence, for $p=(x,y,z)\in\bar P_2$, we have
\begin{align*}
&\,\erz_\R\{ 2x\partial_{x}+z\partial_{z}, 
2y\partial_{y}+z\partial_{z}, 2z\partial_{y}+x\partial_z, 
2z\partial_{x}+y\partial_z\}(p)\\
=&\,  \erz_\R\left\{ \partial_{x}+\frac{z}{2x}\partial_{z}, 
\partial_{y}+\frac{z}{2y}\partial_{z}, \frac{2z}{x}\partial_{y}+\partial_z, 
\frac{2z}{y}\partial_{x}+\partial_z\right\}(p)\\
=&\,\erz_\R\left\{\partial_{x},\partial_{y},\partial_{z}\right\}(p),\\
\end{align*}
where we have used $\frac{z^2}{xy}<1$ and the identities 
\begin{align*}
\left( \partial_{x}+\frac{z}{2x}\partial_{z}\right)
-\frac{z}{2x}\left( \frac{2z}{y}\partial_{x}+\partial_z\right)
&=\left(1-\frac{z^2}{xy} \right)\partial_x,\\
\left( \partial_{y}+\frac{z}{2y}\partial_{z}\right)
-\frac{z}{2y}\left( \frac{2z}{x}\partial_{y}+\partial_z\right)
&=\left(1-\frac{z^2}{xy} \right)\partial_y.
\end{align*}

\medskip

Finally, we study in the same manner the push-forwards of the three one-forms
$\dr f_1$, $\dr f_2$, $\dr f_3$ spanning $\V_G^\circ$. Denote by $\bar
\alpha_1$, $\bar \alpha_2$, $\bar \alpha_3$ these three ``one-forms''
on $\bar M$ (see Subsection \ref{pw_of_vf}).
Since $\dr f_1$, $\dr f_2$, $\dr f_3$ vanish at $0$, we have 
$\bar\alpha_1(0)=\bar \alpha_2(0)=\bar \alpha_3(0)=0$,
by definition, and we conclude $\erz_\R\{\bar\alpha_1(0), \bar \alpha_2(0),
\bar \alpha_3(0)\}=T_0^*\bar P_0$. 
At points of $\bar P_2$, we have $\bar\alpha_1=\dr x$, $\bar\alpha_2=\dr y$,
and $\bar\alpha_3=\dr z$. Finally, at points of 
$M_1$, we have the equality $2f_3\dr f_3=f_1\dr f_2+f_2\dr f_1$ 
and we obtain, as desired, $$\erz_\R\{\bar\alpha_1(x,y,z), \bar \alpha_2(x,y,z),
\bar \alpha_3(x,y,z)\}=\erz_\R\{\dr x,\dr y,\dr z\}_{(x,y,z)}/\sim,$$ where
$\sim $ is the equivalence relation on $\erz_\R\{\dr x,\dr y,\dr
z\}_{(x,y,z)}$ defined by $x\dr y+y\dr x-2z\dr z=0$.

This shows that the restrictions of the ``one-forms'' on $\bar{M}$
to each of its strata span the cotangent space of each stratum, as 
stated in the considerations at the beginning of
 Subsection \ref{pw_of_vf} together with
 Propositions \ref{equ_of_smooth_structures}
and \ref{one-forms}.
\end{example}

\section{Singular reduction of Dirac structures}
\subsection{The special case of conjugated isotropy subgroups}
In the special case of a proper action with conjugated orbit subgroups, the
reduction theorem is shown in \cite{JoRaZa11}. We recall its formulation here because 
the understanding of the 
construction of the reduced Dirac structure in this case can be helpful for the
understanding of the general case. 
\begin{theorem}\label{conj-orbits-red} 
Let $G$ be a connected Lie group acting in a proper way on the manifold $M$,
such that all isotropy subgroups are conjugated. Assume that $D\cap\K^\perp$
has constant rank on $M$.
Then the  Dirac structure $D$ on $M$ induces a Dirac structure $\bar D$ on the
quotient $\bar M=M/G$
given by
\begin{equation}\label{red_dir_conj_iso}
\bar D(\bar m)=\left\{\left(\bar X(\bar m),\bar\alpha(\bar m)\right)\in
  T_{\bar m}\bar M\times T^*_{\bar m}\bar M \,\bigg|
\begin{array}{c}\exists X\in
  \mx(M)\text{ such that } \\X\sim_{\pi}\bar X 
\text{ and }(X,\pi^*\bar\alpha)\in\Gamma(D)\end{array}\right\}
\end{equation} for all $\bar m$ in $\bar M$. If $D$ is integrable, 
then  $\bar D$ is also integrable.
\end{theorem}

\begin{remark}
Note that the method we use for singular reduction yields this
\emph{regular reduction theorem} as a corollary of our general singular
reduction theorems for Dirac structures (see Theorems \ref{closed} and 
\ref{singred} in the next subsection).

As in the Poisson case (compare with \cite{FeOrRa09}), it is also  
possible to prove singular reduction by using regular reduction.
Indeed, if $Q\subseteq M_H\subseteq M$ is a connected component of an  
isotropy type, it is
possible to show that the Dirac structure $D$ on $M$ restricts naturally to
a Dirac structure $D_Q$ on $Q$, which would be $N(H)/H$-invariant if
$D$ were $G$-invariant, and integrable if $D$ were integrable.
To prove these statements (see \cite{JoRa11a}) one needs Proposition  
\ref{prop_intersection} and $G$-invariant averaging (see Subsection  
\ref{tubeth}).
Since the action
of  $N(H)/H$ on $Q$ is free and proper, we can use regular Dirac reduction
on the Dirac manifold $(Q,D_Q)$ and get a smooth quotient Dirac manifold
$(\bar{Q},D_{\bar{Q}})$. The manifold $\bar{Q}$ is diffeomorphic to the
quotient $P/G$ if $P=G\cdot Q$ is the connected component of $M_{(H)}$
containing
$Q$. In fact, we have $\bar{P}=\pi(P)=\pi(Q)$ if $\pi:M\to \bar M$ is  
the orbit map,
and $\bar{Q}$ is then a stratum of the reduced space $\bar{M}$.
By construction, it is then easy to see that the reduced Dirac manifold
$(\bar{Q},D_{\bar{Q}})$ is diffeomorphic to the reduced space
$(\bar{P},D_{\bar{P}})$ that we will get in the next subsection. We  
want to thank R. Loja Fernandes for a discussion that resulted in this  
remark.
\end{remark}

\subsection{The general setting of a proper action}

Consider the subset $\D^G$ of $\Gamma(D)$ defined by
\[ \D^G=\{(X,\alpha)\in\Gamma(D)\mid \alpha\in\Gamma(\V^\circ)^G\text{ and }
[X,\Gamma(\V)]\subseteq\Gamma(\V)\},\]
that is, the set of the descending sections of $D$.

We have seen in Lemma \ref{lem:push_down_der} that each vector field $X$ satisfying
$[X,\Gamma(\V)]\subseteq\Gamma(\V)$ pushes forward to a vector field $\bar X$
on $\bar M$. By the considerations in Subsection \ref{subsection_orbits} (see
also Proposition \ref{restriction_of_X}), we know that for each stratum ${\bar
  P}$ of $\bar M$, 
the restriction of $\bar X$ to points
of ${\bar P}$ is a vector field $X_{\bar P}$ on ${\bar P}$. 
On the other hand, if
$(X,\alpha)\in\D^G$, then we have $\alpha\in\Gamma(\V^\circ)^G$ and it
pushes forward to the one-form $\bar\alpha:=\pi_*\alpha$ such that,
for every $\bar Y\in \mx(\bar M)$ and every section $Y$ of $TM$ satisfying
$Y\sim_\pi \bar Y$, we have 
\[\pi^*(\bar\alpha(\bar Y))=\alpha(Y).\]
Moreover, for each stratum ${\bar P}$ of $\bar M$, the restriction of $\bar\alpha$ to
points of ${\bar P}$ defines a $1$-form $\alpha_{\bar P}$ on ${\bar P}$. Let 
\[\bar\D=\{(\bar X,\bar\alpha)\mid (X,\alpha)\in\D^G\}\]
and for each stratum ${\bar P}$ of $\bar M$, set 
\[\D_{\bar P}=\{(X_{\bar P},\alpha_{\bar P})\mid (\bar X,\bar\alpha)\in\bar\D\}.\]
Define the smooth distribution $D_{\bar P}$ on ${\bar P}$ by
\begin{equation}\label{def_D_barP}
D_{\bar P}(s)=\{(X_{\bar P}(s),\alpha_{\bar P}(s))\in T_s{\bar P}\times
T^*_s{\bar P}\mid (X_{\bar P},
\alpha_{\bar P})\in\D_{\bar P}\}.
\end{equation}
\begin{remark}
\label{remark_section_bar_p}
Note that $\Gamma(D_{\bar{P}}) = \mathcal{D}_{\bar{P}}$. Indeed, any 
$( X_{\bar{P}}, \alpha_{\bar{P}}) \in \Gamma( D_{\bar{P}}) $ can be written as 
\[
( X_{\bar{P}}, \alpha_{\bar{P}}) 
= \sum_{i=1}^k f^i_\bp( X^i_{\bar{P}}, \alpha^i_{\bar{P}}), 
\quad ( X^i_{\bar{P}}, \alpha^i_{\bar{P}}) \in \mathcal{D}_\bp, \quad 
f^i_\bp \in C ^{\infty}(\bar{P}).
\]
Each  $(X^i_{\bar{P}}, \alpha^i_{\bar{P}})$ has a smooth extension
 $(\bar{X}^i, \bar{\alpha}^i) \in \bar{ \mathcal{D}}$ which is a push-forward
 of some element 
$( X^i, \alpha^i ) \in \mathcal{D}^G$. Each function $f^i_\bp$ smoothly
extends to a function 
$\bar{f}^i \in C ^{\infty}(\bar{M})$, by the smooth structure of $\bar{P}$ as
a stratum of $\bar{M}$, 
which is a push-forward of a function $f^i \in C ^{\infty}(M)^G$. Therefore 
$\sum_{i=1}^k f^i(X^i, \alpha^i)$ is a descending section of $D$ and the
restriction to 
$\bar{P}$ of its push-forward to $\bar{M}$ coincides with $( X_{\bar{P}}, 
\alpha_{\bar{P}})$.
\end{remark}

\begin{theorem}\label{singred2}
Let $(M,D)$ be a Dirac manifold with a proper Dirac action of a connected Lie
group $G$ on it. Assume that the intersection $D\cap(\T\operp\V_G^\circ)$ is 
spanned by its descending sections.
Then each element $(\bar X,\bar \alpha)\in\mx(\bar M)\times \Omega^1(\bar M)$
orthogonal to all the sections in $\bar \D$ is already an element of $\bar \D$.
\end{theorem}

\begin{proof}
Let $(\bar X,\bar\alpha)\in \mx(\bar M)\times \Omega^1(\bar M)$ be such that 
$\bar\alpha(\bar Y)+\bar\beta(\bar X)=0$ for all elements $(\bar
Y,\bar\beta)\in\bar \D$. Let $(Y,\beta)\in\Gamma(D\cap(\T\operp\V_G^\circ))$
be such that $Y\sim_\pi \bar Y$ and $\beta=\pi^*\bar\beta$. Choose also
$X\in\mx(M)$ such that $X\sim_\pi\bar X$ and set
$\alpha=\pi^*\bar\alpha\in\Omega^1(M)$ (see Proposition \ref{corres-one-forms} 
and the considerations after Lemma \ref{matsu}). Then we get 
\[\langle(X,\alpha),(Y,\beta)\rangle=\left(\bar\alpha(\bar Y)+\bar\beta(\bar
  X)\right)\circ \pi=0.\]
Since $D\cap(\T\operp\V_G^\circ)$ is spanned by its descending sections, we
get $(X,\alpha)\in\Gamma((D\cap(\T\operp\V_G^\circ))^\perp)$. Hence, with
Lemma \ref{orth_of_inter}, we get  $(X,\alpha)\in\Gamma(D+\K)$ and there exist
$X'\in\mx(M)$ and $V\in\Gamma(\V)$ such that $(X',\alpha)$ is a section of $D$
and $(X,\alpha)=(X',\alpha)+(V,0)$. Because of the definition of $\alpha$,
$X$,
and $V$, we immediately get that $(X',\alpha)$ is a descending section of
$D\cap(\T\operp\V_G^\circ)$. It is easy to see that $X'\sim_\pi \bar
X$, and we have $\alpha=\pi^*\bar\alpha$. Thus $(\bar X,\bar\alpha)$ is an
element of $\bar \D$.
\end{proof}
As a consequence of this theorem, we get that the set of pairs $(\bar X,\bar
\alpha)\in\mx(\bar M)\times\Omega^1(\bar M)$ orthogonal to all the elements of
$\bar{\D}$ is 
$\bar{\D}$ itself. 
Hence, it is natural to ask if $D_\bp$ defines a Dirac structure on $\bp$ for
each stratum $\bp$ of $\bar M$. For the stratum $\bar M^{\rm reg}=\pi(M^{\rm reg})$, 
this is
automatically true since $M^{\rm reg}$ is open and dense in $M$.

\begin{theorem}\label{singred}
Let $(M,D)$ be a Dirac manifold with a proper Dirac action of a connected Lie group 
$G$ on it.
Let $\bar P$ be a stratum of the quotient space $\bar M$.
If $D\cap(\T\operp\V_G^\circ)$ is spanned by its descending sections, then
$D_{\bar P}$ defined in \eqref{def_D_barP} is a Dirac structure on ${\bar P}$.
\end{theorem}

\begin{proof} 
Let the
stratum $\bar P$ be a connected component of $\pi(M_{(H)})$ for a compact Lie
subgroup $H$ of $G$, and  let $P$ be the connected
component of $M_{(H)}$ such that $\pi^{-1}(\bar P)=P$.

The inclusion $D_{\bar P}\subseteq D_{\bar P}^\perp$ is easy. For the other
inclusion,  choose  $\bar p\in \bar P$ and $p\in P\subseteq M_{(H)}$ such that
$\pi(p)=\bar p$.   Recall from Proposition \ref{equ_of_smooth_structures} that
the action $\Phi$ of $G$ on $M$ restricts to the proper action $\Phi^P$ on  $P$, 
and  the
quotient map $\pi$ restricts to $\pi_P:=\pi\arrowvert_{P}:P\to\bar P$.

Let $(X_\bp,\alpha_\bp)\in\mx(\bar P)\times\Omega^1(\bar P)$ be a section of
$D_\bp^\perp$ defined on a neighborhood $U_\bp
$ of $\bar p$. Then there exists
$(\bar X,\bar \alpha)\in\mx(\bar M)\times\Omega^1(\bar M)$, with $\dom(\bar
X,\bar \alpha)=:\bar U$, such that
$X_\bp\sim_{\iota_\bp}\bar X$ and $\alpha_\bp=\iota_\bp^*\bar\alpha$, and
$(X,\alpha)\in\mx(M)^G\times\Gamma(\V^\circ)^G$ defined on $U:=\pi^{-1}\bar
U$, such that $X\sim_\pi\bar X$ and
$\alpha=\pi^*\bar\alpha$. 
By Remark \ref{remark_section_bar_p}, we find for
each $(Y_\bp,\beta_\bp)\in\Gamma(D_\bp)$ sections $(\bar Y,\bar
\alpha)\in\bar\D$ and $(Y,\alpha)\in\D^G$ such that $Y\sim_\pi\bar Y$,
$Y_\bp\sim_{\iota_\bp}\bar Y$, and $\alpha=\pi^*\bar\alpha$,
$\alpha_\bp=\iota_\bp^*\bar\alpha$.  We get the equalities
\begin{align*}
\langle(X,\alpha),(Y,\beta)\rangle\circ\iota_P&
= \langle(\bar X,\bar \alpha),(\bar Y,\bar
\beta)\rangle\circ\pi\circ\iota_P\\
&=\langle(\bar X,\bar \alpha),(\bar Y,\bar
\beta)\rangle\circ\iota_\bp\circ\pi_P\\
&=\langle(X_\bp,\alpha_\bp),(Y_\bp,\beta_\bp)\rangle\circ\pi_P=0.
\end{align*}
Since  $D\cap(\T\operp\V_G^\circ)$ is spanned by its descending
sections, we get that $(X,\alpha)\an{P}$ is a section of 
\[\big((D\cap(\T\operp\V_G^\circ))\an{P}\big)^\perp.\] But since
$D\cap(\T\operp\V_G^\circ)$ is spanned by its 
descending sections it is in particular smooth. By Proposition \ref{sum}, we
get that $(X,\alpha)\an{P}$ is a 
section of \[D\an{P}+\K\an{P}+(TP^\perp\operp
TP^\circ).\] 
 Thus, there exist for all $x$ in the $G$-invariant set $U\cap P$ an open
neighborhood $U_x\subseteq
M$ of $x$, and sections 
$(Z^x,\zeta^x)$ of $D$ and  $V^x$ of $\V$ defined on the whole of $M$
(otherwise multiply them with an appropriate bump-function),  and
sections $W^x\in\Gamma(TP^\perp)$ and $\gamma^x\in\Gamma(TP^\circ)$ defined on
$P$  such
that 
\[(X,\alpha)\an{P\cap U_x}=(Z^x,\zeta^x)\an{P\cap U_x}+(V^x,0)\an{P\cap
  U_x}+(W^x,\gamma^x)\an{P\cap U_x}.\]
Since $M$ is paracompact, its open submanifold $U':=\bigcup_{x\in P\cap U}U_x$ is also
paracompact and there exists a locally finite refinement $\mathcal{U}_\Lambda$ of its
open covering $\{U_x\mid x\in
P\cap U\}$, where $\Lambda$ is a subset of $P\cap U$, and a partition of unity
$\{\rho_\lambda\}_{\lambda\in\Lambda}$ subordinate to $\mathcal{U}_\Lambda$. Set
$\rho_\lambda\an{M\setminus U'} =0$ for all $\lambda\in\Lambda$. Then all the
functions $\rho_\lambda$ are defined on the whole of $M$.

Define the global sections 
\[(Z,\zeta)=\sum_{\lambda\in\Lambda}\rho_\lambda(Z^\lambda,\zeta^\lambda),
\quad (V,0)=\sum_{\lambda\in\Lambda}\rho_\lambda(V^\lambda,0),\]
and
\[(W,\gamma)=\sum_{\lambda\in\Lambda}\rho_\lambda(W^\lambda,\gamma^\lambda).\]

Then $(Z,\zeta)$ is a section of $D$, $V$ a section of $\V$, 
$W\in\Gamma(TP^\perp)$ and $\gamma\in\Gamma(TP^\circ)$, and we have for all
$p'\in P\cap U$: 
\begin{align*}
&\bigl((Z,\zeta)\an{P}+(V,0)\an{P}+(W,\gamma)\bigr)(p')\\
=&\sum_{\lambda\in\Lambda}\rho_\lambda(p')\bigl((Z^\lambda,\zeta^\lambda)(p')
+(V^\lambda,0)(p')+(W^\lambda,\gamma^\lambda)(p')\bigr)\\
=&\sum_{\lambda\in\Lambda}\rho_\lambda(p')(X,\alpha)(p')=(X,\alpha)(p').
\end{align*}

Consider the $G$-invariant averages $Z_G$, $\zeta_G$,
$V_G$, $\gamma_G$, $W_G$ of $Z$, $\zeta$,
$V$, $\gamma$, and $W$. 
We get  for all $p'=[g,b]_H\in P\cap U$,
\begin{align*}
\bigl(Z_G+V_G+W_G\bigr)(p')&
=T_{[e,b]_H}\Phi_g\left(\int_{H}\left(\Phi^*_h(Z+V+W)\right)([e,b]_H)dh\right)\\
&=T_{[e,b]_H}\Phi_g\left(\int_{H}T_{[h,b]_H}\Phi_{h^{-1}}(Z+V+W)([h,b]_H)dh\right)\\
&=T_{[e,b]_H}\Phi_g\left(\int_{H}T_{[h,b]_H}\Phi_{h^{-1}}X([h,b]_H)dh\right)\\
&=T_{[e,b]_H}\Phi_g\left(\int_{H}X([e,b]_H)dh\right)=T_{[e,b]_H}\Phi_gX([e,b]_H)\\
&=(\Phi_g^*X)(p')=X(p'),
\end{align*}
where we have used the fact that $X$ is $G$-invariant and that $[g,b]_H$ in
$P\cap U$ implies $[e,b]_H$ and $[h,b]_H\in P\cap U$ for all $h\in H$.
In the same manner, we show that
\[\alpha(p')=(\zeta_G+\gamma_G)(p').\]
Thus we have 
\begin{align*}
(X,\alpha)=(Z_G,\zeta_G)+(V_G,0)+(W_G,\gamma_G)
\end{align*}
on $P\cap U$.
Since all involved distributions $\V$, $TP^\perp$, $TP^\circ$, and $D$ are
$G$-invariant, we still have $V_G\in\Gamma(\V)$, $W_G\in\Gamma(TP^\perp)$,
$\gamma_G\in\Gamma(TP^\circ)$, and $(Z_G,\zeta_G)\in\Gamma(D)$. 
But now we have $X-Z_G-V_G\in\Gamma(\T)$ and hence
$(X-Z_G-V_G)\an{P}\in\Gamma(TP)$.
Thus, the equality $(X-Z_G-V_G)\an{P}=W_G\in\Gamma(TP^\perp)$ leads to 
$W_G=(X-Z_G-V_G)\an{P}=0$.
The section $X_{\gamma_G}\in\Gamma(TP^\perp)$ satisfying
$\ip{X_{\gamma_G}}\rho\an{P}=\gamma_G$ is $G$-invariant, since $\gamma_G$
is. Hence it is tangent to $P$ and has to be the zero section. 
This shows that $\gamma_G=0$.

Thus, we have 
\[(X,\alpha)\an{P}=(Z_G,\zeta_G)\an{P}+(V_G,0)\an{P},\]
and simultaneously $(Z_G,\zeta_G)\an{P}\in 
\Gamma\left((D\cap(\T\operp\V^\circ_G))\an{P}\right)$.
Hence, there exists a section $(X',\alpha')$ of $D\cap(\T\operp\V^\circ_G)$
defined on a neighborhood $U'\subseteq U$ of $p$ 
such that $(X',\alpha')\an{P\cap U'}=(Z_G,\zeta_G)\an{P\cap U'}$, and hence 
$(X,\alpha)\an{P\cap U'}=(X',\alpha')\an{P\cap U'}+(V_G,0)\an{P\cap U'}$. 
By $G$-invariant average (in the same manner as above, with a partition of
unity if needed), we can assume that $(X',\alpha')$ is
$G$-invariant and that $U'$ is $G$-invariant. Hence, $(X',\alpha')$ is a
descending section of $D$, 
i.e., an element of $\D^G$. 
Thus, there exists $(\bar X',\bar\alpha')\in\bar\D$ such that $X'\sim_\pi\bar
X'$ and $\alpha'=\pi^*\bar\alpha'$, and $(X'_\bp,\alpha'_\bp)\in\D_\bp$
such that $X'_\bp\sim_{\iota_\bp}\bar X'$ and
$\alpha'_\bp=\iota_\bp^*\bar\alpha'$.
At last, we compute
\begin{align*}
\pi_P^*\alpha'_\bp&=\pi_P^*\iota_\bp^*\bar\alpha'
=\iota_P^*\pi^*\bar\alpha'=\iota_P^*\alpha'\\
&=\iota_P^*\alpha=\iota_P^*\pi^*\bar\alpha
=\pi_P^*\iota_\bp^*\bar\alpha=\pi_P^*\alpha_\bp,
\end{align*}
which yields $\alpha'_\bp=\alpha_\bp$ on the open set $\pi(U'\cap P)\subseteq
U_\bp$ with $\bar p\in\pi(U'\cap P)$. In the same manner, we compute
\begin{align*}
(T\iota_\bp\circ X'_\bp)\circ \pi_P&=\bar X'\circ\iota_\bp\circ\pi_P=\bar
X'\circ\pi\circ\iota_P\\
&=(T\pi \circ X')\circ\iota_P=(T\pi \circ (X'+V_G))\circ\iota_P=(T\pi \circ 
X)\circ\iota_P\\
&=\bar X\circ\pi\circ\iota_P=\bar X\circ\iota_\bp\circ\pi_P
=(T\iota_\bp\circ X_\bp)\circ \pi_P.
\end{align*}
Thus, we have shown that $(X_\bp,\alpha_\bp)$ is an element of $\D_\bp$, that
is, $(X_\bp,\alpha_\bp)$ is a section of $D_\bp$ and 
$(X_\bp,\alpha_\bp)(\bar p)\in D_\bp(\bar p)$.
\end{proof}

Analogously to the regular case, we also have:
\begin{theorem}\label{closed}
Let $(M,D)$ be a Dirac manifold with a proper Dirac action of a connected Lie
group $G$ on it.
Let $\bar P$ be a stratum of the quotient space $\bar M$.
Assume that  the Dirac structure $D$ is integrable and that
$D\cap(\T\operp\V_G^\circ)$ is spanned by its descending sections. 
Then the Dirac structure
$D_{\bp}$ on $\bp$ introduced in Theorem \ref{singred} is integrable. 
\end{theorem}

\begin{proof}
Let $(X_\bp,\alpha_\bp)$ and $(Y_\bp,\beta_\bp)$ be sections of $D_\bp$. We
want to show that 
\[[(X_\bp,\alpha_\bp),(Y_\bp,\beta_\bp)]=([X_\bp,Y_\bp],\ldr{X_\bp}\beta_\bp
-\ip{Y_\bp}\dr
\alpha_\bp)\]
is also a section of $D_\bp$. From Remark \ref{remark_section_bar_p},
$(X_\bp,\alpha_\bp)$ 
and $(Y_\bp,\beta_\bp)$ are elements of $\mathcal{D}_\bp$ and thus we find
$(\bar X,\bar \alpha), 
(\bar Y,\bar
\beta) \in \bar \D$ such that 
$X_\bp\sim_{\iota_\bp}\bar X$, $Y_\bp\sim_{\iota_\bp}\bar Y$,
$\alpha_\bp=\iota_\bp^*\bar\alpha$, and
$\beta_\bp=\iota_\bp^*\bar\beta$. Furthermore, let $(X,\alpha)$ and
$(Y,\beta)$ be elements of $\D^G$, i.e., descending sections of $D$ such that 
$(X,\alpha)$ descends to $(\bar X,\bar\alpha)$ and $(Y,\beta)$ descends to
$(\bar Y,\bar\beta)$. By the proof of Theorem \ref{singred}, we can assume
that $(X,\alpha)$ and $(Y,\beta)$ are $G$-invariant.
The section $[(X,\alpha),(Y,\beta)]=([X,Y],\ldr{X}\beta-\ip{Y}\dr\alpha)$ is
then also a $G$-invariant section of $D$, since $D$ is integrable. We have 
\begin{align*}
(\ldr{X}\beta-\ip{Y}\dr\alpha)(\xi_M)=&\,\xi_M(\beta(X))+\dr\beta(X,\xi_M)
- \dr\alpha(Y,\xi_M)\\
=&\,0+
X(\beta(\xi_M))-\xi_M(\beta(X))-\beta([X,\xi_M])\\
&-Y(\alpha(\xi_M))+\xi_M(\alpha(Y))+\alpha([Y,\xi_M])\\
=&\,X(0)-0-\beta(0)- Y(0)+0+\alpha(0)
\end{align*}
for all $\xi\in\lie g$, where we have used that $\beta(X)$ and $\alpha(Y)\in
C^\infty(M)^G$. Hence $\ldr{X}\beta-\ip{Y}\dr\alpha\in\Gamma(\V^\circ)^G$,
$[(X,\alpha),(Y,\beta)]\in\D^G$  and
there exists
$(\bar Z,\bar \gamma)\in\bar \D$ such that $[X,Y]\sim_\pi \bar Z$  and
$\pi^*\bar\gamma=\ldr{X}\beta-\ip{Y}\dr\alpha$. Let $(Z_\bp,\gamma_\bp)$ be the
corresponding pair in $\D_\bp$. We want to show that $Z_\bp=[X_\bp,Y_\bp]$ and
$\ldr{X_\bp}\beta_\bp-\ip{Y_\bp}\dr
\alpha_\bp=\gamma_\bp$. 

Since we have $(X,\alpha)$,
$(Y,\beta)\in\Gamma(D\cap(\T\operp \V_G^\circ))$, there exist $\tilde X$ and
$\tilde Y$ in $\mx(P)$ such that $\tilde X\sim_{\iota_P} X$ and 
$\tilde Y\sim_{\iota_P} Y$. 
Set $\tilde \alpha=\iota_P^*\alpha$ and $\tilde\beta=\iota_P^*\beta$. 
We have $[\tilde X,\tilde Y]\sim_{\iota_P}[X,Y]$ and $[\tilde X,\tilde
Y]\in\mx(P)^G$. We have the equality $\iota_\bp\circ\pi_P=\pi\circ\iota_P$ and
consequently, since $\tilde X\sim_{\iota_P\circ\pi}\bar X$ and $\tilde
Y\sim_{\iota_P\circ\pi}\bar Y$, we have $\tilde X\sim_{\iota_\bp\circ\pi_P}\bar X$ 
and $\tilde
Y\sim_{\iota_\bp\circ\pi_P}\bar Y$. Hence, because  $X_\bp\sim_{\iota_\bp}\bar
X$, $Y_\bp\sim_{\iota_\bp}\bar Y$, and by Proposition \ref{restriction_of_X}, we get
$\tilde X\sim_{\pi_P}X_\bp$ and  $\tilde Y\sim_{\pi_P}Y_\bp$, and also
$[\tilde X,\tilde Y]\sim_{\pi_P}[X_\bp,Y_\bp]$. But in the same manner, we
have $[\tilde X,\tilde Y]\sim_{\iota_P\circ\pi}\bar Z$; thus $[\tilde X,\tilde
Y]\sim_{\iota_\bp\circ\pi_P}\bar Z$ and  $[X_\bp,Y_\bp]\sim_{\iota_\bp}\bar
Z$. Because of the uniqueness of $Z_\bp$ (Proposition \ref{restriction_of_X}), 
we get $Z_\bp=[X_\bp,Y_\bp]$. 

In the same manner, we have
\begin{align*}
\pi_P^*(\gamma_\bp)=\pi_P^*(\iota_\bp^*\bar\gamma)
=(\iota_\bp\circ\pi_P)^*(\bar\gamma)&
=(\pi\circ\iota_P)^*(\bar\gamma)=\iota_P^*(\ldr{X}\beta-\ip{Y}\dr\alpha)\\
&=(\ldr{\tilde
  X}\tilde\beta-\ip{ \tilde
  Y}\dr\tilde\alpha)\\
&=\pi_P^*(\ldr{X_\bp}\beta_\bp-\ip{Y_\bp}\dr\alpha_\bp).
\end{align*}
Thus, using the fact that $\pi_P$ is a smooth surjective submersion, we get
the equality of $\ldr{X_\bp}\beta_\bp-\ip{Y_\bp}\dr\alpha_\bp$ and $\gamma_\bp$.
\end{proof}

We end this subsection with examples. 
\begin{example}
Let $(M,\pois)$ be a smooth Poisson manifold with a canonical and proper action of a 
Lie group $G$ on it
(recall that the action of $G$ on $(M,\pois)$ is canonical
if $\{\Phi_g^*f_1,\Phi_g^*f_2\}=\Phi_g^*\{f_1,f_2\}$ for all $f_1,f_2\in
C^\infty(M)$ and $g\in G$).
Let $D_{\pois}$ be the Dirac structure associated to the Poisson 
structure; that is, $D_{\pois}(m)$ is defined by
\[D_{\pois}(m)=\{(X_f(m),\dr f(m))\mid f\in C^\infty(M)
\text{ and } X_f=\sharp(\dr f)\in\mx(M)\}\]
for all $m\in M$,
where $\sharp:T^*M\to TM$, $\dr f\mapsto X_f=\{\cdot, f\}$,  is the
homomorphism of vector bundles
 associated to $\pois$. Since the action of $G$ on $(M,\pois)$ is canonical, 
it is a Dirac action on $(M,D_{\pois})$. 

By Lemma \ref{description_V_G}, we know that $\V_G^\circ$ 
is generated by the exterior differentials of the 
$G$-invariant functions on $M$. Using the fact that the action of $G$ on $M$ is 
canonical, it is easy to check that the vector field $X_f$ associated to a 
$G$-invariant function $f$ is $G$-invariant. Hence, we get 
\[\left(D\cap(\T\operp\V_G^\circ)\right)(m)
=\erz\{(X_f(m),\dr f(m))\mid f\in C^\infty(M)^G\}.\]
This yields, automatically, that $D\cap(\T\operp\V_G^\circ)$ is spanned by 
its descending sections. Hence we can apply Theorems \ref{singred2} and \ref{singred} 
to the Dirac $G$-manifold $\left(M,D_{\pois}\right)$. Thus, each stratum of $\bar M$ inherits a 
Dirac structure $D_{\bar P}$ induced by $D_{\pois}$. Since $D_{\pois}$ is
integrable, the Dirac structure $D_{\bar P}$ is also integrable by Theorem
\ref{closed}.

We want to show that the codistribution $\mathsf{P}_1^\bp$ induced by $D_\bp$ 
(see Example \ref{ex_Dirac_dis}) on $\bp$ is equal to $T^*\bp$. To see this,  we show that 
$\dr f_\bp\in\Gamma\left(\mathsf{P}_1^\bp\right)$
for all $f_\bp\in C^\infty\left(\bp\right)$. Let $f_\bp\in C^\infty\left(\bp\right)$ and choose
$\bar f\in C^\infty\left(\bar M\right)$ with $\iota_\bp\left(\bar f\right)=f_\bp$. Set $f:=\pi^*\bar
f$. Then, as above, we have $\left(X_f,\dr f\right)\in\D^G$, and hence 
there exists $\bar{X}\in \mx\left(\bar M\right)$ such that $(X_f,\dr f)$
descends 
to $\left(\bar{X},\dr \bar{f}\right)\in\bar\D$. Since the restriction of $\dr
\bar f$ 
to $\bar P$ is equal to $\dr
f_\bp$, we get the existence of $X_\bp\in\mx\left(\bp\right)$ such that 
$\left(X_\bp,\dr f_\bp\right)$ in $\D_\bp$. Hence, $\dr f_\bp$ is a section of $\mathsf{P}_1^\bp$.

Since $D_\bp$ is integrable and $\mathsf{P}_1^\bp$ is constant dimensional and equal to  $T ^\ast \bar{P}$, the
Dirac structure $D_\bp$  defines a Poisson bracket $\pois_\bp$ on $C^\infty(\bar P)$ by
\[\poi{f_\bp,g_\bp}_\bp= - X_{f_\bp}(g_\bp)= X_{g_\bp}(f_\bp),\]
where $X_{f_\bp}$ and $X_{g_\bp}$ are such that $(X_{f_\bp},\dr f_\bp),
(X_{g_\bp},\dr g_\bp)\in \D_\bp=\Gamma(D_\bp)$. For a proof of this, see, for
example, \cite{Blankenstein00}. 

For $f_\bp, g_\bp\in C^\infty\left(\bp\right)$ choose, as above, extensions $\bar{f}, \bar{g}\in C^\infty\left(\bar M\right)$ and set $f=\pi^*\bar f$, $g=\pi^*\bar g$. Then we have
$\pi_P^*g_\bp=\iota_P^*g$ and there exists $\tilde X\in\mx(P)$ such that
$\tilde X\sim_{\iota_P}X_f$ and $\tilde X\sim_{\pi_P}X_{f_\bp}$. Thus, 
\begin{align*}
\iota_P^*\poi{f,g}&=-\iota_P^*(X_f(g))=-\tilde X(\iota_P^*g)=-\tilde
X(\pi_P^*g_\bp)\\
&=-\pi_P^*(X_{f_\bp}(g_\bp))
=\pi_P^*\poi{f_\bp,g_\bp}_\bp.
\end{align*}

This shows that, in the terminology of \cite{OrRa04}, $(M,\pois,P,G)$ is
always reducible
if $G$ acts properly and canonically on the Poisson manifold $M$ and $P$ is
a connected component of an orbit type manifold of the action.
\end{example}
\begin{example}
We consider the example of the proper action $\Phi$ of $G:=\mathds{S}^1$ on 
$M:=\R^3$ given by 
\[\alpha\cdot(x,y,z)=(x\cos\alpha-y\sin\alpha,x\sin\alpha+y\cos\alpha,z).\]
The orbit type manifolds of this action are $P_1=\{0\}\times\{0\}\times\R$,
that is,
$P_1=M_{H_1}$ with $H_1=\mathds{S}^1$, and $P_2=\R^3\setminus P_1$, so
$P_2=M_{H_2}$ with $H_2=\{1\}$. 
The orbit of a point $(x,y,z)\in\R^3$ is $\{(x',y',z')\mid
x'^2+y'^2=x^2+y^2\text{ and } z'=z\}$. Thus the reduced space $\bar M$
can be identified with $[0,+\infty)\times \R$ with the projection $\pi$ given by
\[\pi(x,y,z)=(x^2+y^2,z).\] 
It is easy to compute, for each $\alpha\in\mathds{S}^1$:
\begin{align*}
\Phi_\alpha^*(\partial_x)&=\cos\alpha\partial_x-\sin\alpha\partial_y,\\
\Phi_\alpha^*(\partial_y)&=\sin\alpha\partial_x+\cos\alpha\partial_y,\\
\Phi_\alpha^*(\partial_z)&=\partial_z
\end{align*}
and
\begin{align*}
\Phi_\alpha^*(\dr x)&=\cos\alpha\dr x-\sin\alpha\dr y,\\
\Phi_\alpha^*(\dr y)&=\sin\alpha\dr x+\cos\alpha\dr y,\\
\Phi_\alpha^*(\dr z)&=\dr z.
\end{align*}
Hence, the Dirac structure $D$ given as the span of the sections 
\[ (\partial_x,\dr y), (\partial_y,-\dr x), (\partial_z,0)\]
is $\mathds{S}^1$-invariant; that is, the Lie group $\mathds{S}^1$ acts on $(M,D)$ 
by Dirac actions.

The set $\D^{\mathds{S}^1}$ is spanned as a $C^\infty(M)$-module by the sections 
\[ (y\partial_x-x\partial_y,x\dr x+y\dr y) \quad \text{and} \quad (\partial_z,0).   \]
Note that the section $(x\partial_x+y\partial_y,y\dr x-x\dr y)$ is also
$\mathds{S}^1$-invariant but its cotangent part doesn't annihilate the vertical
space.
Also, since $\mathds{S}^1$ is Abelian, the vertical space is spanned by the
$\mathds{S}^1$-invariant vector field $y\partial_x-x\partial_y$ and  we only have to
consider $\mathds{S}^1$-invariant vector fields and not descending vector fields, in general.
Thus $\bar \D$ is the $C^\infty\left(\bar M\right)$-module generated by
the pairs $(\partial_{\bar z},0)$ and $(0,\bar x\dr \bar x)$, with the
coordinates $\bar x$ and $\bar z$ on $\bar M\simeq [0,\infty)\times\R$.

Then we have
$\mathcal{D}_{\bar P_1}=\erz_{C^\infty(\bar P_1)}\{(\partial_{\bar z},0)\}$
since $\bar x=0$ for all $p=(\bar x,\bar z)\in\bar P_1$. 
Hence, the Dirac structure $D_{\bar P_1}$ on 
$\bar{P}_1=P_1/G=P_1$
is  given as the span of the
section $(\partial_{\bar z},0)$. 

We have $\bar x\neq 0$ for all $p=(\bar x,\bar z)\in\bar P_2$.
Hence, since $$\mathcal{D}_{\bar P_2}=\erz_{C^\infty(\bar
  P_2)}\{(\partial_{\bar z},0), (0,\bar x\dr \bar x)\},$$
the Dirac structure
$D_{\bar P_2}$ on
$\bar P_2=P_2/G=(0,\infty)\times\R$ is given as the span of the sections 
$(\partial_{\bar z},0)$ and $(0,\dr \bar x)$.
\end{example}

\begin{example}
Consider here again Example \ref{bierstones_example}
with the Dirac structure given by
$D=(T\R^3\oplus\{0\})\operp(\{0\}\oplus T^*\R^3)$, i.e.,
$$D=\erz\left\{
(\partial_{x_1},0),(\partial_{y_1},0),(\partial_{z_1},0),
(0,\dr x_2), (0,\dr y_2),(0,\dr z_2)
\right\}.
$$
This Dirac bundle on $\R^3\times \R^3$ is obviously
invariant under the diagonal action of $\SO$, but the intersection
$D\cap(\T\operp\V_G^\circ)$ is not smooth. We have the 
descending sections
\[(x_1\partial_{ x_1}+y_1\partial_{ y_1}+z_1\partial_{ z_1},0),
(0,x_2\dr x_2+y_2\dr y_2+z_2\dr z_2)\]
of $D$. Let $(v,w)$ be a point where the function $f_3$ or one of the coordinates
$x_2,y_2,z_2$ 
vanishes. Then linear algebra arguments show that
there exists a linear combination
of $(x_1\partial_{ x_2}+y_1\partial_{ y_2}+z_1\partial_{ z_2},0)$,
$(x_2\partial_{x_1}+y_2\partial_{ y_1}
+z_2\partial_{ z_1},0)$,
$(y_1\partial_{x_1}-x_1\partial_{y_1}+y_2\partial_{x_2}-x_2\partial_{y_2},0)$,
$(z_1\partial_{y_1}-y_1\partial_{z_1}+z_2\partial_{y_2}-y_2\partial_{z_2},0)$,
and
$(x_1\partial_{z_1}-z_1\partial_{x_1}+x_2\partial_{z_2}-z_2\partial_{x_2},0)$
(sections of $\T\operp\{0\}$)
which is  an element of $((T\R^3\oplus\{0\})\operp\{0\})(v,w)$ and hence
of $D\cap(\T\operp\V_G^\circ)(v,w)$, but
there exists no open neighborhood of this point such that
this vector is the value at $(v,w)$ of a vector field defined on this whole 
neighborhood and having all its values in 
$D\cap(\T\operp\V_G^\circ)$.

Hence the reduction theorems of this paper (Theorems \ref{singred2}, \ref{singred} and 
\ref{closed}) do not apply to this example:
the action is canonical, but the hypothesis on smoothness of the intersection of $D$ with
$\T\operp\V_G^\circ$
is not satisfied.
\end{example}

\begin{example}
Let us illustrate Theorems \ref{singred} and \ref{closed}.
Consider again the manifold $M=\R^3\times\R^3$  this time with the 
(automatically proper) diagonal action of $G=\sfe$ on it, 
i.e., 
\[\begin{array}{cccc}
\Phi:&\sfe\times(\R^3\times\R^3)&\to&\R^3\times\R^3\\
&\left(\alpha,\begin{pmatrix}x_1\\
y_1\\z_1
\end{pmatrix},\begin{pmatrix}x_2\\
y_2\\z_2
\end{pmatrix}\right)&\mapsto&
\left(\begin{pmatrix}x_1\cos\alpha-y_1\sin\alpha\\
x_1\sin\alpha+y_1\cos\alpha\\
z_1\end{pmatrix},
\begin{pmatrix}x_2\cos\alpha-y_2\sin\alpha\\
x_2\sin\alpha+y_2\cos\alpha\\z_2
\end{pmatrix}\right)\!.
\end{array}
\]
The functions 
\begin{align*}
R_1(v,w)&=r_1^2(v,w)=x_1^2+y_1^2=\left\|\begin{pmatrix}
x_1\\
y_1
\end{pmatrix}\right\|^2,\\
R_2(v,w)&=r_2^2(v,w)=x_2^2+y_2^2=\left\|\begin{pmatrix}
x_2\\
y_2
\end{pmatrix}\right\|^2,\\
d(v,w)&=x_1y_2-y_1x_2=\det\begin{pmatrix}
x_1&x_2\\
y_1&y_2
\end{pmatrix},\\
s(v,w)&=x_1x_2+y_1y_2=\left\langle\begin{pmatrix}
x_1\\
y_1
\end{pmatrix},\begin{pmatrix}
x_2\\
y_2
\end{pmatrix}\right\rangle,\\
z_1(v,w)&=z_1,\quad z_2(v,w)=z_2\\
\end{align*}
are $\sfe$-invariant. They also characterize the
$\sfe$-orbits of the action since $d$ and $s$ determine
in a unique way the angle between the vectors $(x_1,y_1)$ and $(x_2,y_2)$. 
Hence, the reduced
manifold is the stratified space $\bar M=\pi(\R^3\times\R^3)\subseteq \R^6$,
where $\pi:\R^3\times\R^3\to\R^6$ is given by
$$\pi(v,w)=(R_1,R_2,d,s,z_1,z_2)(v,w).$$ We conclude that $\bar{M}$ is the
semi-algebraic set 
$$\bar{M}=\{(f_1,f_2,\delta,\sigma,z_1,z_2)\in\R^6\mid f_1,f_2\geq 0 
\text{ and } \sigma^2+\delta^2=f_1f_2\}.$$

The two strata of $\bar{M}$ are $\bar M_0=\{(0,0,0,0,z_1,z_2)\mid z_1,z_2\in
\R\}\subseteq \R^6$, corresponding to the orbit (isotropy) type manifold
$$M_{\sfe}=M_{(\sfe)}= \{(0,0,0,0,z_1,z_2)\mid z_1,z_2\in\R\}\subseteq \R^6$$
with trivial 
$\sfe $-action on it, and 
$\bar M_1=\{(f_1,f_2,\delta,\sigma,z_1,z_2)\in\R^6\mid (f_1,f_2)\neq(0,0)
\text{ and } \delta^2+\sigma^2=f_1f_2\}$, 
corresponding to the orbit (isotropy) type manifold
\begin{align*}
M_{\{0\}}&=M_{(\{0\})}\\
&= \{(x_1,y_1,z_1,x_2,y_2,z_2)\in \R^6\mid
(x_1,y_1)\neq (0,0) \text{ or } (x_2,y_2)\neq(0,0)\}.
\end{align*}

Let $U$ be the open set  $U: = \R_{>0}\times\R^4 \subset \R^5$.
Since the points $(f_1,f_2,\delta,\sigma,z_1,z_2)$ in $\bar M_1$ satisfy
$f_1>0$ or $f_2>0$, we have two charts for $\bar{M}_1$, namely 
$(\psi_1(U), \psi_1^{-1})$ and $(\psi_2(U), \psi_2^{-1})$, where
\[\begin{array}{lccc}
\psi_1:&\R_{>0}\times\R^4 &\to& \bar M_1\\
&(f_1,\delta,\sigma,z_1,z_2,)&\mapsto& 
\left(f_1,\frac{\delta^2+\sigma^2}{f_1},\delta,\sigma,z_1,z_2\right),\\
\psi_1^{-1}:&\psi_1(U)\subseteq \bar M_1&\to&  \R_{>0}\times\R^4\\
&(f_1,f_2,\delta,\sigma,z_1,z_2)&\mapsto& (f_1,\delta,\sigma,z_1,z_2)
\end{array}
\]
and 
\[\begin{array}{lccc}
\psi_2:& \R_{>0}\times\R^4&\to& \bar M_1\\
&(f_2,\delta,\sigma,z_1,z_2)&\mapsto& 
\left(\frac{\delta^2+\sigma^2}{f_2},f_2,\delta,\sigma,z_1,z_2\right),\\
\psi_2^{-1}:&\psi_2(U)\subseteq \bar M_1&\to&  \R_{>0}\times\R^4\\
&(f_1,f_2,\delta,\sigma,z_1,z_2)&\mapsto& (f_2,\delta,\sigma,z_1,z_2).
\end{array}
\]
We compute the distributions $\T=\T_G$ (note that $\V$ is spanned 
by $\sfe$-invariant sections because the Lie group is Abelian) and the 
codistribution $\V_G^\circ$. We have 
\[\V_G^\circ=\erz_{C^\infty(M)}
\left\{\begin{array}{c}
\dr z_1,\quad \dr z_2,\quad x_1\dr x_1+y_1\dr y_1,\\
x_2\dr x_2+y_2\dr y_2, \quad
x_1\dr y_2+y_2\dr x_1-x_2\dr y_1-y_1\dr x_2,\\
x_1\dr x_2+x_2\dr x_1+y_1\dr y_2+y_2\dr y_1 
\end{array}
\right\}
\]
and (see the appendix of \cite{JoRa10b})
\[\T=\erz_{C^\infty(M)}
\left\{\begin{array}{l}
X_1:=\partial_{ z_1},\quad X_2:=\partial_{ z_2},\\
 X_3:=x_1\partial_{x_1}+y_1\partial_{y_1}, \quad
X_4:=x_2\partial_{x_2}+y_2\partial_{y_2}, \\
X_5:=y_1\partial_{x_2}-x_1\partial_{y_2},\quad
X_6:=y_2\partial_{x_1}-x_2\partial_{y_1},\\
X_7:=x_1\partial_{x_2}+y_1\partial_{y_2},\quad
X_8:=x_2\partial_{x_1}+y_2\partial_{y_1},\\
X_9:=x_1\partial_{y_1}-y_1\partial_{x_1},\quad 
X_{10}:=x_2\partial_{y_2}-y_2\partial_{x_2} 
\end{array}
\right\}.
\]
Note that $\V$ is spanned on $M$ by
$X_9+X_{10}=x_1\partial_{y_1}-y_1\partial_{x_1}+x_2\partial_{y_2}-y_2\partial_{x_2}$.

We compute the flows associated to the spanning vector fields of $\T$ and find
(still using the coordinates $(v,w)=(x_1,y_1,z_1,x_2,y_2,z_2)$):
\begin{align*}
\phi^1_t(v,w)&=\left(\begin{pmatrix}x_1\\y_1\\z_1+t \end{pmatrix},
\begin{pmatrix}x_2\\y_2\\z_2 \end{pmatrix}\right),\quad
\phi^2_t(v,w)=\left(\begin{pmatrix}x_1\\y_1\\z_1 \end{pmatrix},
\begin{pmatrix}x_2\\y_2\\z_2+t \end{pmatrix}\right),\\
\phi^3_t(v,w)&=\left(\begin{pmatrix}e^tx_1\\e^ty_1\\z_1 \end{pmatrix},
\begin{pmatrix}x_2\\y_2\\z_2 \end{pmatrix}\right),\quad 
\phi^4_t(v,w)=\left(\begin{pmatrix}x_1\\y_1\\z_1 \end{pmatrix},
\begin{pmatrix}e^tx_2\\e^ty_2\\z_2 \end{pmatrix}\right),\\
\phi^5_t(v,w)&=\left(\begin{pmatrix}x_1\\y_1\\z_1 \end{pmatrix},
\begin{pmatrix}y_1t+x_2 \\-x_1t+y_2\\z_2 \end{pmatrix}\right),\,
\phi^6_t(v,w)=\left(\begin{pmatrix}y_2t+x_1\\-x_2t+y_1\\z_1 \end{pmatrix},
\begin{pmatrix}x_2\\y_2\\z_2 \end{pmatrix}\right),\\
\phi^7_t(v,w)&=\left(\begin{pmatrix}x_1\\y_1\\z_1 \end{pmatrix},
\begin{pmatrix}x_1t+x_2\\y_1t+y_2\\z_2 \end{pmatrix}\right),\,
\phi^8_t(v,w)=\left(\begin{pmatrix}x_2t+x_1\\y_2t+y_1\\z_1 \end{pmatrix},
\begin{pmatrix}x_2\\y_2\\z_2 \end{pmatrix}\right),\\
\phi^9_t(v,w)&=\left(\begin{pmatrix}x_1\cos t-y_1\sin t\\x_1\sin t +y_1\cos t\\z_1 
\end{pmatrix},
\begin{pmatrix}x_2\\y_2\\z_2 \end{pmatrix}\right),\\
\phi^{10}_t(v,w)&=\left(\begin{pmatrix}x_1\\y_1\\z_1 \end{pmatrix},
\begin{pmatrix}x_2\cos t-y_2\sin t\\x_2\sin t+y_2\cos
  t\\z_2 \end{pmatrix}\right),
\end{align*}
which are all easily verified to be $\sfe$-invariant.
It is easy to check that the two orbit (isotropy) type manifolds are the accessible
sets of the distribution $\T=\T_G$. 

We compute with this the flow $\bar \phi^i$ of the vector field $\bar X_i$ 
satisfying $X_i\sim_{\pi}\bar X_i$ for each  $i=1,\ldots,10$.
We have 
\begin{align*}
\bar\phi_t^1(f_1,f_2,\delta,\sigma,z_1,z_2)
&=\bar \phi^1_t(\pi(v,w))=\pi\circ\phi_t^1(v,w)=
(f_1,f_2,\delta,\sigma,z_1+t,z_2),\\
\bar\phi_t^2(f_1,f_2,\delta,\sigma,z_1,z_2)
&=\bar \phi^2_t(\pi(v,w))=\pi\circ\phi_t^2(v,w)=
(f_1,f_2,\delta,\sigma,z_1,z_2+t),\\
\bar \phi^3_t(f_1,f_2,\delta,\sigma,z_1,z_2)
&=(e^{2t}f_1,f_2,e^t\delta,e^t\sigma,z_1,z_2),\\
\bar \phi^4_t(f_1,f_2,\delta,\sigma,z_1,z_2)
&=(f_1,e^{2t}f_2,e^t\delta,e^t\sigma,z_1,z_2),\\
\bar\phi^5_t(f_1,f_2,\delta,\sigma,z_1,z_2)
&=(f_1,t^2f_1+f_2-2t\delta,-f_1t+\delta,\sigma,z_1,z_2),\\
\bar\phi^6_t(f_1,f_2,\delta,\sigma,z_1,z_2)
&=(f_1+t^2f_2+2t\delta,f_2,f_2t+\delta,\sigma,z_1,z_2),\\
\bar\phi^7_t(f_1,f_2,\delta,\sigma,z_1,z_2)
&=(f_1,t^2f_1+f_2+2t\sigma,\delta,f_1t+\sigma,z_1,z_2),\\
\bar\phi^8_t(f_1,f_2,\delta,\sigma,z_1,z_2)
&=(f_1+t^2f_2+2t\sigma,f_2,\delta,f_1t+\sigma,z_1,z_2),\\
\bar\phi^9_t(f_1,f_2,\delta,\sigma,z_1,z_2)
&=(f_1,f_2,\delta\cos t-\sigma\sin t,\delta\sin t+\sigma\cos t,z_1,z_2),\\
\bar\phi^{10}_t(f_1,f_2,\delta,\sigma,z_1,z_2)
&=(f_1,f_2,\sigma\sin t+\delta\cos t,\sigma\cos t-\delta\sin t,z_1,z_2).
\end{align*}
This leads to
\begin{align*}
\bar X_1(f_1,f_2,\delta,\sigma,z_1,z_2)&=\partial_{z_1},\qquad
\bar X_2(f_1,f_2,\delta,\sigma,z_1,z_2)=\partial_{z_2},\\
\bar X_3(f_1,f_2,\delta,\sigma,z_1,z_2)
&=2f_1\partial_{f_1}+\delta\partial_{\delta}+\sigma\partial_{\sigma},\\
\bar X_4(f_1,f_2,\delta,\sigma,z_1,z_2)
&=2f_2\partial_{f_2}+\delta\partial_{\delta}+\sigma\partial_{\sigma},\\
\bar X_5(f_1,f_2,\delta,\sigma,z_1,z_2)
&=-2\delta\partial_{f_2}-f_1\partial_{\delta},\\
\bar X_6(f_1,f_2,\delta,\sigma,z_1,z_2)
&=2\delta\partial_{f_1}+f_2\partial_{\delta},\\
\bar X_7(f_1,f_2,\delta,\sigma,z_1,z_2)
&=2\sigma\partial_{f_2}+f_1\partial_{\sigma},\\
\bar X_8(f_1,f_2,\delta,\sigma,z_1,z_2)
&=2\sigma\partial_{f_1}+f_2\partial_{\sigma},\\
\bar X_9(f_1,f_2,\delta,\sigma,z_1,z_2)
&=-\sigma\partial_{\delta}+\delta\partial_{\sigma},\\
\bar
X_{10}(f_1,f_2,\delta,\sigma,z_1,z_2)
&=\sigma\partial_{\delta}-\delta\partial_{\sigma}=-\bar X_9.
\end{align*}
Recalling that the tangent bundle to the manifold $\bar M_1$ is the kernel
of the one-form $\dr(f_1f_2-\delta^2-\sigma^2)=f_1\dr f_2+f_2\dr f_1
-2\sigma\dr \sigma-2\delta\dr \delta$,
we see that the two strata of $\bar M$ are indeed the accessible sets of the
distribution spanned by $\bar X_1,\ldots, \bar X_{10}$.

Consider the Dirac structure $D\subseteq TM\operp T^*M$ spanned by the pairs
\[(\partial_{x_1},\dr y_1), (\partial_{y_1},-\dr x_1), (\partial_{z_1},0),
(\partial_{x_2},-\dr y_2),(\partial_{y_2},\dr x_2),(0,\dr z_2).
\]
Comparing this with the sections of $\T$ and $\V_G^\circ$ given above, we find
\begin{align*}
\D^{\sfe}&=\erz_{C^\infty(M)^{\sfe}}\left\{
\begin{array}{c}
(\partial_{z_1},0), (0,\dr z_2),\\
(-x_1\partial_{y_1}+y_1\partial_{x_1},x_1\dr x_1+y_1\dr y_1),\\
(x_2\partial_{y_2}-y_2\partial_{x_2},x_2\dr x_2+y_2\dr y_2),\\
(-x_1\partial_{x_2}-y_2\partial_{y_1}-x_2\partial_{x_1}-y_1\partial_{y_2},\\
x_1\dr y_2+y_2\dr x_1-x_2\dr y_1-y_1\dr x_2),\\
(x_1\partial_{y_2}-x_2\partial_{y_1}-y_1\partial_{x_2}+y_2\partial_{x_1},\\
x_1\dr x_2+x_2\dr x_1+y_1\dr y_2+y_2\dr y_1)
\end{array}
\right\}\\
&=\erz_{C^\infty(M)^{\sfe}}\left\{
\begin{array}{c}
(\partial_{z_1},0), (0,\dr z_2),
\left(-X_9,\frac{1}{2}\dr R_1\right),\\
\left(X_{10},\frac{1}{2}\dr R_2\right),
(-X_7-X_8,\dr d),
(X_6-X_5,\dr s)
\end{array}
\right\}.
\end{align*}
Hence, we get 
\begin{align*}
\bar\D&=\erz_{C^\infty(\bar M)}\left\{
\begin{array}{c}
(\partial_{z_1},0), (0,\dr z_2),
\left(-\bar X_9,\frac{1}{2}\dr f_1\right),\\
\left(\bar X_{10},\frac{1}{2}\dr f_2\right),
(-\bar X_7-\bar X_8,\dr \delta),
(\bar X_6-\bar X_5,\dr \sigma)
\end{array}
\right\}\\
&=\erz_{C^\infty(\bar M)}\left\{
\begin{array}{c}
(\partial_{z_1},0), (0,\dr z_2),\\
\left(\sigma\partial_{\delta}-\delta\partial_{\sigma},\frac{1}{2}\dr f_1\right),
\left(\sigma\partial_{\delta}-\delta\partial_{\sigma},\frac{1}{2}\dr f_2\right),\\
(-2\sigma(\partial_{f_1}+\partial_{f_2})-(f_1+f_2)\partial_{\sigma},\dr \delta),\\
(2\delta(\partial_{f_1}+\partial_{f_2})+(f_1+f_2)\partial_{\delta},\dr \sigma)
\end{array}
\right\}.
\end{align*}
Recall the definition of one-forms on the stratified space 
$\bar{M}$ in Subsection \ref{pw_of_vf}. The ``one-forms'' $\dr f_1$ and $\dr f_2$ 
are \emph{not} derivatives of smooth
coordinates; they vanish at the points where $f_1$ and, respectively, $f_2$
vanish, by definition. 

Now we compute the induced Dirac structures on the two strata
$\bar M_0$ and $\bar M_1$.
The pairs $\left(-\bar X_9,\frac{1}{2}\dr f_1\right),
\left(\bar X_{10},\frac{1}{2}\dr f_2\right),
(-\bar X_7-\bar X_8,\dr \delta)$ and 
$(\bar X_6-\bar X_5,\dr \sigma)
$
are all zero on $\bar M_0$. So we get $\D_{\bar M_0}=\erz_{C^\infty(\bar M_0)}\{
(\partial_{z_1},0), (0,\dr z_2)\}$ and hence
$D_{\bar M_0}(\bar m)=\erz_\R\{(\partial_{z_1}\an{\bar m},0), (0,\dr z_2(\bar{m}))\}$
for all $\bar m\in\bar M_0$.

For the stratum $\bar M_1$, we give the Dirac structure
in the two charts $(\psi_1(U),\psi_1^{-1})$ and $(\psi_2(U),\psi_2^{-1})$.
We start with $(\psi_1(U),\psi_1^{-1})$, that is, the points of $\bar M_1$ where $f_1$ does not
vanish.
We have \[\psi_1^*\dr f_2
=\frac{1}{f_1}\left(-\left(\frac{\sigma^2+\delta^2}{f_1}\right)\dr f_1
+2\sigma\dr\sigma+2\delta\dr\delta
\right)
\quad \text{
and }
\quad \partial_{f_2}\sim_{\psi_1^{-1}}0.
\]
Hence, the pairs in
$\bar \D$ restrict to sections on $\psi_1(U)\subseteq \bar M_1$ that 
span the Dirac structure defined by
\begin{align*}
&D_{\bar M_1}(f_1,\sigma,\delta,z_1,z_2)\\
=&\erz_\R\left\{
\begin{array}{c}
(\partial_{z_1},0), (0,\dr z_2),
\left(2\sigma\partial_{\delta}-2\delta\partial_{\sigma},\dr f_1\right),\\
\left(0,\dr f_1
-\frac{1}{f_1}\left(-\left(\frac{\sigma^2+\delta^2}{f_1}\right)\dr f_1
+2\sigma\dr\sigma+2\delta\dr\delta
\right)\right),\\
\left(-2\sigma\partial_{f_1}
-\left(f_1+\frac{\sigma^2+\delta^2}{f_1}\right)\partial_{\sigma},
\dr \delta\right)),\\
\left(2\delta\partial_{f_1}
+\left(f_1+\frac{\sigma^2+\delta^2}{f_1}\right)\partial_{\delta},\dr \sigma\right))
\end{array}
\right\}(f_1,\sigma,\delta,z_1,z_2)\\
=&\erz_\R\left\{
\begin{array}{c}
(\partial_{z_1},0), (0,\dr z_2),
\left(2\sigma\partial_{\delta}-2\delta\partial_{\sigma},\dr f_1\right),\\
\left(0,
\left(f_1+\frac{\sigma^2+\delta^2}{f_1}\right)\dr f_1
-2\sigma\dr\sigma-2\delta\dr\delta
\right),\\
\left(-2\sigma\partial_{f_1}
-\left(f_1+\frac{\sigma^2+\delta^2}{f_1}\right)\partial_{\sigma},
\dr \delta\right),\\
\left(2\delta\partial_{f_1}
+\left(f_1+\frac{\sigma^2+\delta^2}{f_1}\right)\partial_{\delta},\dr \sigma\right)
\end{array}
\right\}(f_1,\sigma,\delta,z_1,z_2)
\end{align*}
for all $(f_1,\sigma,\delta,z_1,z_2)$ in $U$.
Since 
\begin{align}
&\left(0,
\left(f_1+\frac{\sigma^2+\delta^2}{f_1}\right)\dr f_1-2\sigma\dr\sigma
-2\delta\dr\delta
\right)\nonumber\\
=&\left(f_1+\frac{\sigma^2+\delta^2}{f_1}\right)
\left(2\sigma\partial_{\delta}-2\delta\partial_{\sigma},\dr f_1\right)\nonumber\\
&
-2\delta\left(-2\sigma\partial_{f_1}-\left(f_1
+\frac{\sigma^2+\delta^2}{f_1}\right)\partial_{\sigma},
\dr \delta\right)\nonumber\\
&-2\sigma\left(2\delta\partial_{f_1}
+\left(f_1+\frac{\sigma^2+\delta^2}{f_1}\right)\partial_{\delta},\dr
\sigma\right)\label{add_section},
\end{align}
this leads to
\begin{align}\label{DbarM1}
&D_{\bar{M}_1}(f_1,\sigma,\delta,z_1,z_2)\nonumber\\
=&\erz_\R\left\{
\begin{array}{c}
(\partial_{z_1},0), (0,\dr z_2),
\left(2\sigma\partial_{\delta}-2\delta\partial_{\sigma},\dr f_1\right),\\
\left(-2\sigma\partial_{f_1}
-\left(f_1+\frac{\sigma^2+\delta^2}{f_1}\right)\partial_{\sigma},
\dr \delta\right),\\
\left(2\delta\partial_{f_1}
+\left(f_1+\frac{\sigma^2+\delta^2}{f_1}\right)\partial_{\delta},\dr \sigma\right)
\end{array}
\right\}(f_1,\sigma,\delta,z_1,z_2).
\end{align}
In the same manner, we get in the chart $(\psi_2(U),\psi_2^{-1})$:
\begin{align*}
&D_{\bar{M}_1}(f_2,\sigma,\delta,z_1,z_2)\\
=&\erz_\R\left\{
\begin{array}{c}
(\partial_{z_1},0), (0,\dr z_2),
\left(2\sigma\partial_{\delta}-2\delta\partial_{\sigma},\dr f_2\right),\\
\left(-2\sigma\partial_{f_2}
-\left(f_2+\frac{\sigma^2+\delta^2}{f_2}\right)\partial_{\sigma},
\dr \delta\right),\\
\left(2\delta\partial_{f_2}
+\left(f_2+\frac{\sigma^2+\delta^2}{f_2}\right)\partial_{\delta},\dr \sigma\right)
\end{array}
\right\}(f_2,\sigma,\delta,z_1,z_2).
\end{align*}
It is easy to verify in both charts that this indeed defines  a Dirac bundle 
on $\bar M_1$; it is constant $\dim \bar M_1$-dimensional
and Lagrangian relative to the pairing on $T\bar M_1\operp T^*\bar M_1$.
\medskip

Since the Dirac structure $D$ on $M$ is integrable, we check that the reduced
Dirac 
manifolds $(\bar{M}_0,D_{\bar{M}_0})$ and $(\bar{M}_1,D_{\bar{M}_1})$ are also integrable.
For $(\bar M_0,D_{\bar M_0})$ this is obvious. For 
$(\bar{M}_1,D_{\bar{M}_1})$, we have to compute several Courant brackets.
Since the expressions for $D_{\bar{M}_1}$ are the same in both charts
$(\psi_1(U), \psi_1^{-1})$ 
and $(\psi_2(U), \psi_2^{-1})$, it suffices to carry out these computations only in the first chart.
Denote by $(X_i,\alpha_i)$, $i=1,\ldots,5$, the five spanning sections
of $D_{\bar M_1}$ in the chart $(\psi_1(U), \psi_1^{-1})$ in the order of formula
\eqref{DbarM1}. 
We only have to check that
$[(X_i,\alpha_i),(X_j,\alpha_j)]\in\Gamma(D_{\bar M_1})$ for
all $i,j\in\{1,\ldots,5\}$.

To see this, we begin by noting that
$$[(X_1,\alpha_1),(X_j,\alpha_j)]=[(X_2,\alpha_2),(X_j,\alpha_j)]=0\in\Gamma(D_{\bar{M}_1})$$
for 
$j=1,\ldots,5$. Next, we compute
\begin{align*}
&[(X_3,\alpha_3),(X_4,\alpha_4)]\\
=&\left[\left(2\sigma\partial_{\delta}-2\delta\partial_{\sigma},\dr f_1\right),
\left(-2\sigma\partial_{f_1}
-\left(f_1+\frac{\sigma^2+\delta^2}{f_1}\right)\partial_{\sigma},
\dr \delta\right)\right]\\
=&\Biggl(2\sigma\cdot\frac{-2\delta}{f_1}\partial_{\sigma}
-2\delta\cdot(-2)\partial_{f_1}
-2\delta\cdot\frac{-2\sigma}{f_1}\partial_{\sigma}
+\left(f_1+\frac{\sigma^2+\delta^2}{f_1}\right)\cdot 2\partial_{\delta},\\
&\qquad\qquad\qquad\qquad\qquad\qquad\qquad\qquad\dr\left( \dr\delta(2\sigma\partial_{\delta}-2\delta\partial_{\sigma})\right)
\Biggr)\\
=&\,2\left(
2\delta\partial_{f_1}
+\left(f_1+\frac{\sigma^2+\delta^2}{f_1}\right)\partial_{\delta},
\dr\sigma
\right)=2(X_5,\alpha_5)\in\Gamma(D_{\bar{M}_1}).
\end{align*}
In the same manner,
\begin{align*}
&[(X_3,\alpha_3),(X_5,\alpha_5)]\\
=&\left[\left(2\sigma\partial_{\delta}-2\delta\partial_{\sigma},\dr f_1\right),
\left(2\delta\partial_{f_1}
+\left(f_1+\frac{\sigma^2+\delta^2}{f_1}\right)\partial_{\delta},
\dr \sigma\right)\right]\\
&=-2\left(
-2\sigma\partial_{f_1}
-\left(f_1+\frac{\sigma^2+\delta^2}{f_1}\right)\partial_{\sigma},
\dr\delta
\right)=-2(X_4,\alpha_4)\in\Gamma(D_{\bar
  M_1})
\end{align*}
and 
\begin{align*}
&[(X_4,\alpha_4),(X_5,\alpha_5)]\\
=&\left[\left(
-2\sigma\partial_{f_1}
-\left(f_1+\frac{\sigma^2+\delta^2}{f_1}\right)\partial_{\sigma},
\dr\delta
\right),\left(2\delta\partial_{f_1}
+\left(f_1+\frac{\sigma^2+\delta^2}{f_1}\right)\partial_{\delta},
\dr \sigma\right)
\right]\\
=&\left(-2\sigma\cdot\left(1-\frac{\sigma^2+\delta^2}{f_1^2}\right)\partial_{\delta}
-\left(f_1+\frac{\sigma^2+\delta^2}{f_1}\right)\cdot\frac{2\sigma}{f_1}\partial_{\delta}\right.\\
&\left.\quad+2\delta\cdot\left(1-\frac{\sigma^2+\delta^2}{f_1^2}\right)\partial_{\sigma}
+\left(f_1+\frac{\sigma^2+\delta^2}{f_1}\right)\cdot\frac{2\delta}{f_1}\partial_{\sigma},
-\dr\left(f_1+\frac{\sigma^2+\delta^2}{f_1}\right)
\right)\\
=&\left(-4\sigma\partial_{\delta}+4\delta\partial_{\sigma},
-\dr f_1+\frac{\sigma^2+\delta^2}{f_1^2}\dr
  f_1-\frac{2\sigma\dr\sigma}{f_1}-\frac{2\delta\dr\delta}{f_1}\right)\\
=&-2(2\sigma\partial_{\delta}-2\delta\partial_{\sigma},
\dr f_1)+\frac{1}{f_1}\left(0,\left(f_1+\frac{\sigma^2+\delta^2}{f_1}\right)\dr f_1-2\sigma\dr\sigma
-2\delta\dr\delta
\right).
\end{align*}
This is a section of $D_{\bar M_1}$ since the first summand is the pair
$(X_3,\alpha_3)$ and the second summand is shown in \eqref{add_section}
to be a section of $D_{\bar M_1}$.

Thus, we have  directly verified in this example the conclusions of Theorems \ref{singred} and \ref{closed}.
\end{example}

\subsection{Reduction of dynamics}
\begin{definition}
The function $f\in C^\infty(M)$ will be called \emph{admissible} if there
exists a vector field $X_f\in \mx(M)$ such that 
\begin{equation}\label{admissible}
(X_f,\dr f)\in\Gamma(D).
\end{equation}
\end{definition}
Note that  we have $(X_f+Y,\dr f)\in\Gamma(D)$ for all
sections $Y$ of $\mathsf{G}_0$. Hence if the distribution
$\mathsf{G}_0$ is not trivial, equation \eqref{admissible} does not define a
unique vector field $X_f$.
\begin{theorem}
Let $f\in C^\infty(M)^G$ be admissible. Then there exists $X_f\in\mx(M)^G$
such that $(X_f,\dr f)$ is a section of $D$. Hence, for each
stratum $\bp$ of $\bar M$ satisfying the conditions of Theorem \ref{singred},
there exists a section $(X_\bp,\alpha_\bp)$ such that $X_f\sim_\pi \bar
X\in\mx(\bar M)$, $X_\bp\sim_{\iota_\bp}\bar X$, and
$\alpha_\bp=\iota_\bp^*\bar\alpha$, $\pi^*\bar\alpha=\dr f$. 
The vector field $X_\bp$ is a solution of the implicit Hamiltonian system
\[(X_\bp,\dr f_\bp)\in\Gamma(D_\bp),\]
where $f_\bp\in C^\infty(\bp)$ is the function defined by
$f_\bp=\iota_\bp^*\bar f$, with $\bar f\in C^\infty(\bar M)$ defined by
$\pi^*(\bar f\,)=f$.
If $X_\bp'$ is another solution of this equation, there exists an element
$Y$ of $\Gamma(\mathsf{G}_0)^G$ such that $X_f+Y$ descends to a vector field
on $\bar M$ that restricts to $X'_\bp$.
\end{theorem}  
\begin{proof}
Let $f\in C^\infty(M)^G$ be admissible. Let $X$ be a vector field satisfying
$(X,\dr f)\in\Gamma(D)$, and consider the average of this pair. Since $\dr
f$ is already $G$-invariant, it remains unchanged and the $G$-invariant
average of $(X,\dr f)$ is $(X_G,\dr f)$ with a $G$-invariant vector field
$X_G$. Since $D$ is $G$-invariant, the section $(X_G,\dr f)$ is also a
section of $D$. If $X_G$ disappears, the solutions of the implicit Hamiltonian system
span the generalized tangent distribution $\mathsf{G}_0$.  Set $X_G= :X_f$.
Then  the first statement is proved, and $(X_f,\dr f)\in\D^G$. The remainder of the
theorem follows immediately.
\end{proof}
\bibliographystyle{amsplain}

\def\cprime{$'$} \def\polhk#1{\setbox0=\hbox{#1}{\ooalign{\hidewidth
  \lower1.5ex\hbox{`}\hidewidth\crcr\unhbox0}}}
\providecommand{\bysame}{\leavevmode\hbox to3em{\hrulefill}\thinspace}
\providecommand{\MR}{\relax\ifhmode\unskip\space\fi MR }
\providecommand{\MRhref}[2]{%
  \href{http://www.ams.org/mathscinet-getitem?mr=#1}{#2}
}
\providecommand{\href}[2]{#2}

\end{document}